\documentclass[a4paper]{extarticle}
\usepackage[utf8]{inputenc}
\usepackage[english]{babel}
\usepackage{amsmath}
\usepackage{amssymb}
\usepackage{amsfonts}
\usepackage{amsthm}
\usepackage{bbm}  % need for \mathbbm{1}, \mathbb does not supports numbers
\usepackage{longtable}

\usepackage{tikz}  % for figures
\usepackage{tikz-cd}  % for commutative diagrams
\usetikzlibrary{
	knots,
	hobby
}
\usepackage{indentfirst}  % add indent to first paragraph in the section
\usepackage{textcomp}

\newtheoremstyle{blockstyle} % name
{} % Space above
{0.5em} % Space below
{} % Body font
{\parindent} % Indent amount
{\bfseries} % Theorem head font
{.} % Punctuation after theorem head
{.5em} % Space after theorem head
{} %

\theoremstyle{blockstyle}
\newtheorem{theorem}{Theorem}

\newtheorem{example}{Example}
\newtheorem{remark}{Remark}
\newtheorem{lemma}{Lemma}
\newtheorem{proposition}{Proposition}

\usepackage{enumitem}
\setlist{
label*={\arabic*.},
}
\setlist[1]{
wide=0pt, 
leftmargin=5pt, 
itemindent=\parindent, 
labelsep=5pt
}

\usepackage{titlesec}
\titlespacing*{\section}{\parindent}{10pt}{0pt}
\titlespacing*{\subsection}{\parindent}{10pt}{5pt}
\titlespacing*{\subsubsection}{\parindent}{10pt}{0pt}

\titleformat{\section}{\Large\bfseries}{\thesection.}{5pt}{}
\titleformat{\subsection}{\large\bfseries}{\thesubsection.}{5pt}{}
\titleformat{\subsubsection}{\normalsize\bfseries}{\thesubsubsection.}{5pt}{}

\renewenvironment{proof}{{\textit{Proof.}}}{\hfill $\square$}

\usepackage{tocloft}

\AtBeginDocument{}

\usepackage[symbol]{footmisc}

\usepackage[labelsep=period]{caption}

\usepackage[left=20mm,right=20mm,top=20mm,bottom=20mm,bindingoffset=0mm]{geometry}

\newcommand{\E}{\mathfrak{E}}
\newcommand{\U}{\mathbf{1}}
\newcommand{\A}{\mathbf{A}}

\definecolor{arrow_blue}{RGB}{79, 140, 209}
\definecolor{arrow_green}{RGB}{73, 163, 57}
\definecolor{arrow_red}{RGB}{181, 51, 51}
\definecolor{arrow_orange}{RGB}{189, 117, 28}

\tikzset{ma_blue/.style={draw=arrow_blue, -latex}}
\tikzset{ma_green/.style={draw=arrow_green, -latex}}
\tikzset{ma_red/.style={draw=arrow_red, -latex}}
\tikzset{ma_orange/.style={draw=arrow_orange, -latex}}
\tikzset{ma_black/.style={draw=black, -latex}}
\tikzset{
	knot diagram/every strand/.append style={
		thick,
		black
  }
}
\tikzset{knot_diagram/.style={draw=black, thick}}
\tikzset{split_line/.style={draw=gray, thin, dashed}}

\begin{document}

\begin{center}
\textbf{{\large Invariants for links and 3-manifolds from the modular category with two simple objects}}

Korablev Ph. G.\footnote{The research was supported by RSF (project No. 23-21-10014)}

\emph{korablev@csu.ru}
\end{center}

\begin{abstract}
We describe the simplest non-trivial modular category $\E$ with two simple objects. Then we extract from this category the invariant for non-oriented links in 3-sphere and two invariants for 3-manifolds: the complex-valued Turaev -- Reshetikhin type invariant $tr_{\varepsilon}$ and the real-valued Turaev -- Viro type invariant $tv_{\varepsilon}$. These two invariants for 3-manifolds are related by the equality $|tr_{\varepsilon}|^2\cdot (\varepsilon + 2) = tv_{\varepsilon}$, where $\varepsilon$ is a root of the equation $\varepsilon^2 = \varepsilon + 1$. Finally, we show that $tv_{\varepsilon}$ coincides with the well-known $\varepsilon$ invariant for 3-manifolds.
\end{abstract}

\tableofcontents

\section{Introduction}

This article is devoted to the implementation of an approach to the construction of quantum invariants for links in a three-dimensional sphere and for closed three-dimensional manifolds. This approach was originally proposed in \cite{TC}, its more detailed presentation is presented in the books \cite{T, TurVer} (see also \cite{BK}). The essence of the approach is that from each modular category $\mathcal{V}$ one can extract, using a special technique, two invariants $\tau_{\mathcal{V}}(M)$ and $|M|_{\mathcal{V}}$ of a closed oriented 3-manifold $M$. The values of these two invariants are related by the relation $|M|_{\mathcal{V}} = \tau_{\mathcal{V}}(M)\cdot \tau_{\mathcal{V}}(-M)$, where $-M$ is a manifold with opposite orientation.

Originally, the construction of the invariant $\tau_{\mathcal{V}}$ was given in the work of N. Reshetikhin and V. Turaev \cite{TR} for the case of the category of irreducible representations of the quantum group $U_q(sl_2)$. Therefore it is natural to call all invariants obtained from modular categories as Reshetikhin -- Turaev type invariants. Similarly, the construction of the invariant $|M|_{\mathcal{V}}$ was first proposed in the work of V. Turaev and O. Viro \cite{TV} using $q$-$6j$ symbols for the quantum group $U_q(sl_2)$. Therefore, all invariants $|M|_{\mathcal{V}}$ arising from modular categories are naturally called Turaev -- Viro type invariants.

One of the tasks is to study the properties of Reshetikhin -- Turaev and Turaev -- Viro type invariants for simple modular categories. A rather convenient class of categories for these purposes are fusion categories (see \cite{ENO, EC, ENOQT}). The paper \cite{RSW} gives a complete classification of fusion categories which are modular and whose rank (i.e. the number of isomorphism classes of simple objects) does not exceed four. These categories can be used to construct the simplest invariants of the Reshetikhin -- Turaev and Turaev -- Viro type.

The category $\E$ plays a central role in this article. This category is well known. For example, it is the basis of the anionic Fibonacci model of quantum computing (see \cite{FS, W}). This category appears quite often as one of the simplest non-trivial example of a fusion category with two simple objects (see \cite{ORT}). One of the features of the category $\E$ is that it contains non-trivial associators. Despite the "coherence theorem" (see \cite{MC}), which states that every monoidal category is equivalent to a strictly monoidal category (i.e. a category in which all associators are trivial), the category $\E$ turns out to be convenient for performing certain calculations. The Reshetikhin -- Turaev type invariant arising from the category $\E$ is denoted by $tr_{\varepsilon}$, and the Turaev -- Viro type invariant is denoted by $tv_{\varepsilon}$. The values of both invariants depend on the parameter $\varepsilon$, which can be any square root of the equation $\varepsilon^2 = \varepsilon + 1$.

The constant $\varepsilon$, which is used in the construction of the invariants $tr_{\varepsilon}$ and $tv_{\varepsilon}$, appears in the definition of the $\varepsilon$-invariant for 3-manifolds (see \cite{MOS}). This invariant can be defined in several equivalent ways. One of them is to define it as the homologically trivial part of the classical Turaev -- Viro invariant of order 5. The $\varepsilon$-invariant plays an important role in the complexity theory of 3-manifolds (see \cite{MCT}). For example, this invariant has been used to compute explicit complexity values for several infinite series of 3-manifolds (see \cite{VF, VTF}).

In this paper it was proved that the value of the $\varepsilon$-invariant for closed 3-manifolds coincides with the $tv_{\varepsilon}$ invariant (theorem \ref{Theorem:TVEqualT}). Since the $tr_{\varepsilon}$ invariant is stronger than the $tv_{\varepsilon}$ invariant (example \ref{Example:TRStronger}), it can be considered as a stronger version of the $\varepsilon$-invariant. Since the $\varepsilon$-invariant is part of one of the classical Turaev -- Viro invariants, it is quite natural that the $tr_{\varepsilon}$ invariant is also part of one of the classical Reshetikhin -- Turaev invariants (see \cite{KM}). The $tr_{\varepsilon}$ invariant is not new, but the approach used allows to give the definition in a purely combinatorial way without using the theory of representations of quantum groups. When describing the category $\E$ and the invariants $tr_{\varepsilon}$ and $tv_{\varepsilon}$, we mainly use the terminology and notations from the book \cite{T}.

The structure of the article is as follows. Section 2 gives a detailed description of the category $\E$. A similar description of this category, from a different point of view, is given in \cite{W}. Although the construction of the category $\E$ is well known, its description is given for several reasons. First, to fix the notation and to develop a diagrammatic approach to the representation of morphisms of the category $\E$. Second, to fix one of the equivalent views on the category $\E$. Third, such a description makes the text more or less self-contained.

In section 3 the invariant $tr_{\varepsilon}$ for unoriented links and 3-manifolds is defined and several examples are computed. In particular, an explicit formula is derived for the values of the invariant $tr_{\varepsilon}$ for generalised Hopf links and for lens spaces.

In section 4 it is proved that the Turaev -- Viro type invariant $tv_{\varepsilon}$ arising from the category $\E$ coincides with the $\varepsilon$-invariant for closed 3-manifolds.

\section{Category $\E$}

In this section we are going to construct the modular category $\E$. The construction is straightforward and purely combinatorial.

\subsection{Objects and morphisms}

\subsubsection{Objects}

Let $I = \{\U, \A\}$ be a set of two elements $\U$ and $\A$. These two elements play the role of simple objects in the category. Every object $X$ in $\E$ is a finite ordered tuple $(x_1, x_2, \ldots, x_n)$, where every $x_i\in I$, $i = 1, \ldots, n$. In most cases it will be convenient to understand $X$ as a noncommutative sum $X = x_1 + x_2 + \ldots + x_n$. Graphically each object can be drawn as a column, consisting of several cells, each cell containing a symbol either $\U$ or $\A$ (figure \ref{Figure:ObjectExample}). Columns are read from top to bottom.

\begin{figure}[h]
\begin{center}
\begin{tikzpicture}[scale=0.5]
	\draw (0, 0) rectangle ++(1, 1) node[pos=0.5] {$\U$};
	\draw (0, -1) rectangle ++(1, 1) node[pos=0.5] {$\A$};
	\draw (0, -2) rectangle ++(1, 1) node[pos=0.5] {$\U$};
	\draw (0, -3) rectangle ++(1, 1) node[pos=0.5] {$\A$};
\end{tikzpicture}
\hspace{1.0cm}
\begin{tikzpicture}[scale=0.5]
	\draw (0, 0) rectangle ++(1, 1) node[pos=0.5] {$\A$};
	\draw (0, -1) rectangle ++(1, 1) node[pos=0.5] {$\A$};
	\draw (0, -2) rectangle ++(1, 1) node[pos=0.5] {$\U$};
	\draw (0, -3) rectangle ++(1, 1) node[pos=0.5] {$\U$};
\end{tikzpicture}
\end{center}
\caption{\label{Figure:ObjectExample}Diagrams of objects $\U + \A + \U + \A$ (on the left) and $\A + \A + \U + \U$ (on the right)}
\end{figure}
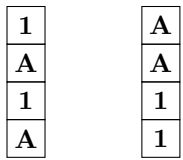

If $X = \sum\limits_{i = 1}^{n}x_i$, $x_i\in I$, $i = 1, \ldots, n$, then denote $|X|_{\U}$ the number of elements $\U$ in $X$ and $|X|_{\A}$ the number of elements $\A$. It's clear that $|X|_{\U} + |X|_{\A} = n$.

\subsubsection{Morphisms}

Let $X = \sum\limits_{i = 1}^{n} x_i$ and $Y = \sum\limits_{j = 1}^{m} y_j$, $x_i, y_j\in I$, are two objects. Every morphism $f\in Hom(X, Y)$ from $X$ to $Y$ is a map $(x_i, y_j) \mapsto a_{i}^{j}\in\mathbb{C}$ with the following properties: if $x_i \neq y_j$ then $a_{i}^{j} = 0$, if $x_i = y_j$ then the value $a_{i}^{j}$ can be any complex number.

Graphically, each morphism is drawn as arrows connecting each cell of the first object to each cell of the second. Each arrow is marked by a complex number. It's convenient not to draw arrows with zero values. Example of a morphism is shown in figure \ref{Figure:MorphismExample}.

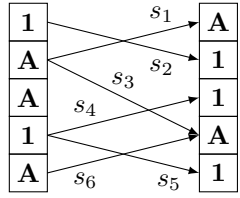
\begin{figure}[h]
\begin{center}
\begin{tikzpicture}[scale=0.5]
	\draw (0, 0) rectangle ++(1, 1) node[pos=0.5] {$\U$};
	\draw (0, -1) rectangle ++(1, 1) node[pos=0.5] {$\A$};
	\draw (0, -2) rectangle ++(1, 1) node[pos=0.5] {$\A$};
	\draw (0, -3) rectangle ++(1, 1) node[pos=0.5] {$\U$};
	\draw (0, -4) rectangle ++(1, 1) node[pos=0.5] {$\A$};
	
	\draw[-latex] (1.0, 0.5) -- (5.0, -0.5) node[below, near end] {$s_2$};
	\draw[-latex] (1.0, -0.5) -- (5.0, 0.5) node[above, near end] {$s_1$};
	\draw[-latex] (1.0, -0.5) -- (5.0, -2.5) node[above, midway] {$s_3$};
	\draw[-latex] (1.0, -2.5) -- (5.0, -1.5) node[above, near start] {$s_4$};
	\draw[-latex] (1.0, -2.5) -- (5.0, -3.5) node[below, pos=0.8] {$s_5$};
	\draw[-latex] (1.0, -3.5) -- (5.0, -2.5) node[below, near start] {$s_6$};
	
	\draw (5, 0) rectangle ++(1, 1) node[pos=0.5] {$\A$};
	\draw (5, -1) rectangle ++(1, 1) node[pos=0.5] {$\U$};
	\draw (5, -2) rectangle ++(1, 1) node[pos=0.5] {$\U$};
	\draw (5, -3) rectangle ++(1, 1) node[pos=0.5] {$\A$};
	\draw (5, -4) rectangle ++(1, 1) node[pos=0.5] {$\U$};
\end{tikzpicture}
\end{center}
\caption{\label{Figure:MorphismExample}Example of a morphism from the object $\U + \A + \A + \U + \A$ to the object $\A + \U + \U + \A + \U$}
\end{figure}

Let $f\in Hom(X, Y)$ be a morphism from object $X = \sum\limits_{i = 1}^{n}x_i$ to object $Y = \sum\limits_{j = 1}^{m}y_j$. Then we can construct two matrices: $[f]_{\U}$ with the size $|Y|_{\U} \times |X|_{\U}$ and $[f]_{\A}$ with the size $|Y|_{\A} \times |X|_{\A}$. The columns of the matrix $[f]_{\U}$ are bijective to the elements of $X$, equal to $\U$. The rows of the matrix $[f]_{\U}$ are bijective to the elements of $Y$, equal to $\U$. The element $a_i^j$ of the matrix $[f]_{\U}$ is equal to the value associated with the arrow connecting the $i$-th symbol $\U$ in $X$ to the $j$-th symbol $\U$ in $Y$. The matrix $[f]_{\A}$ is constructed in a similar way, but it contains only the values associated with the arrows connecting the symbols $\A$.

For example, for the morphism $f$ shown in the figure \ref{Figure:MorphismExample} we have
\begin{center}
$[f]_{\U} = \begin{pmatrix}
s_2 & 0 \\
0 & s_4 \\
0 & s_5
\end{pmatrix}$ and $[f]_{\A} = \begin{pmatrix}
s_1 & 0 & 0 \\
s_3 & 0 & s_6
\end{pmatrix}$.
\end{center}

\subsubsection{Morphisms composition}

Let $X = \sum\limits_{i = 1}^{n} x_i$, $Y = \sum\limits_{j = 1}^{m} y_j$ and $Z = \sum\limits_{k = 1}^{l} z_k$ be three objects. Let $f\in Hom (X, Y)$ be a morphism from $X$ to $Y$, $g\in Hom (Y, Z)$ be a morphism from $Y$ to $Z$. Then the composition $f\circ g$ is a morphism from $X$ to $Z$ defined by the following matrices:
\begin{center}
$[f\circ g]_{\U} = [g]_{\U} \cdot [f]_{\U}$ and $[f\circ g]_{\A} = [g]_{\A}\cdot [f]_{\A}$.
\end{center}

Note that we write the composition $f\circ g$ of the morphisms $f$ and $g$ from left to right. First we apply the left morphism $f$ and then the right morphism $g$.

Equivalently, we can define the composition $f\circ g$ on diagrams. To define the value of the arrow connecting the cell $x_i$ of the object $X$ with the cell $z_k$ of the object Z, we should find all oriented two-step paths from $x_i$ to $z_k$ in diagrams of morphisms $f$ and $g$. Then, for each path, compute the product of the values associated with the arrows of the path. Finally, get the sum of the computed products (figure \ref{Figure:ExampleComposition}).

\begin{figure}[h]
\begin{center}
\begin{tikzpicture}[scale=0.5, baseline=(current bounding box.center)]
	\draw (0, 0) rectangle ++(1, 1) node[pos=0.5] {$\U$};
	\draw (0, -1) rectangle ++(1, 1) node[pos=0.5] {$\A$};
	\draw (0, -2) rectangle ++(1, 1) node[pos=0.5] {$\U$};
	\draw (0, -3) rectangle ++(1, 1) node[pos=0.5] {$\A$};
	
	\draw[-latex] (1.0, -0.5) -- (5.0, 0.5) node[above, midway] {$f_1$};
 	\draw[-latex] (1.0, -0.5) -- (5.0, -1.5) node[above, near end] {$f_2$};
 	\draw[-latex] (1.0, -1.5) -- (5.0, -2.5) node[below, midway] {$f_3$};
	
	\draw (5, 0) rectangle ++(1, 1) node[pos=0.5] {$\A$};
	\draw (5, -1) rectangle ++(1, 1) node[pos=0.5] {$\U$};
	\draw (5, -2) rectangle ++(1, 1) node[pos=0.5] {$\A$};
	\draw (5, -3) rectangle ++(1, 1) node[pos=0.5] {$\U$};
\end{tikzpicture}
\hspace{0.1cm}$\circ$\hspace{0.1cm}
\begin{tikzpicture}[scale=0.5, baseline=(current bounding box.center)]
	\draw (0, 0) rectangle ++(1, 1) node[pos=0.5] {$\A$};
	\draw (0, -1) rectangle ++(1, 1) node[pos=0.5] {$\U$};
	\draw (0, -2) rectangle ++(1, 1) node[pos=0.5] {$\A$};
	\draw (0, -3) rectangle ++(1, 1) node[pos=0.5] {$\U$};
	
	\draw[-latex] (1.0, 0.5) -- (5.0, 0.5) node[below, midway] {$g_1$};
 	\draw[-latex] (1.0, -1.5) -- (5.0, 0.5) node[above, very near start] {$g_2$};
 	\draw[-latex] (1.0, -2.5) -- (5.0, -0.5) node[below, midway] {$g_3$};
	
	\draw (5, 0) rectangle ++(1, 1) node[pos=0.5] {$\A$};
	\draw (5, -1) rectangle ++(1, 1) node[pos=0.5] {$\U$};
\end{tikzpicture}
\hspace{0.1cm}$=$\hspace{0.1cm}
\begin{tikzpicture}[scale=0.5, baseline=(current bounding box.center)]
	\draw (0, 0) rectangle ++(1, 1) node[pos=0.5] {$\U$};
	\draw (0, -1) rectangle ++(1, 1) node[pos=0.5] {$\A$};
	\draw (0, -2) rectangle ++(1, 1) node[pos=0.5] {$\U$};
	\draw (0, -3) rectangle ++(1, 1) node[pos=0.5] {$\A$};
	
	\draw[-latex] (1.0, -0.5) -- (6.0, 0.5) node[above, midway, yshift=2.0] {$f_1 g_1 + f_2 g_2$};
 	\draw[-latex] (1.0, -1.5) -- (6.0, -0.5) node[below, midway] {$f_3\cdot g_3$};
	
	\draw (6, 0) rectangle ++(1, 1) node[pos=0.5] {$\A$};
	\draw (6, -1) rectangle ++(1, 1) node[pos=0.5] {$\U$};
\end{tikzpicture}
\end{center}
\caption{\label{Figure:ExampleComposition}Composition of two morphisms}
\end{figure}
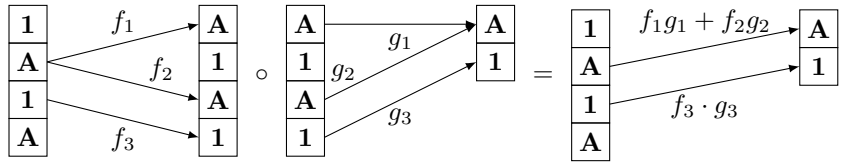

\subsubsection{Identity morphisms}

Let $X = \sum\limits_{i = 1}^{n} x_i$ be an object. Then the identity morphism $id_X$ is defined by two matrices: $[id_X]_{\U}$ is the identity matrix of size $|X|_{\U}$ and $[id_X]_{\A}$ is the identity matrix of size $|X|_{\A}$. The diagram of the identity morphism $id_X$ contains $n$ parallel arrows, with the number $1$ associated with each of them (figure \ref{Figure:ExampleIdentity}).

\begin{figure}[h]
\begin{center}
\begin{tikzpicture}[scale=0.5]
	\draw (0, 0) rectangle ++(1, 1) node[pos=0.5] {$\U$};
	\draw (0, -1) rectangle ++(1, 1) node[pos=0.5] {$\A$};
	\draw (0, -2) rectangle ++(1, 1) node[pos=0.5] {$\U$};
	\draw (0, -3) rectangle ++(1, 1) node[pos=0.5] {$\A$};
	
	\draw[-latex] (1.0, 0.5) -- (5.0, 0.5) node[above, midway] {$1$};
	\draw[-latex] (1.0, -0.5) -- (5.0, -0.5) node[above, midway] {$1$};
	\draw[-latex] (1.0, -1.5) -- (5.0, -1.5) node[above, midway] {$1$};
	\draw[-latex] (1.0, -2.5) -- (5.0, -2.5) node[above, midway] {$1$};
	
	\draw (5, 0) rectangle ++(1, 1) node[pos=0.5] {$\U$};
	\draw (5, -1) rectangle ++(1, 1) node[pos=0.5] {$\A$};
	\draw (5, -2) rectangle ++(1, 1) node[pos=0.5] {$\U$};
	\draw (5, -3) rectangle ++(1, 1) node[pos=0.5] {$\A$};
\end{tikzpicture}
\end{center}
\caption{\label{Figure:ExampleIdentity}Identity morphism $id_X$ for the object $X = \U + \A + \U + \A$}
\end{figure}
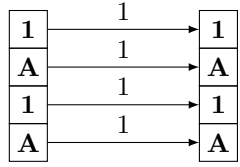

\begin{theorem}
\label{Threorm:Category}
$\E$ is the category.
\end{theorem}
\begin{proof}
First, we should check that the composition of morphisms is associative, i.e. for any four objects $X, Y, Z, W$ and three morphisms $f\in Hom(X, Y)$, $g\in Hom(Y, Z)$ and $h\in Hom(Z, W)$: $$(f\circ g)\circ h = f\circ (g\circ h).$$ Second, we should check that the composition with an identity morphism does not change the other morphism, i.e. for any two objects $X, Y$ and morphism $f\in Hom(X, Y)$: $$id_X\circ f = f\circ id_Y = f.$$

Both statements are obvious and follow from the associativity of matrix multiplication and multiplication with the identity matrix.
\end{proof}

\subsection{Moniodal structure}

Here we define a monoidal structure on the category $\E$. This makes this category monoidal. We will define tensor products of objects and morphisms, associativity isomorphisms, the unit object, and left and right unit isomorphisms. It turns out that the category $\E$ is not strict. It contains non-identity associativity isomorphisms.

\subsubsection{Tensor product of objects}

Let $X, Y\in I$ be simple objects. Define the tensor product $X\otimes Y$ as follows
\begin{center}
$\begin{array}{l}
\U\otimes \U = \U, \\
\U \otimes \A = \A, \\
\A\otimes \U = \A, \\
\A\otimes \A = \U + \A.
\end{array}$
\end{center}

Extends the defined tensor product of simple objects to all objects of the category $\E$ by linearity:
\begin{center}
$(X_1 + X_2)\otimes Y = X_1\otimes Y + X_2\otimes Y$ for $X_1, X_2, Y\in I$
\end{center}
and 
\begin{center}
$X\otimes (Y_1 + Y_2) = X\otimes Y_1 + X\otimes Y_2$ for $X, Y_1, Y_2\in I$.
\end{center}

Since objects of the category $\E$ are represented by non-commutative sums, we should define the order of the summands precisely. If the first and second elements in the product are not simple objects, we define this order as follows:
\begin{center}
$(X_1 + X_2)\otimes (Y_1 + Y_2) = X_1\otimes Y_1 + X_1\otimes Y_2 + X_2\otimes Y_1 + X_2\otimes Y_2$ for $X_1, X_2, Y_1, Y_2\in I.$
\end{center}

If objects $X$ or $Y$ are not simple, then the result $X\otimes Y$ is correctly defined by inductively applying the previous rules.

\begin{example}
\label{Example:ObjectsTensorProduct}
Compute $X\otimes (Y\otimes Z)$ for $X = \A + \U$, $Y = \A$ and $Z = \U + \A$:
\begin{multline*}
(\A + \U)\otimes (\A\otimes (\U + \A)) = (\A + \U)\otimes (\A\otimes \U + \A\otimes \A) = (\A + \U)\otimes (\A + \U + \A) = \\ = \A\otimes \A + \A\otimes \U + \A\otimes \A + \U\otimes \A + \U\otimes \U + \U\otimes \A = \U + \A + \A + \U + \A + \A + \U + \A.
\end{multline*}
\end{example}

\begin{remark}
In the general case $(X\otimes Y)\otimes Z \neq X\otimes (Y\otimes Z)$. But of course $|(X\otimes Y)\otimes Z|_{\U} = |X\otimes (Y\otimes Z)|_{\U}$ and $|(X\otimes Y)\otimes Z|_{\A} = |X\otimes (Y\otimes Z)|_{\A}$. In fact the objects $(X\otimes Y)\otimes Z$ and $X\otimes (Y\otimes Z)$ differ only in the order of the summands.
\end{remark}

\subsubsection{Tensor product of morphisms}

Let $X = \sum\limits_{i = 1}^{n}x_i$, $Y = \sum\limits_{j = 1}^{m}y_j$, $Z = \sum\limits_{r = 1}^{p}z_r$ and $T = \sum\limits_{s = 1}^{q} t_s$ be objects of the category $\E$, all $x_i, y_j, z_r, t_s\in I$ are simple objects. Let $f\in Hom(X, Y)$ be a morphism from $X$ to $Y$ and $g\in Hom(Z, T)$ be a morphism from $Z$ to $T$. Finally, let the morphism $f$ be defined by the values $f(x_i, y_j) = f_i^j\in \mathbb{C}$, $i\in\{1, \ldots, n\}, j\in\{1, \ldots, m\}$ and the morphism $g$ defined by the values $g(z_r, t_s) = g_r^s\in \mathbb{C}$, $r\in\{1, \ldots, p\}, s\in\{1, \ldots, q\}$.

Define the morphism $f\otimes g\in Hom (X\otimes Z, Y\otimes T)$ as follows. Consider the summands $x_i\otimes z_r$ in $X\otimes Z$ and $y_j\otimes t_s$ in $Y\otimes T$. Each of these is either one summand (if at least one of the objects is $\U$) or two summands (if both of the simple objects are $\A$). If $x_i\neq y_j$ or $z_r\neq t_s$, then define the value of the morphism $f\otimes g$, associated to all summands of $x_i\otimes z_r$ and $y_j\otimes t_s$, equal to zero. If $x_i = y_j$ and $z_r = t_s$, then $x_i\otimes z_r = y_j\otimes t_s$. If either $x_i$ or $z_r$ is $\U$, define the value of the morphism $f\otimes g$ associated with the pair $(x_i\otimes z_r, y_j\otimes t_s)$ to be $f_i^j\cdot g_r^s\in \mathbb{C}$. In the case where $x_i = z_r = \A$ (and $y_j = t_s = \A$), the result is $x_i\otimes z_r = \U + \A$ and $y_j\otimes t_s = \U + \A$. Define the values of the morphism $f\otimes g$ associated with the pairs $(\U, \U)$ and $(\A, \A)$ to be $f_i^j\cdot g_r^s\in\mathbb{C}$ and those associated with the other two pairs to be zero.

The tensor product of morphisms can be described on diagrams. We will draw the first morphism $f\in Hom(X, Y)$ at the top, and the second morphism $g\in Hom (Z, T)$ at the bottom (figure \ref{Figure:MorphismsProduct} on the left). Choose a non-zero arrow in the diagram for $f$ and a non-zero arrow in the digram for $g$. Let the values associated with these arrows be $f_i^j$ and $g_r^s$ respectively. Compute the result of the tensor product of these two arrows by the rule shown in the figure \ref{Figure:MorphismsProduct} in the middle and on the right. Do this for each pair of non-zero arrows in the diagrams for $f$ and $g$.

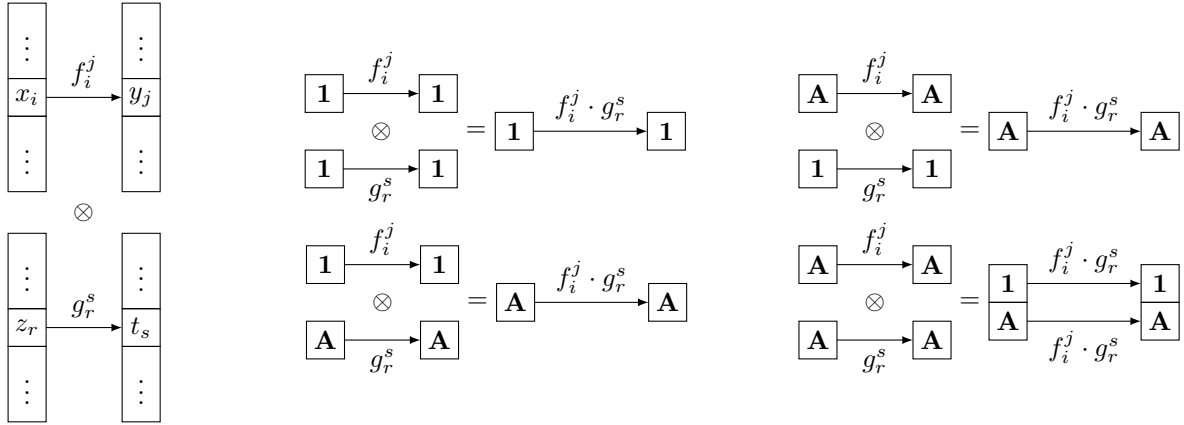
\begin{figure}[h]
\begin{center}
\ \hfill
\begin{minipage}{0.2\linewidth}
	\begin{tikzpicture}[scale=0.5]
		\draw (0, 0) rectangle ++(1, 2) node[pos=0.5] {$\vdots$};
		\draw (0, -1) rectangle ++(1, 1) node[pos=0.5] {$x_i$};
		\draw (0, -3) rectangle ++(1, 2) node[pos=0.5] {$\vdots$};
	
		\draw[-latex] (1.0, -0.5) -- (3.0, -0.5) node[above, midway] {$f_i^j$};
	
		\draw (3, 0) rectangle ++(1, 2) node[pos=0.5] {$\vdots$};
		\draw (3, -1) rectangle ++(1, 1) node[pos=0.5] {$y_j$};
		\draw (3, -3) rectangle ++(1, 2) node[pos=0.5] {$\vdots$};
	\end{tikzpicture}
	
	\begin{tikzpicture}[scale=0.5]
		\draw (0, 0) node {};
		\draw (3, 0) node {};
		\draw (1.8, 0) node {$\otimes$};
	\end{tikzpicture}
	
	\begin{tikzpicture}[scale=0.5]
		\draw (0, 0) rectangle ++(1, 2) node[pos=0.5] {$\vdots$};
		\draw (0, -1) rectangle ++(1, 1) node[pos=0.5] {$z_r$};
		\draw (0, -3) rectangle ++(1, 2) node[pos=0.5] {$\vdots$};
	
		\draw[-latex] (1.0, -0.5) -- (3.0, -0.5) node[above, midway] {$g_r^s$};
	
		\draw (3, 0) rectangle ++(1, 2) node[pos=0.5] {$\vdots$};
		\draw (3, -1) rectangle ++(1, 1) node[pos=0.5] {$t_s$};
		\draw (3, -3) rectangle ++(1, 2) node[pos=0.5] {$\vdots$};
	\end{tikzpicture}
\end{minipage}
\hfill
\begin{minipage}{0.35\linewidth}
	\begin{tikzpicture}[scale=0.5]
		\draw (0, 0) rectangle ++(1, 1) node[pos=0.5] {$\U$};
		\draw[-latex] (1.0, 0.5) -- (3.0, 0.5) node[above, midway] {$f_i^j$};
		\draw (3, 0) rectangle ++(1, 1) node[pos=0.5] {$\U$};
		
		\draw (2.0, -0.5) node {$\otimes$};
		
		\draw (0.0, -2.0) rectangle ++(1, 1) node[pos=0.5] {$\U$};
		\draw[-latex] (1.0, -1.5) -- (3.0, -1.5) node[below, midway] {$g_r^s$};
		\draw (3, -2.0) rectangle ++(1, 1) node[pos=0.5] {$\U$};
		
		\draw (4.5, -0.5) node {$=$};
		
		\draw (5, -1) rectangle ++(1, 1) node[pos=0.5] {$\U$};
		\draw[-latex] (6.0, -0.5) -- (9.0, -0.5) node[above, midway] {$f_i^j\cdot g_r^s$};
		\draw (9, -1) rectangle ++(1, 1) node[pos=0.5] {$\U$};
	\end{tikzpicture}
	
	\begin{tikzpicture}[scale=0.5]
		\draw (0, 0) rectangle ++(1, 1) node[pos=0.5] {$\U$};
		\draw[-latex] (1.0, 0.5) -- (3.0, 0.5) node[above, midway] {$f_i^j$};
		\draw (3, 0) rectangle ++(1, 1) node[pos=0.5] {$\U$};
		
		\draw (2.0, -0.5) node {$\otimes$};
		
		\draw (0.0, -2.0) rectangle ++(1, 1) node[pos=0.5] {$\A$};
		\draw[-latex] (1.0, -1.5) -- (3.0, -1.5) node[below, midway] {$g_r^s$};
		\draw (3, -2.0) rectangle ++(1, 1) node[pos=0.5] {$\A$};
		
		\draw (4.5, -0.5) node {$=$};
		
		\draw (5, -1) rectangle ++(1, 1) node[pos=0.5] {$\A$};
		\draw[-latex] (6.0, -0.5) -- (9.0, -0.5) node[above, midway] {$f_i^j\cdot g_r^s$};
		\draw (9, -1) rectangle ++(1, 1) node[pos=0.5] {$\A$};
	\end{tikzpicture}
\end{minipage}
\hfill
\begin{minipage}{0.35\linewidth}
	\begin{tikzpicture}[scale=0.5]
		\draw (0, 0) rectangle ++(1, 1) node[pos=0.5] {$\A$};
		\draw[-latex] (1.0, 0.5) -- (3.0, 0.5) node[above, midway] {$f_i^j$};
		\draw (3, 0) rectangle ++(1, 1) node[pos=0.5] {$\A$};
		
		\draw (2.0, -0.5) node {$\otimes$};
		
		\draw (0.0, -2.0) rectangle ++(1, 1) node[pos=0.5] {$\U$};
		\draw[-latex] (1.0, -1.5) -- (3.0, -1.5) node[below, midway] {$g_r^s$};
		\draw (3, -2.0) rectangle ++(1, 1) node[pos=0.5] {$\U$};
		
		\draw (4.5, -0.5) node {$=$};
		
		\draw (5, -1) rectangle ++(1, 1) node[pos=0.5] {$\A$};
		\draw[-latex] (6.0, -0.5) -- (9.0, -0.5) node[above, midway] {$f_i^j\cdot g_r^s$};
		\draw (9, -1) rectangle ++(1, 1) node[pos=0.5] {$\A$};
	\end{tikzpicture}
	
	\begin{tikzpicture}[scale=0.5]
		\draw (0, 0) rectangle ++(1, 1) node[pos=0.5] {$\A$};
		\draw[-latex] (1.0, 0.5) -- (3.0, 0.5) node[above, midway] {$f_i^j$};
		\draw (3, 0) rectangle ++(1, 1) node[pos=0.5] {$\A$};
		
		\draw (2.0, -0.5) node {$\otimes$};
		
		\draw (0.0, -2.0) rectangle ++(1, 1) node[pos=0.5] {$\A$};
		\draw[-latex] (1.0, -1.5) -- (3.0, -1.5) node[below, midway] {$g_r^s$};
		\draw (3, -2.0) rectangle ++(1, 1) node[pos=0.5] {$\A$};
		
		\draw (4.5, -0.5) node {$=$};
		
		\draw (5, -0.5) rectangle ++(1, 1) node[pos=0.5] {$\U$};
		\draw (5, -1.5) rectangle ++(1, 1) node[pos=0.5] {$\A$};
		\draw[-latex] (6.0, 0) -- (9.0, 0) node[above, midway] {$f_i^j\cdot g_r^s$};
		\draw[-latex] (6.0, -1) -- (9.0, -1) node[below, midway] {$f_i^j\cdot g_r^s$};
		\draw (9, -0.5) rectangle ++(1, 1) node[pos=0.5] {$\U$};
		\draw (9, -1.5) rectangle ++(1, 1) node[pos=0.5] {$\A$};
	\end{tikzpicture}
\end{minipage}
\hfill\ 
\end{center}
\caption{\label{Figure:MorphismsProduct}Tensor product of morphism diagrams (left), results of the tensor product of morphisms between simple objects (centre and right)}
\end{figure}

\begin{example}
\label{Example:MorphismsProduct}
Let the morphism $f\in Hom(\A + \U + \A, 1 + \A)$ be defined by matrices $$[f]_{\U} = \begin{pmatrix}
f_2
\end{pmatrix}, [f]_{\A} = \begin{pmatrix}
f_1 & 0
\end{pmatrix},$$
and let the morphism $g\in Hom(\U + \A, \U + \A)$ be defined by matrices $$[g]_{\U} = \begin{pmatrix}
g_1
\end{pmatrix}, [g]_{\A} = \begin{pmatrix}
g_2
\end{pmatrix}.
$$

Diagrams of these morphisms are shown in the figure \ref{Figure:ExampleMorphismProduct} on the left.

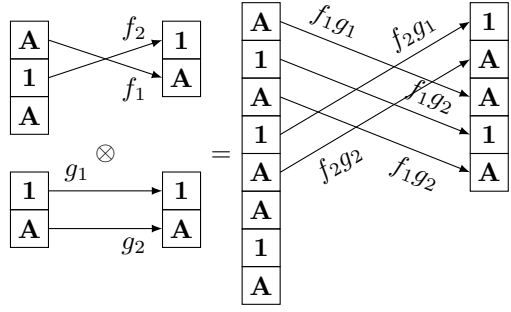
\begin{figure}[h]
\begin{center}
\begin{tikzpicture}[scale=0.5, baseline={([yshift=-2.2ex]current bounding box.center)}]
	\draw (0, -1) rectangle ++(1, 1) node[pos=0.5] {$\A$};
	\draw (0, -2) rectangle ++(1, 1) node[pos=0.5] {$\U$};
	\draw (0, -3) rectangle ++(1, 1) node[pos=0.5] {$\A$};
	
	\draw (4, -1) rectangle ++(1, 1) node[pos=0.5] {$\U$};
	\draw (4, -2) rectangle ++(1, 1) node[pos=0.5] {$\A$};
	
	\draw[-latex] (1.0, -0.5) -- (4, -1.5) node[below, near end] {$f_1$};
	\draw[-latex] (1.0, -1.5) -- (4, -0.5) node[above, near end] {$f_2$};
	\draw[-latex] (1.0, -4.5) -- (4, -4.5) node[above, near start] {$g_1$};
	\draw[-latex] (1.0, -5.5) -- (4, -5.5) node[below, near end] {$g_2$};
	
	\draw (2.5, -3.5) node {$\otimes$};
	
	\draw (0, -5) rectangle ++(1, 1) node[pos=0.5] {$\U$};
	\draw (0, -6) rectangle ++(1, 1) node[pos=0.5] {$\A$};
	
	\draw (4, -5) rectangle ++(1, 1) node[pos=0.5] {$\U$};
	\draw (4, -6) rectangle ++(1, 1) node[pos=0.5] {$\A$};
\end{tikzpicture}
$=$
\begin{tikzpicture}[scale=0.5, baseline={([yshift=-1.2ex]current bounding box.center)}]
	
	\draw (7, -1) rectangle ++(1, 1) node[pos=0.5] {$\A$};
	\draw (7, -2) rectangle ++(1, 1) node[pos=0.5] {$\U$};
	\draw (7, -3) rectangle ++(1, 1) node[pos=0.5] {$\A$};
	\draw (7, -4) rectangle ++(1, 1) node[pos=0.5] {$\U$};
	\draw (7, -5) rectangle ++(1, 1) node[pos=0.5] {$\A$};
	\draw (7, -6) rectangle ++(1, 1) node[pos=0.5] {$\A$};
	\draw (7, -7) rectangle ++(1, 1) node[pos=0.5] {$\U$};
	\draw (7, -8) rectangle ++(1, 1) node[pos=0.5] {$\A$};
	
	\draw (13, -1) rectangle ++(1, 1) node[pos=0.5] {$\U$};
	\draw (13, -2) rectangle ++(1, 1) node[pos=0.5] {$\A$};
	\draw (13, -3) rectangle ++(1, 1) node[pos=0.5] {$\A$};
	\draw (13, -4) rectangle ++(1, 1) node[pos=0.5] {$\U$};
	\draw (13, -5) rectangle ++(1, 1) node[pos=0.5] {$\A$};
	
	\draw[-latex] (8, -0.5) -- (13, -2.5) node[above, near start, sloped] {$f_1 g_1$};
	\draw[-latex] (8, -1.5) -- (13, -3.5) node[above, near end, sloped] {$f_1 g_2$};
	\draw[-latex] (8, -2.5) -- (13, -4.5) node[below, near end, sloped] {$f_1 g_2$};
	
	\draw[-latex] (8, -3.5) -- (13, -0.5) node[above, near end, sloped] {$f_2 g_1$};
	\draw[-latex] (8, -4.5) -- (13, -1.5) node[below, near start, sloped] {$f_2 g_2$};
\end{tikzpicture}
\end{center}
\caption{\label{Figure:ExampleMorphismProduct} Example of a tensor product of two morphisms}
\end{figure}

The result of the tensor product $f\otimes g$ is a morphism from $(\A + \U + \A)\otimes (\U + \A) = \A + \U + \A + \U + \A + \A + \U + \A$ to $(\U + \A)\otimes (\U + \A) = \U + \A + \A + \U + \A$, defined by matrices $$ [f\otimes g]_{\U} = \begin{pmatrix}
0 & f_2 g_1 & 0 \\
f_1 g_2 & 0 & 0
\end{pmatrix}, [f\otimes g]_{\A} = \begin{pmatrix}
0 & 0 & f_2 g_2 & 0 & 0 \\
f_1 g_1 & 0 & 0 & 0 & 0 \\
0 & f_1 g_2 & 0 & 0 & 0
\end{pmatrix}.
$$

Diagram of the morphism $f\otimes g$ shown in the figure \ref{Figure:ExampleMorphismProduct} on the right.
\end{example}

\begin{lemma}
\label{Lemma:MorphismTensorProduct}
Let $X, Y, Z, X', Y', Z'$ be objects of the category $\E$, and let $f\in Hom(X, Y)$, $g\in Hom(Y, Z)$, $f'\in Hom(X', Y')$, $g'\in Hom(Y', Z')$ be morphisms. Then $$(f\otimes f')\circ (g\otimes g') = (f\circ g)\otimes (f'\circ g').$$
\end{lemma}
\begin{proof}
It's enough to prove the lemma only for simple objects $X, Y, Z, X', Y', Z'$. The only non-trivial case is when $X=Y=Z=\A$ and also $X'=Y'=Z'=\A$. The diagram for the left side of the lemma statement is shown in the figure \ref{Figure:MorphismsProveLeft}.

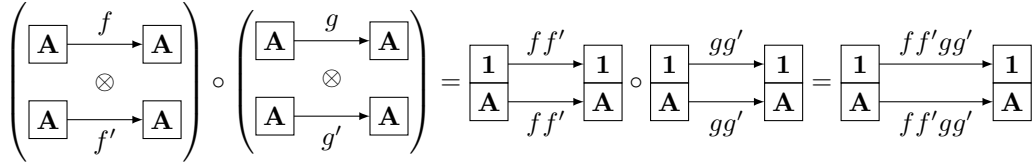
\begin{figure}[h]
$$\left(
\begin{tikzpicture}[scale=0.5, baseline={([yshift=-.5ex]current bounding box.center)}]
	\draw (0, -1) rectangle ++(1, 1) node[pos=0.5] {$\A$};
	\draw (3, -1) rectangle ++(1, 1) node[pos=0.5] {$\A$};
	\draw[-latex] (1, -0.5) -- (3, -0.5) node[above, midway] {$f$};
	
	\draw (2, -1.5) node {$\otimes$};
	
	\draw (0, -3) rectangle ++(1, 1) node[pos=0.5] {$\A$};
	\draw (3, -3) rectangle ++(1, 1) node[pos=0.5] {$\A$};
	\draw[-latex] (1, -2.5) -- (3, -2.5) node[below, midway] {$f'$};
\end{tikzpicture}
\right) 
\circ
\left(
\begin{tikzpicture}[scale=0.5, baseline={([yshift=-.5ex]current bounding box.center)}]
	\draw (0, -1) rectangle ++(1, 1) node[pos=0.5] {$\A$};
	\draw (3, -1) rectangle ++(1, 1) node[pos=0.5] {$\A$};
	\draw[-latex] (1, -0.5) -- (3, -0.5) node[above, midway] {$g$};
	
	\draw (2, -1.5) node {$\otimes$};
	
	\draw (0, -3) rectangle ++(1, 1) node[pos=0.5] {$\A$};
	\draw (3, -3) rectangle ++(1, 1) node[pos=0.5] {$\A$};
	\draw[-latex] (1, -2.5) -- (3, -2.5) node[below, midway] {$g'$};
\end{tikzpicture}
\right) 
= 
\begin{tikzpicture}[scale=0.5, baseline={([yshift=-.5ex]current bounding box.center)}]
	\draw (0, -1) rectangle ++(1, 1) node[pos=0.5] {$\U$};
	\draw (0, -2) rectangle ++(1, 1) node[pos=0.5] {$\A$};
	\draw (3, -1) rectangle ++(1, 1) node[pos=0.5] {$\U$};
	\draw (3, -2) rectangle ++(1, 1) node[pos=0.5] {$\A$};
	\draw[-latex] (1, -0.5) -- (3, -0.5) node[above, midway] {$f f'$};
	\draw[-latex] (1, -1.5) -- (3, -1.5) node[below, midway] {$f f'$};
\end{tikzpicture}
\circ
\begin{tikzpicture}[scale=0.5, baseline={([yshift=-.5ex]current bounding box.center)}]
	\draw (0, -1) rectangle ++(1, 1) node[pos=0.5] {$\U$};
	\draw (0, -2) rectangle ++(1, 1) node[pos=0.5] {$\A$};
	\draw (3, -1) rectangle ++(1, 1) node[pos=0.5] {$\U$};
	\draw (3, -2) rectangle ++(1, 1) node[pos=0.5] {$\A$};
	\draw[-latex] (1, -0.5) -- (3, -0.5) node[above, midway] {$g g'$};
	\draw[-latex] (1, -1.5) -- (3, -1.5) node[below, midway] {$g g'$};
\end{tikzpicture}
=
\begin{tikzpicture}[scale=0.5, baseline={([yshift=-.5ex]current bounding box.center)}]
	\draw (0, -1) rectangle ++(1, 1) node[pos=0.5] {$\U$};
	\draw (0, -2) rectangle ++(1, 1) node[pos=0.5] {$\A$};
	\draw (4, -1) rectangle ++(1, 1) node[pos=0.5] {$\U$};
	\draw (4, -2) rectangle ++(1, 1) node[pos=0.5] {$\A$};
	\draw[-latex] (1, -0.5) -- (4, -0.5) node[above, midway] {$f f' g g'$};
	\draw[-latex] (1, -1.5) -- (4, -1.5) node[below, midway] {$f f' g g'$};
\end{tikzpicture}
$$
\caption{\label{Figure:MorphismsProveLeft}Diagram of the left hand side of the statement of the lemma \ref{Lemma:MorphismTensorProduct}}
\end{figure}

The digram for the right hand side of the lemma statement is shown in the figure \ref{Figure:MorphismsProveRight}.

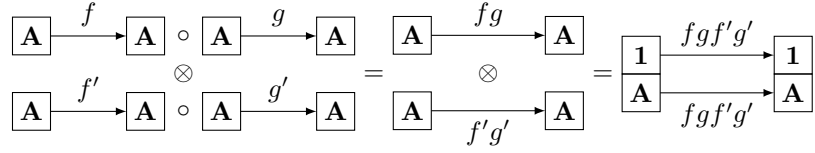
\begin{figure}[h]
$$
\begin{tikzpicture}[scale=0.5, baseline={([yshift=-1.5ex]current bounding box.center)}]
	\draw (0, -1) rectangle ++(1, 1) node[pos=0.5] {$\A$};
	\draw (3, -1) rectangle ++(1, 1) node[pos=0.5] {$\A$};
	\draw (5, -1) rectangle ++(1, 1) node[pos=0.5] {$\A$};
	\draw (8, -1) rectangle ++(1, 1) node[pos=0.5] {$\A$};
	
	\draw[-latex] (1, -0.5) -- (3, -0.5) node[above, midway] {$f$};
	\draw[-latex] (6, -0.5) -- (8, -0.5) node[above, midway] {$g$};
	
	\draw (4.5, -0.5) node {$\circ$};
	
	\draw (4.5, -1.5) node {$\otimes$};
	
	\draw (0, -3) rectangle ++(1, 1) node[pos=0.5] {$\A$};
	\draw (3, -3) rectangle ++(1, 1) node[pos=0.5] {$\A$};
	\draw (5, -3) rectangle ++(1, 1) node[pos=0.5] {$\A$};
	\draw (8, -3) rectangle ++(1, 1) node[pos=0.5] {$\A$};
	
	\draw[-latex] (1, -2.5) -- (3, -2.5) node[above, midway] {$f'$};
	\draw[-latex] (6, -2.5) -- (8, -2.5) node[above, midway] {$g'$};
	
	\draw (4.5, -2.5) node {$\circ$};
\end{tikzpicture}
=
\begin{tikzpicture}[scale=0.5, baseline={([yshift=-.5ex]current bounding box.center)}]
	\draw (0, -1) rectangle ++(1, 1) node[pos=0.5] {$\A$};
	\draw (4, -1) rectangle ++(1, 1) node[pos=0.5] {$\A$};
	
	\draw[-latex] (1, -0.5) -- (4, -0.5) node[above, midway] {$fg$};
	
	\draw (2.5, -1.5) node {$\otimes$};
	
	\draw (0, -3) rectangle ++(1, 1) node[pos=0.5] {$\A$};
	\draw (4, -3) rectangle ++(1, 1) node[pos=0.5] {$\A$};
	
	\draw[-latex] (1, -2.5) -- (4, -2.5) node[below, midway] {$f'g'$};
\end{tikzpicture}
=
\begin{tikzpicture}[scale=0.5, baseline={([yshift=-.5ex]current bounding box.center)}]
	\draw (0, -1) rectangle ++(1, 1) node[pos=0.5] {$\U$};
	\draw (0, -2) rectangle ++(1, 1) node[pos=0.5] {$\A$};
	\draw (4, -1) rectangle ++(1, 1) node[pos=0.5] {$\U$};
	\draw (4, -2) rectangle ++(1, 1) node[pos=0.5] {$\A$};
	\draw[-latex] (1, -0.5) -- (4, -0.5) node[above, midway] {$f g f' g'$};
	\draw[-latex] (1, -1.5) -- (4, -1.5) node[below, midway] {$f g f' g'$};
\end{tikzpicture}
$$
\caption{\label{Figure:MorphismsProveRight}Diagram of the right hand side of the statement of the lemma \ref{Lemma:MorphismTensorProduct}}
\end{figure}

Both morphisms are the same.
\end{proof}

\begin{lemma}
\label{Lemma:IdentityMorphismsProduct}
Let $X, Y$ be objects of the category $\E$. Then $id_{X\otimes Y} = id_X\otimes id_Y$.
\end{lemma}
\begin{proof}
The statement of the lemma is obvious.
\end{proof}

\subsubsection{Associativity isomorphisms}

We should define the family of isomorphisms $\alpha_{X, Y, Z}\in Hom((X\otimes Y)\otimes Z, X\otimes (Y\otimes Z))$ for each triple of objects $X, Y, Z$ in the category $\E$. Consider the case where all objects $X, Y, Z$ are simple objects. If one of them is equal to $\U$, then define $\alpha_{X, Y, Z} = id_{(X\otimes Y)\otimes Z}$. If $X = Y = Z = \A$, then $\alpha_{\A, \A, \A}\in Hom(\A + \U + \A, \A + \U + \A)$ and define $\alpha_{\A, \A, \A}$ by two matrices $$
[\alpha_{\A, \A, \A}]_{\U} = \begin{pmatrix}
1
\end{pmatrix},
[\alpha_{\A, \A, \A}]_{\A} = \begin{pmatrix}
\frac{1}{\varepsilon} & \frac{x}{\sqrt{\varepsilon}} \\
\frac{1}{x\sqrt{\varepsilon}} & -\frac{1}{\varepsilon}
\end{pmatrix},
$$
where $\varepsilon$ is any root of the equation $\varepsilon^2 = \varepsilon + 1$, and $x\in\mathbb{C}$ is any non-zero complex number.

\begin{remark}
\label{Remark:InverseAssociator}
The morphism $\alpha_{\A, \A, \A}$ is an isomorphism. The inverse morphism $\alpha_{\A, \A, \A}^{-1}\in Hom(\A\otimes (\A\otimes \A), (\A\otimes \A)\otimes \A)$ defined by the same matrices $$
[\alpha_{\A, \A, \A}^{-1}]_{\U} = \begin{pmatrix}
1
\end{pmatrix}, [\alpha_{\A, \A, \A}^{-1}]_{\A} = \begin{pmatrix}
\frac{1}{\varepsilon} & \frac{x}{\sqrt{\varepsilon}} \\
\frac{1}{x\sqrt{\varepsilon}} & -\frac{1}{\varepsilon}
\end{pmatrix}.
$$

It follows from the fact that $$
\begin{pmatrix}
\frac{1}{\varepsilon} & \frac{x}{\sqrt{\varepsilon}} \\
\frac{1}{x\sqrt{\varepsilon}} & -\frac{1}{\varepsilon}
\end{pmatrix} \cdot \begin{pmatrix}
\frac{1}{\varepsilon} & \frac{x}{\sqrt{\varepsilon}} \\
\frac{1}{x\sqrt{\varepsilon}} & -\frac{1}{\varepsilon}
\end{pmatrix} = \begin{pmatrix}
\frac{1}{\varepsilon^2} + \frac{1}{\varepsilon} & 0 \\
0 & \frac{1}{\varepsilon^2} + \frac{1}{\varepsilon}
\end{pmatrix} = \begin{pmatrix}
1 & 0 \\
0 & 1
\end{pmatrix}.
$$
\end{remark}

Extend defined associativity isomorphisms to all objects of the category $\E$ by linearity. This means the following. Let $X_1, X_2, Y, Z$ be simple objects. Let the object $((X_1 + X_2)\otimes Y)\otimes Z$ be represented as a non-commutative sum of the simple objects $\U, \A$. Some summands in this sum correspond to the result $(X_1\otimes Y)\otimes Z$, and other summands correspond to the result $(X_2\otimes Y)\otimes Z$. Define the isomorphism $\alpha_{X_1 + X_2, Y, Z}$ as follows: it is equal to the corresponding value of $\alpha_{X_1, Y, Z}$ for summands corresponding to $(X_1\otimes Y)\otimes Z$, and it is equal to the corresponding value of $\alpha_{X_2, Y, Z}$ for all other summands. Similarly, define the associativity isomorphisms $\alpha_{X, Y_1 + Y_2, Z}$ for $X, Y_1, Y_2, Z\in I$ and $\alpha_{X, Y, Z_1 + Z_2}$ for $X, Y, Z_1, Z_2\in I$. The associativity isomorphism $\alpha_{X, Y, Z}$ for any three objects $X, Y, Z$ defined by recursively applying these rules.

\begin{example}
\label{Example:Associator}
Isomorphism $\alpha_{\U + \A, \A, \A}\in Hom(\U + \A + \A + \U + \A, \U + \A + \A + \U + \A)$ defined by matrices $$
[\alpha_{\U + \A, \A, \A}]_{\U} = \begin{pmatrix}
1 & 0 \\
0 & 1
\end{pmatrix}, [\alpha_{\U + \A, \A, \A}]_{\A} = \begin{pmatrix}
1 & 0 & 0 \\
0 & \frac{1}{\varepsilon} & \frac{x}{\sqrt{\varepsilon}} \\
0 & \frac{1}{x\sqrt{\varepsilon}} & -\frac{1}{\varepsilon}
\end{pmatrix}.
$$

Isomorphism $\alpha_{\A, \U + \A, \A}\in Hom(\U + \A + \A + \U + \A, \U + \A + \A + \U + \A)$ defined by the same matrices $$
[\alpha_{\A, \U + \A, \A}]_{\U} = \begin{pmatrix}
1 & 0 \\
0 & 1
\end{pmatrix}, [\alpha_{\A, \U + \A, \A}]_{\A} = \begin{pmatrix}
1 & 0 & 0 \\
0 & \frac{1}{\varepsilon} & \frac{x}{\sqrt{\varepsilon}} \\
0 & \frac{1}{x\sqrt{\varepsilon}} & -\frac{1}{\varepsilon}
\end{pmatrix}.
$$

Isomorphism $\alpha_{\A, \A, \U + \A}\in Hom(\U + \A + \A + \U + \A, \U + \A + \A + \U + \A)$ defined by matrices $$
[\alpha_{\A, \A, \U + \A}]_{\U} = \begin{pmatrix}
1 & 0 \\
0 & 1
\end{pmatrix}, [\alpha_{\A, \A, \U + \A}]_{\A} = \begin{pmatrix}
0 & 1 & 0 \\
\frac{1}{\varepsilon} & 0 & \frac{x}{\sqrt{\varepsilon}} \\
\frac{1}{x\sqrt{\varepsilon}} & 0 & -\frac{1}{\varepsilon}
\end{pmatrix}.
$$

Let's explain the matrix $[\alpha_{\A, \A, \U + \A}]_{\A}$ in more detail. The isomorphism $\alpha_{\A, \A, \U + \A}$ is a morphism from the object $(\A\otimes \A)\otimes (\U + \A)$ to the object $\A\otimes (\A\otimes (\U + \A))$. Present the first object as a sum of simple objects, and use single underscores to mark summands, corresponding to $\U$, and double underscores for summands, corresponding to $\A$: $$
(\A\otimes \A)\otimes (\underline{\U} + \underline{\underline{\A}}) = (\U + \A)\otimes (\underline{\U} + \underline{\underline{\A}}) = \underline{\U} + \underline{\underline{\A}} + \underline{\A} + \underline{\underline{\U}} + \underline{\underline{\A}}.
$$

Do the same for the second object $\A\otimes (\A\otimes (\U + \A))$:
$$
\A\otimes (\A\otimes (\underline{\U} + \underline{\underline{\A}})) = \A\otimes (\underline{\A} + \underline{\underline{\U}} + \underline{\underline{\A}}) = \underline{\U} + \underline{\A} + \underline{\underline{\A}} + \underline{\underline{\U}} + \underline{\underline{\A}}.
$$

So the matrix $[\alpha_{\A, \A\, \U + \A}]_{\A}$ contains elements from two matrices. From the matrix $[\alpha_{\A, \A, \U}]_{\A}$, which is an identical matrix, the number $1$ corresponds to the second symbol $\A$ of the first object and to the first symbol $\A$ of the second object. Therefore the second column of the matrix $[\alpha_{\A, \A\, \U + \A}]_{\A}$ contains the number $1$ in the first position and the other values are zeros. From the matrix $[\alpha_{\A, \A, \A}]_{\A}$ we take all four elements and place them in the first and third columns (because in the first object the first and third symbols $\A$ are double underscored) and in the second and third rows (because in the second object the second and third symbols $\A$ are double underscored).
\end{example}

It follows from the definition that if any of the objects $X, Y, Z$ is equal to $\U$, then $\alpha_{X, Y, Z} = id_{(X\otimes Y)\otimes Z}$.

To draw diagrams of associativity isomorphisms, it's convenient to use arrows of different colours. These arrows correspond to different values. The arrow with the value $\frac{1}{\varepsilon}$ is drawn in blue, the arrow with the value $\frac{x}{\sqrt{\varepsilon}}$ in red, the arrow with the value $\frac{1}{x\sqrt{\varepsilon}}$ in green, the arrow with the value $-\frac{1}{\varepsilon}$ in orange and the arrow with the value $1$ in black (figure \ref{Figure:AssociatorArrows}).

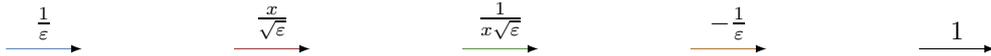
\begin{figure}[h]
\begin{center}
\ \hfill
\begin{tikzpicture}
	\draw[ma_blue] (0, 0) -- (1, 0) node[above, midway] {$\frac{1}{\varepsilon}$};
\end{tikzpicture}
\hfill
\begin{tikzpicture}
	\draw[ma_red] (0, 0) -- (1, 0) node[above, midway] {$\frac{x}{\sqrt{\varepsilon}}$};
\end{tikzpicture}
\hfill
\begin{tikzpicture}
	\draw[ma_green] (0, 0) -- (1, 0) node[above, midway] {$\frac{1}{x\sqrt{\varepsilon}}$};
\end{tikzpicture}
\hfill
\begin{tikzpicture}
	\draw[ma_orange] (0, 0) -- (1, 0) node[above, midway] {$-\frac{1}{\varepsilon}$};
\end{tikzpicture}
\hfill
\begin{tikzpicture}
	\draw[-latex] (0, 0) -- (1, 0) node[above, midway] {$1$};
\end{tikzpicture}
\hfill \ \ 
\end{center}
\caption{\label{Figure:AssociatorArrows}Different colours for arrows with different values}
\end{figure}

The figure \ref{Figure:AssociatorDiagrams} shows diagrams for associative isomorphisms $\alpha_{\A, \A, \A}$, $\alpha_{\U + \A, \A, \A}$, $\alpha_{\A, \U + \A, \A}$, $\alpha_{\A, \A, \U + \A}$, $\alpha_{\A, \A, \A}\otimes id_{\A}$ and $id_{\A}\otimes \alpha_{\A, \A, \A}$.

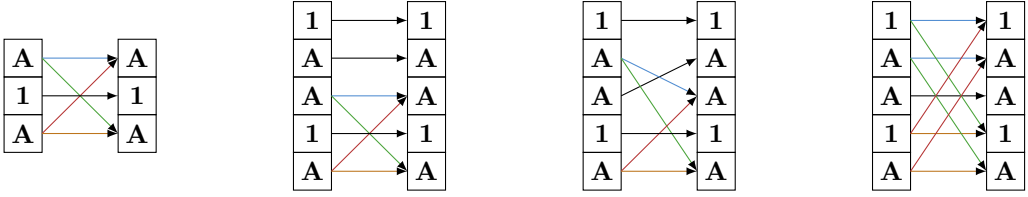
\begin{figure}[h]
\begin{center}
\ \hfill
\begin{tikzpicture}[scale=0.5, baseline={(current bounding box.center)}]
	\draw (0, -1) rectangle ++(1, 1) node[pos=0.5] {$\A$};
	\draw (0, -2) rectangle ++(1, 1) node[pos=0.5] {$\U$};
	\draw (0, -3) rectangle ++(1, 1) node[pos=0.5] {$\A$};
	
	\draw (3, -1) rectangle ++(1, 1) node[pos=0.5] {$\A$};
	\draw (3, -2) rectangle ++(1, 1) node[pos=0.5] {$\U$};
	\draw (3, -3) rectangle ++(1, 1) node[pos=0.5] {$\A$};
	
	\draw[ma_blue] (1, -0.5) -- (3, -0.5);
	\draw[ma_green] (1, -0.5) -- (3, -2.5);
	\draw[ma_red] (1, -2.5) -- (3, -0.5);
	\draw[ma_orange] (1, -2.5) -- (3, -2.5);
	
	\draw[ma_black] (1, -1.5) -- (3, -1.5);
\end{tikzpicture}
\hfill
\begin{tikzpicture}[scale=0.5, baseline={(current bounding box.center)}]
	\draw (0, -1) rectangle ++(1, 1) node[pos=0.5] {$\U$};
	\draw (0, -2) rectangle ++(1, 1) node[pos=0.5] {$\A$};
	\draw (0, -3) rectangle ++(1, 1) node[pos=0.5] {$\A$};
	\draw (0, -4) rectangle ++(1, 1) node[pos=0.5] {$\U$};
	\draw (0, -5) rectangle ++(1, 1) node[pos=0.5] {$\A$};
	
	\draw (3, -1) rectangle ++(1, 1) node[pos=0.5] {$\U$};
	\draw (3, -2) rectangle ++(1, 1) node[pos=0.5] {$\A$};
	\draw (3, -3) rectangle ++(1, 1) node[pos=0.5] {$\A$};
	\draw (3, -4) rectangle ++(1, 1) node[pos=0.5] {$\U$};
	\draw (3, -5) rectangle ++(1, 1) node[pos=0.5] {$\A$};
	
	\draw[ma_black] (1, -0.5) -- (3, -0.5);
	\draw[ma_black] (1, -1.5) -- (3, -1.5);
	\draw[ma_blue] (1, -2.5) -- (3, -2.5);
	\draw[ma_green] (1, -2.5) -- (3, -4.5);
	\draw[ma_red] (1, -4.5) -- (3, -2.5);
	\draw[ma_orange] (1, -4.5) -- (3, -4.5);
	
	\draw[ma_black] (1, -3.5) -- (3, -3.5);
\end{tikzpicture}
\hfill
\begin{tikzpicture}[scale=0.5, baseline={(current bounding box.center)}]
	\draw (0, -1) rectangle ++(1, 1) node[pos=0.5] {$\U$};
	\draw (0, -2) rectangle ++(1, 1) node[pos=0.5] {$\A$};
	\draw (0, -3) rectangle ++(1, 1) node[pos=0.5] {$\A$};
	\draw (0, -4) rectangle ++(1, 1) node[pos=0.5] {$\U$};
	\draw (0, -5) rectangle ++(1, 1) node[pos=0.5] {$\A$};
	
	\draw (3, -1) rectangle ++(1, 1) node[pos=0.5] {$\U$};
	\draw (3, -2) rectangle ++(1, 1) node[pos=0.5] {$\A$};
	\draw (3, -3) rectangle ++(1, 1) node[pos=0.5] {$\A$};
	\draw (3, -4) rectangle ++(1, 1) node[pos=0.5] {$\U$};
	\draw (3, -5) rectangle ++(1, 1) node[pos=0.5] {$\A$};
	
	\draw[ma_black] (1, -0.5) -- (3, -0.5);
	\draw[ma_black] (1, -2.5) -- (3, -1.5);
	\draw[ma_blue] (1, -1.5) -- (3, -2.5);
	\draw[ma_green] (1, -1.5) -- (3, -4.5);
	\draw[ma_red] (1, -4.5) -- (3, -2.5);
	\draw[ma_orange] (1, -4.5) -- (3, -4.5);
	
	\draw[ma_black] (1, -3.5) -- (3, -3.5);
\end{tikzpicture}
\hfill
\begin{tikzpicture}[scale=0.5, baseline={(current bounding box.center)}]
	\draw (0, -1) rectangle ++(1, 1) node[pos=0.5] {$\U$};
	\draw (0, -2) rectangle ++(1, 1) node[pos=0.5] {$\A$};
	\draw (0, -3) rectangle ++(1, 1) node[pos=0.5] {$\A$};
	\draw (0, -4) rectangle ++(1, 1) node[pos=0.5] {$\U$};
	\draw (0, -5) rectangle ++(1, 1) node[pos=0.5] {$\A$};
	
	\draw (3, -1) rectangle ++(1, 1) node[pos=0.5] {$\U$};
	\draw (3, -2) rectangle ++(1, 1) node[pos=0.5] {$\A$};
	\draw (3, -3) rectangle ++(1, 1) node[pos=0.5] {$\A$};
	\draw (3, -4) rectangle ++(1, 1) node[pos=0.5] {$\U$};
	\draw (3, -5) rectangle ++(1, 1) node[pos=0.5] {$\A$};
	
	\draw[ma_blue] (1, -0.5) -- (3, -0.5);
	\draw[ma_blue] (1, -1.5) -- (3, -1.5);
	\draw[ma_black] (1, -2.5) -- (3, -2.5);
	\draw[ma_orange] (1, -3.5) -- (3, -3.5);
	\draw[ma_orange] (1, -4.5) -- (3, -4.5);
	\draw[ma_green] (1, -0.5) -- (3, -3.5);
	\draw[ma_green] (1, -1.5) -- (3, -4.5);
	\draw[ma_red] (1, -3.5) -- (3, -0.5);
	\draw[ma_red] (1, -4.5) -- (3, -1.5);
\end{tikzpicture}
\hfill \ \ 
\end{center}
\caption{\label{Figure:AssociatorDiagrams}From left to right: diagrams of the isomorphisms $\alpha_{\A, \A, \A}$ (first), $\alpha_{\U + \A, \A, \A}$ and $\alpha_{\A, \U + \A, \A}$ (second), $\alpha_{\A, \A, \U + \A}$ (third), $\alpha_{\A, \A, \A}\otimes id_{\A}$ and $id_{\A}\otimes \alpha_{\A, \A, \A}$ (forth)}
\end{figure}

\begin{lemma}
\label{Lemma:Associators}
The family of isomorphisms $\alpha_{X, Y, Z}\in Hom((X\otimes Y)\otimes Z), X\otimes (Y\otimes Z))$ for all objects $X, Y, Z$ of the category $\E$ satisfies to the pentagon relation $$(\alpha_{X, Y, Z}\otimes id_W) \circ \alpha_{X, Y\otimes Z, W} \circ (id_X\otimes \alpha_{Y, Z, W}) = \alpha_{X\otimes Y, Z, W}\circ \alpha_{X, Y, Z\otimes W}
$$ for any objects $X, Y, Z, W$.
\end{lemma}
\begin{remark}
As we have already mentioned, the composition of the morphisms in the statement of the lemma should be read from the left to the right.

The alternative way of formulating the pentagon relation is to say that the diagram shown in the figure \ref{Figure:PentagonRelation} is commutative.
\end{remark}

\begin{figure}[h]
\begin{center}
\begin{tikzcd}[column sep=-0.75cm]
 & (X\otimes (Y\otimes Z))\otimes W \arrow[rr, "\alpha_{X, Y\otimes Z, W}"] & & X\otimes((Y\otimes Z)\otimes W) \arrow[rd, "id_X\otimes \alpha_{Y, Z, W}"]& \\
((X\otimes Y)\otimes Z)\otimes W  \arrow[ru, "\alpha_{X, Y, Z}\otimes id_W"] \arrow[rrd, "\alpha_{X\otimes Y, Z, W}"]& & & & X\otimes(Y\otimes(Z\otimes W)) \\
 & & (X\otimes Y)\otimes (Z\otimes W) \arrow[rru, "\alpha_{X, Y, Z\otimes W}"] & & 
\end{tikzcd}
\end{center}
\caption{\label{Figure:PentagonRelation}Pentagon relation for associativity isomorphisms}
\end{figure}
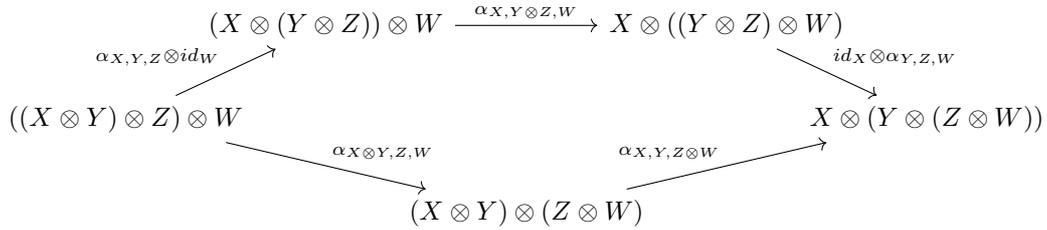
\begin{proof}
It's enough to prove the lemma only for simple objects $X, Y, Z, W\in I$. If one of them is equal to $\U$, then the statement is trivial. Consider the case where $X = Y = Z = W = \A$.

The diagram for the morphism $\alpha_{\A, \A\otimes \A, \A}$ is shown on figure \ref{Figure:AssociatorDiagrams} (second from the left). The diagram for the morphisms $\alpha_{\A, \A, \A}\otimes id_{\A}$ and $id_{\A}\otimes \alpha_{\A, \A, \A}$ is shown on the figure \ref{Figure:AssociatorDiagrams} (the most right).

Denote the morphism on the left hand side of the lemma statement as $L\in Hom (((\A\otimes \A)\otimes \A)\otimes \A, \A\otimes (\A\otimes (\A\otimes \A)))$. This morphism is defined by the following two matrices: 
\begin{multline*}
[L]_{\U} = [id_{\A}\otimes \alpha_{\A, \A, \A}]_{\U}\cdot [\alpha_{\A, \U + \A, \A}]_{\U}\cdot [\alpha_{\A, \A, \A}\otimes id_{\A}]_{\U} = \\ = \begin{pmatrix}
\frac{1}{\varepsilon} & \frac{x}{\sqrt{\varepsilon}} \\
\frac{1}{x\sqrt{\varepsilon}} & -\frac{1}{\varepsilon}
\end{pmatrix} \cdot \begin{pmatrix}
1 & 0 \\
0 & 1
\end{pmatrix} \cdot \begin{pmatrix}
\frac{1}{\varepsilon} & \frac{x}{\sqrt{\varepsilon}} \\
\frac{1}{x\sqrt{\varepsilon}} & -\frac{1}{\varepsilon}
\end{pmatrix} = \begin{pmatrix}
\frac{1}{\varepsilon^2} + \frac{1}{\varepsilon} & 0 \\
0 & \frac{1}{\varepsilon^2} + \frac{1}{\varepsilon}
\end{pmatrix} = \begin{pmatrix}
1 & 0 \\
0 & 1
\end{pmatrix},
\end{multline*}
\begin{multline*}
[L]_{\A} = [id_{\A}\otimes \alpha_{\A, \A, \A}]_{\A}\cdot [\alpha_{\A, \U + \A, \A}]_{\A}\cdot [\alpha_{\A, \A, \A}\otimes id_{\A}]_{\A} = \\ =
\begin{pmatrix}
\frac{1}{\varepsilon} & 0 & \frac{x}{\sqrt{\varepsilon}} \\
0 & 1 & 0 \\
\frac{1}{x\sqrt{\varepsilon}} & 0 & -\frac{1}{\varepsilon}
\end{pmatrix}
\cdot
\begin{pmatrix}
1 & 0 & 0 \\
0 & \frac{1}{\varepsilon} & \frac{x}{\varepsilon} \\
0 & \frac{1}{x\sqrt{\varepsilon}} & -\frac{1}{\sqrt{\varepsilon}}
\end{pmatrix} \cdot 
\begin{pmatrix}
\frac{1}{\varepsilon} & 0 & \frac{x}{\sqrt{\varepsilon}} \\
0 & 1 & 0 \\
\frac{1}{x\sqrt{\varepsilon}} & 0 & -\frac{1}{\varepsilon}
\end{pmatrix} = \\ 
\begin{pmatrix}
0 & \frac{1}{\varepsilon} & \frac{x}{\varepsilon\sqrt{\varepsilon}}\left(1 + \frac{1}{\varepsilon}\right) \\
\frac{1}{\varepsilon} & \frac{1}{\varepsilon} & -\frac{x}{\varepsilon\sqrt{\varepsilon}} \\
\frac{1}{x\varepsilon\sqrt{\varepsilon}}\left(1 + \frac{1}{\varepsilon}\right) & -\frac{1}{x\varepsilon\sqrt{\varepsilon}} & \frac{1}{\varepsilon} - \frac{1}{\varepsilon^3}
\end{pmatrix} = 
\begin{pmatrix}
0 & \frac{1}{\varepsilon} & \frac{x}{\sqrt{\varepsilon}} \\
\frac{1}{\varepsilon} & \frac{1}{\varepsilon} & -\frac{x}{\varepsilon\sqrt{\varepsilon}} \\
\frac{1}{x\sqrt{\varepsilon}} & -\frac{1}{x\varepsilon\sqrt{\varepsilon}} & \frac{1}{\varepsilon^2}
\end{pmatrix}.
\end{multline*}

We use the fact that $1 + \frac{1}{\varepsilon} = \varepsilon$ and $\frac{1}{\varepsilon} - \frac{1}{\varepsilon^3} = \frac{1}{\varepsilon^2}$.

Next, denote the morphism in the right part of the lemma statement as $R\in Hom (((\A\otimes \A)\otimes \A)\otimes \A, \A\otimes (\A\otimes (\A\otimes \A)))$. Matrices of the morphisms $\alpha_{\U + \A, \A, \A}$ and $\alpha_{\A, \A, \U + \A}$ computed in the example \ref{Example:Associator}. The morphism $R$ defined by the following two matrices:
$$
[R]_{\U} = [\alpha_{\A, \A, \U + \A}]_{\U}\cdot [\alpha_{\U + \A, \A, \A}]_{\U} = 
\begin{pmatrix}
1 & 0 \\
0 & 1
\end{pmatrix} \cdot \begin{pmatrix}
1 & 0 \\
0 & 1
\end{pmatrix} = \begin{pmatrix}
1 & 0 \\
0 & 1
\end{pmatrix},
$$

$$
[R]_{\A} = [\alpha_{\A, \A, \U + \A}]_{\A}\cdot [\alpha_{\U + \A, \A, \A}]_{\A} = 
\begin{pmatrix}
0 & 1 & 0 \\
\frac{1}{\varepsilon} & 0 & \frac{x}{\sqrt{\varepsilon}} \\
\frac{1}{x\sqrt{\varepsilon}} & 0 & -\frac{1}{\varepsilon}
\end{pmatrix} \cdot 
\begin{pmatrix}
1 & 0 & 0 \\
0 & \frac{1}{\varepsilon} & \frac{x}{\sqrt{\varepsilon}} \\
0 & \frac{1}{x\sqrt{\varepsilon}} & -\frac{1}{\varepsilon}
\end{pmatrix} = 
\begin{pmatrix}
0 & \frac{1}{\varepsilon} & \frac{x}{\sqrt{\varepsilon}} \\
\frac{1}{\varepsilon} & \frac{1}{\varepsilon} & -\frac{x}{\varepsilon\sqrt{\varepsilon}} \\
\frac{1}{x\sqrt{\varepsilon}} & -\frac{1}{x\varepsilon\sqrt{\varepsilon}} & \frac{1}{\varepsilon^2}
\end{pmatrix}.
$$

So, $[L]_{\U} = [R]_{\U}$ and $[L]_{\A} = [R]_{\A}$.
\end{proof}

\begin{remark}
\label{Remark:AssociatorValues}
We have defined the associativity isomorphisms by postulating matrices $[\alpha_{\A, \A, \A}]_{\U}$ and $[\alpha_{\A \A, \A}]_{\A}$.  It's easy to prove that these associativity isomorphisms satisfy the pentagon relation. But there's still the question of how to find these initial matrices.

We used the following approach. For simplicity, postulate that $\alpha_{X, Y, Z}$ are identical morphisms if at least one of the objects $X, Y, Z$ is equal to $\U$. Then consider matrices that define the morphism $\alpha_{\A, \A, \A}$, with variable values
$$
[\alpha_{\A, \A, \A}]_{\U} = \begin{pmatrix}
t
\end{pmatrix}, [\alpha_{\A, \A, \A}]_{\A} = \begin{pmatrix}
a & b \\
c & d
\end{pmatrix},
$$
where $t, a, b, c, d\in \mathbb{C}$. Next we need to write conditions for these variables. These conditions come from the pentagon relation.

$
[\alpha_{\A, \A, \A}\otimes id_{\A}]_{\U} = [id_{\A}\otimes \alpha_{\A, \A, \A}]_{\U} = \begin{pmatrix}
a & b \\
c & d
\end{pmatrix}, [\alpha_{\A, \A\otimes \A, \A}]_{\U} = [\alpha_{\A\otimes \A, \A, \A}]_{\U} = [\alpha_{\A, \A, \A\otimes \A}]_{\U} = \begin{pmatrix}
1 & 0 \\
0 & t
\end{pmatrix}.
$

$
[\alpha_{\A, \A, \A}\otimes id_{\A}]_{\A} = [id_{\A}\otimes \alpha_{\A, \A, \A}]_{\A} = \begin{pmatrix}
a & 0 & b \\
0 & 1 & 0 \\
c & 0 & d
\end{pmatrix}$, 

$[\alpha_{\A, \A\otimes \A, \A}]_{\A} = [\alpha_{\A\otimes \A, \A, \A}]_{\A} = \begin{pmatrix}
1 & 0 & 0 \\
0 & a & b \\
0 & c & d
\end{pmatrix}, [\alpha_{\A, \A, \A\otimes \A}]_{\A} = \begin{pmatrix}
0 & 1 & 0 \\
a & 0 & b \\
c & 0 & d
\end{pmatrix}.
$

The pentagon relation gives the next two matrix equations:
$$
\begin{pmatrix}
a & b\\
c & d
\end{pmatrix}\cdot \begin{pmatrix}
1 & 0\\
0 & t
\end{pmatrix}\cdot \begin{pmatrix}
a & b\\
c & d
\end{pmatrix} = \begin{pmatrix}
1 & 0\\
0 & t
\end{pmatrix}\cdot \begin{pmatrix}
1 & 0\\
0 & t
\end{pmatrix},
$$

$$
\begin{pmatrix}
a & 0 & b \\
0 & 1 & 0 \\
c & 0 & d
\end{pmatrix} \cdot \begin{pmatrix}
1 & 0 & 0 \\
0 & a & b \\
0 & c & d
\end{pmatrix} \cdot \begin{pmatrix}
a & 0 & b \\
0 & 1 & 0 \\
c & 0 & d
\end{pmatrix} = \begin{pmatrix}
a & 0 & b \\
0 & 1 & 0 \\
c & 0 & d
\end{pmatrix} \cdot \begin{pmatrix}
1 & 0 & 0 \\
0 & a & b \\
0 & c & d
\end{pmatrix}.
$$

This leads to the following system of twelve equations:
\begin{center}
\begin{tabular}{lll}
$a^2 + b\cdot c\cdot t = 1$, & $a^2 + b\cdot c\cdot d = 0$, & $b\cdot d\cdot t = b\cdot d$, \\
$a\cdot b + b\cdot d\cdot t = 0$, & $b\cdot c\cdot t = a$, & $a\cdot c + c\cdot d^2 = c$, \\
$a\cdot c + c\cdot d\cdot t = 0$, & $a\cdot b + b\cdot d^2 = b$, & $c\cdot d\cdot t = c\cdot d$, \\
$b\cdot c + d^2\cdot t = t^2$, & $a\cdot t^2 = b\cdot c$, & $b\cdot c + d^3 = d^2$.
\end{tabular}
\end{center}

This system has an infinite family of non-trivial solutions, parametrised by complex parameter $x\neq 0$: $$t = 1, a = \frac{1}{\varepsilon}, b = \frac{x}{\sqrt{\varepsilon}}, c = \frac{1}{x\sqrt{\varepsilon}}, d = -\frac{1}{\varepsilon},$$ where $\varepsilon^2 = \varepsilon + 1$.
\end{remark}

\begin{lemma}
\label{Lemma:MorphismsAssociativity}
Let $f\in Hom(X_1, Y_1)$, $g\in Hom (X_2, Y_2)$ and $h\in Hom(X_3, Y_3)$ are three morphisms in the category $\E$. Then $$(f\otimes g)\otimes h \circ \alpha_{Y_1, Y_2, Y_3} = \alpha_{X_1, X_2, X_3}\circ f\otimes (g\otimes h).$$
\end{lemma}
\begin{proof}
As in the previous lemmas, it's sufficient to prove the statement only for the case where $X_1 = Y_1 = X_2 = Y_2 = X_3 = Y_3 = \A$. In this case each morphism $f, g, h\in Hom (\A, \A)$ is defined by a number, which we will denote $\widehat{f}, \widehat{g}, \widehat{h}\in \mathbb{C}$ respectively.

For any morphism $\gamma$ and any number $c\in\mathbb{C}$ denote the $c\cdot \gamma$ morphism which acts like a $\gamma$, but all values multiplied by $c$. Then
\begin{multline*}
(f\otimes g)\otimes h\circ \alpha_{\A, \A, \A} = (\widehat{f}\cdot \widehat{g}\cdot \widehat{h})\cdot id_{\A + \U + \A}\circ \alpha_{\A, \A, \A} = (\widehat{f}\cdot \widehat{g}\cdot \widehat{h})\cdot \alpha_{\A, \A, \A} = \\ = (\widehat{f}\cdot \widehat{g}\cdot \widehat{h})\cdot \alpha_{\A, \A, \A} \circ id_{\A + \U + \A} = \alpha_{\A, \A, \A}\circ (\widehat{f}\cdot \widehat{g}\cdot \widehat{h})\cdot id_{\A + \U + \A} = \alpha_{\A, \A, \A}\circ f\otimes (g\otimes h).
\end{multline*}
\end{proof}

\subsubsection{Unit object}

Define the unit object $\mathbbm{1} = \U$.

\begin{lemma}
\label{Lemma:Unit}
For any object $X$ in the category $\E$: $X\otimes \mathbbm{1} = \mathbbm{1}\otimes X = X$.

For any morphism $f\in Hom (X, Y)$: $f\otimes id_{\mathbbm{1}} = id_{\mathbbm{1}}\otimes f = f$.
\end{lemma}
\begin{proof}
Both statements are obvious.
\end{proof}

\subsubsection{Left and right unit isomorphisms}

Define the family of isomorphisms $l_{X}\in Hom (\mathbbm{1}\otimes X, X)$ and $r_X\in Hom(X\otimes \mathbbm{1}, X)$ for each object $X$ of the category $\E$ as identical morphisms, i.e. $$l_X = id_X, r_X = id_X.$$

\begin{lemma}
\label{Lemma:UnitIsomorphisms}
The family of isomorphisms $l_X\in Hom(\mathbbm{1}\otimes X, X)$ and $r_X\in Hom(X\otimes \mathbbm{1}, X)$ for all objects $X$ in the category $\E$ satisfies to the triangle relation $$
(r_X\otimes id_Y)\circ (id_X\otimes l_Y) = \alpha_{X, \mathbbm{1}, Y}
$$
for any objects $X, Y$.

For any morphism $f\in Hom(X, Y)$
\begin{center}
$(id_{\mathbbm{1}}\otimes f)\circ l_Y = l_X\circ f$ and $(f\otimes id_{\mathbbm{1}})\circ r_Y = r_X\circ f$.
\end{center}
\end{lemma}
\begin{proof}
Both statements of the lemma are obvious, because $\alpha_{X, \mathbbm{1}, Y} = id_{X\otimes Y}$ and $l_X = r_X = id_X$.
\end{proof}

\begin{theorem}
\label{Theorem:MonoidalCategory}
The category $\E$ is monoidal.
\end{theorem}
\begin{proof}
The theorem follows from the following:
\begin{enumerate}
\item The tensor product is a functor from $\E\times \E$ to $\E$ by the lemmas \ref{Lemma:MorphismTensorProduct} and \ref{Lemma:IdentityMorphismsProduct};
\item The family of associativity isomorphisms $\alpha_{X, Y, Z}$ satisfies the pentagon relation by lemma \ref{Lemma:Associators};
\item The family of unit isomorphisms $l_X$ and $r_X$ satisfies the triangle relation by lemma \ref{Lemma:UnitIsomorphisms} (first statement);
\item Naturality of the associativity isomorphisms $\alpha_{X, Y, Z}$ follows from the lemma \ref{Lemma:MorphismsAssociativity};
\item Naturality of the unit isomorphisms $l_X$ and $r_X$ follows from the lemma \ref{Lemma:UnitIsomorphisms} (second statement).
\end{enumerate}
\end{proof}

\subsection{Braiding}

In this subsection we will define the family of isomorphisms $c_{X, Y}\in Hom (X\otimes Y, Y\otimes X)$ defined for each pair of objects $X, Y$ in the category $\E$.

In the previous subsection we defined associativity isomorphisms using the constant $\varepsilon$, which satisfies the equation $\varepsilon^2 = \varepsilon + 1$. There are two real numbers with this property: one is positive, the other negative. Let $\xi = e^{\frac{\pi}{5}}$ --- one of the primitive roots of $1$ of degree 10. If the constant $\varepsilon$ is positive, then $\varepsilon = \xi + \xi^{-1}$, and if it is negative, then $\varepsilon = \xi^3 + \xi^{-3}$. Introduce two constants $\beta_{\varepsilon}^{+}$ and $\beta_{\varepsilon}^{-}$, defined with respect to the value of the constant $\varepsilon$, in the following way:
\begin{center}
$\beta_{\varepsilon}^{+} = \begin{cases}
\xi^3, \ \text{if}\  \varepsilon = \xi + \xi^{-1}\\
\xi, \ \text{if}\  \varepsilon = \xi^3 + \xi^{-3}
\end{cases}$ and $\beta_{\varepsilon}^{-} = \begin{cases}
\xi^{-3}, \ \text{if}\  \varepsilon = \xi + \xi^{-1}\\
\xi^{-1}, \ \text{if}\  \varepsilon = \xi^3 + \xi^{-3}
\end{cases}$.
\end{center}

By $\beta_{\varepsilon}$ we mean any of $\beta_{\varepsilon}^{+}$ or $\beta_{\varepsilon}^{-}$.

\begin{remark}
\label{Remark:BraidingConstant}
It's easy to check that $$\beta_{\varepsilon}^2\cdot \varepsilon + \beta_{\varepsilon} + \varepsilon = 0.$$

Then $$\beta_{\varepsilon}^2 = -1 - \frac{\beta_{\varepsilon}}{\varepsilon}, \beta_{\varepsilon}^3 = \frac{1 - \beta_{\varepsilon}}{\varepsilon}, \beta_{\varepsilon}^4 = \beta_{\varepsilon} + \varepsilon - 1, \beta_{\varepsilon}^5 = -1.$$
\end{remark}

Let $X, Y\in I$ be simple objects of the category $\E$. If at least one of these objects is equal to $\U$, then define $c_{X, Y} = id_{X\otimes Y}$. If $X = Y = \A$, then $X\otimes Y = Y\otimes X = \U + \A$. In this case, define the isomorphism $c_{\A, \A}$ by the following two matrices: $$
[c_{\A, \A}]_{\U} = \begin{pmatrix}
\beta_{\varepsilon}^2
\end{pmatrix}, [c_{\A, \A}]_{\A} = \begin{pmatrix}
\beta_{\varepsilon}
\end{pmatrix}.
$$

Extends the defined isomorphisms $c_{X, Y}$ for simple objects to all objects of the category $\E$ by linearity (i.e. the same way as we did for associativity isomorphisms).

\begin{example}
\label{Example:Braidings}
Morphisms $c_{\A, \A}\otimes id_{\A}, id_{\A}\otimes c_{\A, \A}\in Hom(\A + \U + \A, \A + \U + \A)$ defined by the same matrices $$
[c_{\A, \A}\otimes id_{\A}]_{\U} = [id_{\A}\otimes c_{\A, \A}]_{\U} = \begin{pmatrix}
\beta_{\varepsilon}
\end{pmatrix}, 
[c_{\A, \A}\otimes id_{\A}]_{\A} = [id_{\A}\otimes c_{\A, \A}]_{\A} = \begin{pmatrix}
\beta_{\varepsilon}^2 & 0 \\
0 & \beta_{\varepsilon}
\end{pmatrix}.
$$

Morphisms $c_{\U + \A, \A}, c_{\A, \U + \A}\in Hom(\A + \U + \A, \A + \U + \A)$ are also defined by the same matrices $$
[c_{\U + \A, \A}]_{\U} = [c_{\A, \U + \A}]_{\U} = \begin{pmatrix}
\beta_{\varepsilon}^2
\end{pmatrix}, 
[c_{\U + \A, \A}]_{\A} = [c_{\A, \U + \A}]_{\A} = \begin{pmatrix}
1 & 0 \\
0 & \beta_{\varepsilon}
\end{pmatrix}.
$$

Diagrams of the morphisms $c_{\A, \A}, c_{\A, \A}\otimes id_{\A}, id_{\A}\otimes c_{\A, \A}, c_{\U + \A, \A}$ and $c_{\A, \U + \A}$ are shown in the figure \ref{Figure:BraidingExamples}.

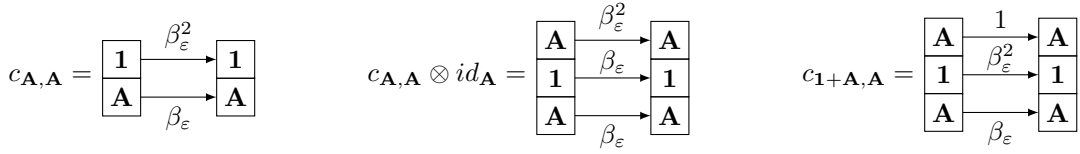
\begin{figure}[h]
\begin{center}
\ \hfill $c_{\A, \A} = $
\begin{tikzpicture}[scale=0.5, baseline={([yshift=-0.5ex]current bounding box.center)}]
	\draw (0, -1) rectangle ++(1, 1) node[pos=0.5] {$\U$};
	\draw (0, -2) rectangle ++(1, 1) node[pos=0.5] {$\A$};
	
	\draw (3, -1) rectangle ++(1, 1) node[pos=0.5] {$\U$};
	\draw (3, -2) rectangle ++(1, 1) node[pos=0.5] {$\A$};
	
	\draw[-latex] (1, -0.5) -- (3, -0.5) node[above, midway] {$\beta_{\varepsilon}^2$};
	\draw[-latex] (1, -1.5) -- (3, -1.5) node[below, midway] {$\beta_{\varepsilon}$};
\end{tikzpicture}
\hfill $c_{\A, \A}\otimes id_{\A} =$
\begin{tikzpicture}[scale=0.5, baseline={([yshift=-0.5ex]current bounding box.center)}]
	\draw (0, -1) rectangle ++(1, 1) node[pos=0.5] {$\A$};	
	\draw (0, -2) rectangle ++(1, 1) node[pos=0.5] {$\U$};
	\draw (0, -3) rectangle ++(1, 1) node[pos=0.5] {$\A$};
	
	\draw (3, -1) rectangle ++(1, 1) node[pos=0.5] {$\A$};
	\draw (3, -2) rectangle ++(1, 1) node[pos=0.5] {$\U$};
	\draw (3, -3) rectangle ++(1, 1) node[pos=0.5] {$\A$};
	
	\draw[-latex] (1, -0.5) -- (3, -0.5) node[above, midway] {$\beta_{\varepsilon}^2$};
	\draw[-latex] (1, -1.5) -- (3, -1.5) node[above=-2pt, midway] {$\beta_{\varepsilon}$};
	\draw[-latex] (1, -2.5) -- (3, -2.5) node[below, midway] {$\beta_{\varepsilon}$};
\end{tikzpicture}
\hfill $c_{\U + \A, \A} = $
\begin{tikzpicture}[scale=0.5, baseline={([yshift=-0.3ex]current bounding box.center)}]
	\draw (0, -1) rectangle ++(1, 1) node[pos=0.5] {$\A$};	
	\draw (0, -2) rectangle ++(1, 1) node[pos=0.5] {$\U$};
	\draw (0, -3) rectangle ++(1, 1) node[pos=0.5] {$\A$};
	
	\draw (3, -1) rectangle ++(1, 1) node[pos=0.5] {$\A$};
	\draw (3, -2) rectangle ++(1, 1) node[pos=0.5] {$\U$};
	\draw (3, -3) rectangle ++(1, 1) node[pos=0.5] {$\A$};
	
	\draw[-latex] (1, -0.5) -- (3, -0.5) node[above, midway] {$1$};
	\draw[-latex] (1, -1.5) -- (3, -1.5) node[above=-2.5pt, midway] {$\beta_{\varepsilon}^2$};
	\draw[-latex] (1, -2.5) -- (3, -2.5) node[below, midway] {$\beta_{\varepsilon}$};
\end{tikzpicture}
\hfill \ \ 
\end{center}
\caption{\label{Figure:BraidingExamples}Diagrams of the morphism $c_{\A, \A}$ (on the left), $c_{\A, \A}\otimes id_{\A}, id_{\A}\otimes c_{\A, \A}$ (in the center) and $c_{\U + \A, \A}, c_{\A, \U + \A}$ (on the right)}
\end{figure}
\end{example}

\begin{lemma}
\label{Lemma:BraidingNaturality}
Let $X_1, X_2, Y_1, Y_2$ be objects of the category $\E$. Let $f\in Hom(X_1, Y_1)$ and $g\in Hom(X_2, Y_2)$ be two morphisms. Then $$(f\otimes g)\circ c_{Y_1, Y_2} = c_{X_1, X_2}\circ (g\otimes f).$$
\end{lemma}
\begin{proof}
It's enough to prove the lemma for simple objects only. If at least one of the objects $X_1, X_2, Y_1, Y_2$ is equal to $\U$, then the statement is trivial. Consider the case where $X_1 = X_2 = Y_1 = Y_2 = \A$. Let $\widehat{f}$ be a number defining the morphism $f$, and let $\widehat{g}$ be a number defining the morphism $g$. The diagram of the left side of the lemma statement is shown in the figure \ref{Figure:BraidingNaturlityLeft}.

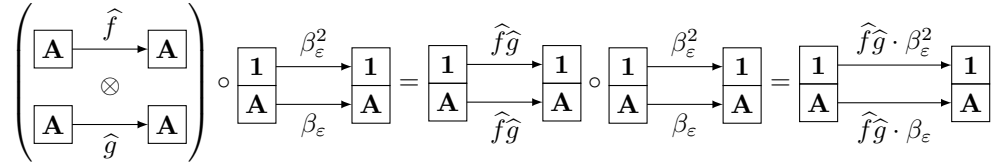
\begin{figure}[h]
$$\left(
\begin{tikzpicture}[scale=0.5, baseline={([yshift=-.5ex]current bounding box.center)}]
	\draw (0, -1) rectangle ++(1, 1) node[pos=0.5] {$\A$};
	\draw (3, -1) rectangle ++(1, 1) node[pos=0.5] {$\A$};
	\draw[-latex] (1, -0.5) -- (3, -0.5) node[above, midway] {$\widehat{f}$};
	
	\draw (2, -1.5) node {$\otimes$};
	
	\draw (0, -3) rectangle ++(1, 1) node[pos=0.5] {$\A$};
	\draw (3, -3) rectangle ++(1, 1) node[pos=0.5] {$\A$};
	\draw[-latex] (1, -2.5) -- (3, -2.5) node[below, midway] {$\widehat{g}$};
\end{tikzpicture}
\right) 
\circ
\begin{tikzpicture}[scale=0.5, baseline={([yshift=-.5ex]current bounding box.center)}]
	\draw (0, -1) rectangle ++(1, 1) node[pos=0.5] {$\U$};
	\draw (0, -2) rectangle ++(1, 1) node[pos=0.5] {$\A$};
	
	\draw (3, -1) rectangle ++(1, 1) node[pos=0.5] {$\U$};
	\draw (3, -2) rectangle ++(1, 1) node[pos=0.5] {$\A$};
	\draw[-latex] (1, -0.5) -- (3, -0.5) node[above, midway] {$\beta_{\varepsilon}^2$};
	\draw[-latex] (1, -1.5) -- (3, -1.5) node[below, midway] {$\beta_{\varepsilon}$};
\end{tikzpicture}
= 
\begin{tikzpicture}[scale=0.5, baseline={([yshift=-.5ex]current bounding box.center)}]
	\draw (0, -1) rectangle ++(1, 1) node[pos=0.5] {$\U$};
	\draw (0, -2) rectangle ++(1, 1) node[pos=0.5] {$\A$};
	\draw (3, -1) rectangle ++(1, 1) node[pos=0.5] {$\U$};
	\draw (3, -2) rectangle ++(1, 1) node[pos=0.5] {$\A$};
	\draw[-latex] (1, -0.5) -- (3, -0.5) node[above, midway] {$\widehat{f}\widehat{g}$};
	\draw[-latex] (1, -1.5) -- (3, -1.5) node[below, midway] {$\widehat{f}\widehat{g}$};
\end{tikzpicture}
\circ
\begin{tikzpicture}[scale=0.5, baseline={([yshift=-.5ex]current bounding box.center)}]
	\draw (0, -1) rectangle ++(1, 1) node[pos=0.5] {$\U$};
	\draw (0, -2) rectangle ++(1, 1) node[pos=0.5] {$\A$};
	
	\draw (3, -1) rectangle ++(1, 1) node[pos=0.5] {$\U$};
	\draw (3, -2) rectangle ++(1, 1) node[pos=0.5] {$\A$};
	\draw[-latex] (1, -0.5) -- (3, -0.5) node[above, midway] {$\beta_{\varepsilon}^2$};
	\draw[-latex] (1, -1.5) -- (3, -1.5) node[below, midway] {$\beta_{\varepsilon}$};
\end{tikzpicture}
=
\begin{tikzpicture}[scale=0.5, baseline={([yshift=-.5ex]current bounding box.center)}]
	\draw (0, -1) rectangle ++(1, 1) node[pos=0.5] {$\U$};
	\draw (0, -2) rectangle ++(1, 1) node[pos=0.5] {$\A$};
	
	\draw (4, -1) rectangle ++(1, 1) node[pos=0.5] {$\U$};
	\draw (4, -2) rectangle ++(1, 1) node[pos=0.5] {$\A$};
	\draw[-latex] (1, -0.5) -- (4, -0.5) node[above, midway] {$\widehat{f}\widehat{g}\cdot \beta_{\varepsilon}^2$};
	\draw[-latex] (1, -1.5) -- (4, -1.5) node[below, midway] {$\widehat{f}\widehat{g}\cdot \beta_{\varepsilon}$};
\end{tikzpicture}
$$
\caption{\label{Figure:BraidingNaturlityLeft}Diagram of the left hand side of the lemma \ref{Lemma:BraidingNaturality} statement}
\end{figure}

The similar diagram of the right hand side of the lemma statement is shown in the figure \ref{Figure:BraidingNaturlityRight}. These morphisms are the same.

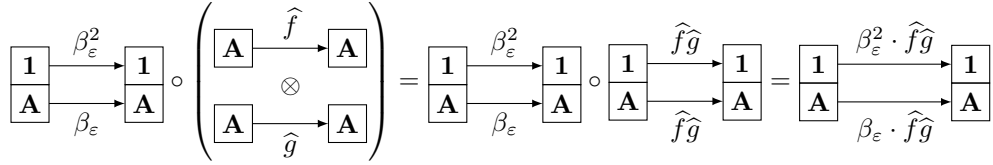
\begin{figure}[h]
$$
\begin{tikzpicture}[scale=0.5, baseline={([yshift=-.5ex]current bounding box.center)}]
	\draw (0, -1) rectangle ++(1, 1) node[pos=0.5] {$\U$};
	\draw (0, -2) rectangle ++(1, 1) node[pos=0.5] {$\A$};
	
	\draw (3, -1) rectangle ++(1, 1) node[pos=0.5] {$\U$};
	\draw (3, -2) rectangle ++(1, 1) node[pos=0.5] {$\A$};
	\draw[-latex] (1, -0.5) -- (3, -0.5) node[above, midway] {$\beta_{\varepsilon}^2$};
	\draw[-latex] (1, -1.5) -- (3, -1.5) node[below, midway] {$\beta_{\varepsilon}$};
\end{tikzpicture}
\circ
\left(
\begin{tikzpicture}[scale=0.5, baseline={([yshift=-.5ex]current bounding box.center)}]
	\draw (0, -1) rectangle ++(1, 1) node[pos=0.5] {$\A$};
	\draw (3, -1) rectangle ++(1, 1) node[pos=0.5] {$\A$};
	\draw[-latex] (1, -0.5) -- (3, -0.5) node[above, midway] {$\widehat{f}$};
	
	\draw (2, -1.5) node {$\otimes$};
	
	\draw (0, -3) rectangle ++(1, 1) node[pos=0.5] {$\A$};
	\draw (3, -3) rectangle ++(1, 1) node[pos=0.5] {$\A$};
	\draw[-latex] (1, -2.5) -- (3, -2.5) node[below, midway] {$\widehat{g}$};
\end{tikzpicture}
\right)
= 
\begin{tikzpicture}[scale=0.5, baseline={([yshift=-.5ex]current bounding box.center)}]
	\draw (0, -1) rectangle ++(1, 1) node[pos=0.5] {$\U$};
	\draw (0, -2) rectangle ++(1, 1) node[pos=0.5] {$\A$};
	
	\draw (3, -1) rectangle ++(1, 1) node[pos=0.5] {$\U$};
	\draw (3, -2) rectangle ++(1, 1) node[pos=0.5] {$\A$};
	\draw[-latex] (1, -0.5) -- (3, -0.5) node[above, midway] {$\beta_{\varepsilon}^2$};
	\draw[-latex] (1, -1.5) -- (3, -1.5) node[below, midway] {$\beta_{\varepsilon}$};
\end{tikzpicture}
\circ
\begin{tikzpicture}[scale=0.5, baseline={([yshift=-.5ex]current bounding box.center)}]
	\draw (0, -1) rectangle ++(1, 1) node[pos=0.5] {$\U$};
	\draw (0, -2) rectangle ++(1, 1) node[pos=0.5] {$\A$};
	\draw (3, -1) rectangle ++(1, 1) node[pos=0.5] {$\U$};
	\draw (3, -2) rectangle ++(1, 1) node[pos=0.5] {$\A$};
	\draw[-latex] (1, -0.5) -- (3, -0.5) node[above, midway] {$\widehat{f}\widehat{g}$};
	\draw[-latex] (1, -1.5) -- (3, -1.5) node[below, midway] {$\widehat{f}\widehat{g}$};
\end{tikzpicture}
=
\begin{tikzpicture}[scale=0.5, baseline={([yshift=-.5ex]current bounding box.center)}]
	\draw (0, -1) rectangle ++(1, 1) node[pos=0.5] {$\U$};
	\draw (0, -2) rectangle ++(1, 1) node[pos=0.5] {$\A$};
	
	\draw (4, -1) rectangle ++(1, 1) node[pos=0.5] {$\U$};
	\draw (4, -2) rectangle ++(1, 1) node[pos=0.5] {$\A$};
	\draw[-latex] (1, -0.5) -- (4, -0.5) node[above, midway] {$\beta_{\varepsilon}^2\cdot \widehat{f}\widehat{g}$};
	\draw[-latex] (1, -1.5) -- (4, -1.5) node[below, midway] {$\beta_{\varepsilon}\cdot \widehat{f}\widehat{g}$};
\end{tikzpicture}
$$
\caption{\label{Figure:BraidingNaturlityRight}Diagram of the right hand side of the lemma \ref{Lemma:BraidingNaturality} statement}
\end{figure}
\end{proof}

\begin{lemma}
\label{Lemma:Braiding}
For any three objects $X, Y, Z$ of the category $\E$: $$
\alpha_{X, Y, Z}\circ c_{X, Y\otimes Z}\circ \alpha_{Y, Z, X} =  (c_{X, Y}\otimes id_{Z})\circ \alpha_{Y, X, Z}\circ(id_{Y}\otimes c_{X, Z}),
$$
$$
\alpha_{X, Y, Z}^{-1}\circ c_{X\otimes Y, Z}\circ \alpha_{Z, X, Y}^{-1} = (id_{X}\otimes c_{Y, Z})\circ \alpha_{X, Z, Y}^{-1}\circ (c_{X, Z}\otimes id_{Y}).
$$
\end{lemma}
\begin{proof}
We will prove the first statement of the lemma, the second is similar. As in the previous lemmas, it's sufficient to prove only for simple objects. If one of the objects $X, Y, Z$ is equal to $\U$, then the statement is trivial. Consider the case $X = Y = Z = \A$.

Denote the left hand side of the lemma statement as $L\in Hom((\A\otimes \A)\otimes \A, (\A\otimes \A)\otimes \A)$. Then 
$$
[L]_{\U} = [\alpha_{\A, \A, \A}]_{\U}\cdot [c_{\A, \U + \A}]_{\U} \cdot [\alpha_{\A, \A, \A}]_{\U} = \begin{pmatrix}
1
\end{pmatrix}
\cdot \begin{pmatrix}
\beta_{\varepsilon}^2
\end{pmatrix} \cdot \begin{pmatrix}
1
\end{pmatrix} = \begin{pmatrix}
\beta_{\varepsilon}^2
\end{pmatrix},
$$
\begin{multline*}
[L]_{\A} = [\alpha_{\A, \A, \A}]_{\A}\cdot [c_{\A, \U + \A}]_{\A} \cdot [\alpha_{\A, \A, \A}]_{\A} = \begin{pmatrix}
\frac{1}{\varepsilon} & \frac{x}{\sqrt{\varepsilon}} \\
\frac{1}{x\sqrt{\varepsilon}} & -\frac{1}{\varepsilon}
\end{pmatrix} \cdot \begin{pmatrix}
1 & 0 \\
0 & \beta_{\varepsilon}
\end{pmatrix} \cdot \begin{pmatrix}
\frac{1}{\varepsilon} & \frac{x}{\sqrt{\varepsilon}} \\
\frac{1}{x\sqrt{\varepsilon}} & -\frac{1}{\varepsilon}
\end{pmatrix} = \\ = \begin{pmatrix}
\frac{\beta_{\varepsilon}}{\varepsilon} + 1 - \frac{1}{\varepsilon} & \frac{x}{\varepsilon\sqrt{\varepsilon}}(1 - \beta_{\varepsilon}) \\
\frac{1}{x\varepsilon\sqrt{\varepsilon}}(1 - \beta_{\varepsilon}) & \frac{1}{\varepsilon} + \frac{\beta_{\varepsilon}}{\varepsilon^2}
\end{pmatrix}.
\end{multline*}

Denote the right hand side of the lemma statement as $R\in Hom ((\A\otimes \A)\otimes \A, (\A\otimes \A)\otimes \A)$. Then $$
[R]_{\U} = [id_{\A}\otimes c_{\A, \A}]_{\U} \cdot [\alpha_{\A, \A, \A}]_{\U}\cdot [c_{\A, \A}\otimes id_{\A}]_{\U} = \begin{pmatrix}
\beta_{\varepsilon}
\end{pmatrix} \cdot \begin{pmatrix}
1
\end{pmatrix} \cdot \begin{pmatrix}
\beta_{\varepsilon}
\end{pmatrix} = \begin{pmatrix}
\beta_{\varepsilon}^2
\end{pmatrix},
$$
\begin{multline*}
[R]_{\A} = [id_{\A}\otimes c_{\A, \A}]_{\A} \cdot [\alpha_{\A, \A, \A}]_{\A}\cdot [c_{\A, \A}\otimes id_{\A}]_{\A} = \begin{pmatrix}
\beta_{\varepsilon}^2 & 0 \\
0 & \beta_{\varepsilon}
\end{pmatrix} \cdot \begin{pmatrix}
\frac{1}{\varepsilon} & \frac{x}{\sqrt{\varepsilon}} \\
\frac{1}{x\sqrt{\varepsilon}} & -\frac{1}{\varepsilon}
\end{pmatrix} \cdot \begin{pmatrix}
\beta_{\varepsilon}^2 & 0 \\
0 & \beta_{\varepsilon}
\end{pmatrix} = \\ = \begin{pmatrix}
\frac{\beta_{\varepsilon}^4}{\varepsilon} & \frac{x\beta_{\varepsilon}^3}{\sqrt{\varepsilon}} \\
\frac{\beta_{\varepsilon}^3}{x\sqrt{\varepsilon}} & -\frac{\beta_{\varepsilon}^2}{\varepsilon}
\end{pmatrix} = \begin{pmatrix}
\frac{\beta_{\varepsilon}}{\varepsilon} + 1 - \frac{1}{\varepsilon} & \frac{x(1 - \beta_{\varepsilon})}{\varepsilon\sqrt{\varepsilon}} \\
\frac{1 - \beta_{\varepsilon}}{x\varepsilon\sqrt{\varepsilon}} & \frac{1}{\varepsilon} + \frac{\beta_{\varepsilon}}{\varepsilon^2}
\end{pmatrix}.
\end{multline*}

In the last equality we used the remark \ref{Remark:BraidingConstant}.

So, $[L]_{\U} = [R]_{\U}$ and $[L]_{\A} = [R]_{\A}$. Hence $L = R$.

The second statement of the lemma is proved in a similar way. For the morphism $\alpha_{\A, \A, \A}^{-1}$ we can use the matrices $[\alpha_{\A, \A, \A}]_{\U}$ and $[\alpha_{\A, \A, \A}]_{\A}$, because each of them has order 2.
\end{proof}

\begin{remark}
As in the remark \ref{Remark:AssociatorValues}, describe how it's possible to find the isomorphism $c_{\A, \A}$. All other isomorphisms $c_{X, Y}$ follow from this.

Let the isomorphism $c_{\A, \A}$ be defined by two matrices $$
[c_{\A, \A}]_{\U} = \begin{pmatrix}
p
\end{pmatrix}, [c_{\A, \A}]_{\A} = \begin{pmatrix}
q
\end{pmatrix}.
$$

Then $$[c_{\A, \U + \A}]_{\U} = [c_{\U + \A, \A}]_{\U} = \begin{pmatrix}
p
\end{pmatrix}, [c_{\A, \U + \A}]_{\A} = [c_{\U + \A, \A}]_{\A} = \begin{pmatrix}
1 & 0 \\
0 & q
\end{pmatrix},
$$
$$
[c_{\A, \A}\otimes id_{\A}]_{\U} = [id_{\A}\otimes c_{\A, \A}]_{\U} = \begin{pmatrix}
q
\end{pmatrix}, [c_{\A, \A}\otimes id_{\A}]_{\A} = [id_{\A}\otimes c_{\A, \A}]_{\A} = \begin{pmatrix}
p & 0 \\
0 & q
\end{pmatrix}.
$$

Both conditions of the lemma \ref{Lemma:Braiding} lead to the following matrix equations
$$
\begin{pmatrix}
1
\end{pmatrix} \cdot \begin{pmatrix}
p
\end{pmatrix} \cdot \begin{pmatrix}
1
\end{pmatrix} = \begin{pmatrix}
q
\end{pmatrix} \cdot \begin{pmatrix}
1
\end{pmatrix} \cdot \begin{pmatrix}
q
\end{pmatrix},
$$
$$
\begin{pmatrix}
\frac{1}{\varepsilon} & \frac{x}{\sqrt{\varepsilon}} \\
\frac{1}{x\sqrt{\varepsilon}} & -\frac{1}{\varepsilon}
\end{pmatrix} \cdot \begin{pmatrix}
1 & 0 \\
0 & q
\end{pmatrix} \cdot \begin{pmatrix}
\frac{1}{\varepsilon} & \frac{x}{\sqrt{\varepsilon}} \\
\frac{1}{x\sqrt{\varepsilon}} & -\frac{1}{\varepsilon}
\end{pmatrix} = \begin{pmatrix}
p & 0 \\
0 & q
\end{pmatrix} \cdot \begin{pmatrix}
\frac{1}{\varepsilon} & \frac{x}{\sqrt{\varepsilon}} \\
\frac{1}{x\sqrt{\varepsilon}} & -\frac{1}{\varepsilon}
\end{pmatrix} \cdot \begin{pmatrix}
p & 0 \\
0 & q
\end{pmatrix}.
$$

This leads to the following system with two four equations and two variables ($p$ and $q$):
$$
\begin{cases}
p = q^2 \\
\frac{1}{\varepsilon^2} + \frac{q}{\varepsilon} = \frac{p^2}{\varepsilon} \\
\frac{1}{\varepsilon} - \frac{q}{\varepsilon} = pq \\
\frac{1}{\varepsilon} + \frac{q}{\varepsilon^2} = -\frac{q^2}{\varepsilon}
\end{cases}.
$$

The values $p = \beta_{\varepsilon}^2$ and $q = \beta_{\varepsilon}$ are the solution of this system.
\end{remark}

\begin{theorem}
\label{Theorem:BraidedCategory}
The category $\E$ is a braided monoidal category.
\end{theorem}
\begin{proof}
The family of isomorphisms $c_{X, Y}$ for all objects $X, Y$ in the category $\E$ form a braiding on $\E$. The naturality of these morphisms follows from the lemma \ref{Lemma:BraidingNaturality}. Hexagon braiding relations are proved in the lemma \ref{Lemma:Braiding}.
\end{proof}

\subsection{Twist and duality}

\subsubsection{Twist}

For each object $X$ of the category $\E$, define the isomorphism $\theta_{X}\in Hom (X, X)$ in the following way: $\theta_{\U} = id_{\U}$, $\theta_{\A}$ defined by the number $\frac{1}{\beta_{\varepsilon}^2}$, and extend these isomorphisms by linearity to all other objects of the category $\E$.

\begin{remark}
\label{Remark:Twists}
The diagram of an arbitrary morphism $\theta_{X}\in Hom(X, X)$ is very simple. If $X = \sum\limits_{i = 1}^{n}x_i$, $x_i\in I$, then the diagrams consist of $n$ parallel arrows. If the arrow starts and ends at the object $x_i = \U$, then the associated value is $1$, otherwise the associated value is $\frac{1}{\beta_{\varepsilon}^2}$.
\end{remark}

\begin{lemma}
\label{Lemma:TwistNaturality}
Let $f\in Hom(X, Y)$ be a morphism in the category $\E$. Then $$f\circ \theta_{Y} = \theta_{X}\circ f.$$.
\end{lemma}
\begin{proof}
It follows easily from the definition of isomorphisms $\theta_{X}$ for all objects $X$ of the category $\E$.
\end{proof}

\begin{lemma}
\label{Lemma:TwistAndBraiding}
Let $X, Y$ be two objects of the category $\E$. Then $$\theta_{X\otimes Y} = (\theta_{X}\otimes \theta_{Y})\circ c_{X, Y}\circ c_{Y, X}.$$
\end{lemma}
\begin{proof}
It's enough to prove the lemma for simple objects only. If at least one of the objects $X, Y$ is equal to $\U$, then the statement is trivial. Consider the case where $X = Y = \A$. The diagram of the left side of the lemma statement is shown in the figure \ref{Figure:TwistTensorObjects} at the top, and the diagram of the right side of the lemma statement is shown in the same figure \ref{Figure:TwistTensorObjects} at the bottom. The diagrams of these morphisms are the same.
\begin{figure}[h]
\begin{center}
$\theta_{\A\otimes\A} = $
\begin{tikzpicture}[scale=0.5, baseline={([yshift=.2ex]current bounding box.center)}]
	\draw (0, -1) rectangle ++(1, 1) node[pos=0.5] {$\U$};
	\draw (0, -2) rectangle ++(1, 1) node[pos=0.5] {$\A$};
	\draw (3, -1) rectangle ++(1, 1) node[pos=0.5] {$\U$};
	\draw (3, -2) rectangle ++(1, 1) node[pos=0.5] {$\A$};
	\draw[-latex] (1, -0.5) -- (3, -0.5) node[above, midway] {$1$};
	\draw[-latex] (1, -1.5) -- (3, -1.5) node[below, midway] {$\frac{1}{\beta_{\varepsilon}^2}$};
\end{tikzpicture}

$(\theta_{\A}\otimes \theta_{\A})\circ c_{\A, \A}\circ c_{\A, \A} = $
\begin{tikzpicture}[scale=0.5, baseline={([yshift=-0.5ex]current bounding box.center)}]
	\draw (0, -1) rectangle ++(1, 1) node[pos=0.5] {$\U$};
	\draw (0, -2) rectangle ++(1, 1) node[pos=0.5] {$\A$};
	\draw (3, -1) rectangle ++(1, 1) node[pos=0.5] {$\U$};
	\draw (3, -2) rectangle ++(1, 1) node[pos=0.5] {$\A$};
	\draw[-latex] (1, -0.5) -- (3, -0.5) node[above, midway] {$\frac{1}{\beta_{\varepsilon}^4}$};
	\draw[-latex] (1, -1.5) -- (3, -1.5) node[below, midway] {$\frac{1}{\beta_{\varepsilon}^4}$};
\end{tikzpicture}
$\circ$
\begin{tikzpicture}[scale=0.5, baseline={([yshift=-.7ex]current bounding box.center)}]
	\draw (0, -1) rectangle ++(1, 1) node[pos=0.5] {$\U$};
	\draw (0, -2) rectangle ++(1, 1) node[pos=0.5] {$\A$};
	\draw (3, -1) rectangle ++(1, 1) node[pos=0.5] {$\U$};
	\draw (3, -2) rectangle ++(1, 1) node[pos=0.5] {$\A$};
	\draw[-latex] (1, -0.5) -- (3, -0.5) node[above, midway] {$\beta_{\varepsilon}^2$};
	\draw[-latex] (1, -1.5) -- (3, -1.5) node[below, midway] {$\beta_{\varepsilon}$};
\end{tikzpicture}
$\circ$
\begin{tikzpicture}[scale=0.5, baseline={([yshift=-.7ex]current bounding box.center)}]
	\draw (0, -1) rectangle ++(1, 1) node[pos=0.5] {$\U$};
	\draw (0, -2) rectangle ++(1, 1) node[pos=0.5] {$\A$};
	\draw (3, -1) rectangle ++(1, 1) node[pos=0.5] {$\U$};
	\draw (3, -2) rectangle ++(1, 1) node[pos=0.5] {$\A$};
	\draw[-latex] (1, -0.5) -- (3, -0.5) node[above, midway] {$\beta_{\varepsilon}^2$};
	\draw[-latex] (1, -1.5) -- (3, -1.5) node[below, midway] {$\beta_{\varepsilon}$};
\end{tikzpicture}
$=$
\begin{tikzpicture}[scale=0.5, baseline={([yshift=0.2ex]current bounding box.center)}]
	\draw (0, -1) rectangle ++(1, 1) node[pos=0.5] {$\U$};
	\draw (0, -2) rectangle ++(1, 1) node[pos=0.5] {$\A$};
	\draw (3, -1) rectangle ++(1, 1) node[pos=0.5] {$\U$};
	\draw (3, -2) rectangle ++(1, 1) node[pos=0.5] {$\A$};
	\draw[-latex] (1, -0.5) -- (3, -0.5) node[above, midway] {$1$};
	\draw[-latex] (1, -1.5) -- (3, -1.5) node[below, midway] {$\frac{1}{\beta_{\varepsilon}^2}$};
\end{tikzpicture}
\end{center}
\caption{\label{Figure:TwistTensorObjects}Diagrams for morphisms $\theta_{\A,\otimes \A}$ (at the top) and $(\theta_{\A}\otimes \theta_{\A})\circ c_{\A, \A}\circ c_{\A, \A}$ (at the bottom)}
\end{figure}
\end{proof}

\begin{remark}
\label{Remark:TwistCondition}
Denote the value that defines the isomorphism $\theta_{\A}$ by $s\in\mathbb{C}$. The diagram of the statement of the lemma \ref{Lemma:TwistAndBraiding} for $X = \A$ is shown in the figure \ref{Figure:TwistCondition}.
\begin{figure}[h]
\begin{center}
\begin{tikzpicture}[scale=0.5, baseline={([yshift=-0.7ex]current bounding box.center)}]
	\draw (0, -1) rectangle ++(1, 1) node[pos=0.5] {$\U$};
	\draw (0, -2) rectangle ++(1, 1) node[pos=0.5] {$\A$};
	\draw (3, -1) rectangle ++(1, 1) node[pos=0.5] {$\U$};
	\draw (3, -2) rectangle ++(1, 1) node[pos=0.5] {$\A$};
	\draw[-latex] (1, -0.5) -- (3, -0.5) node[above, midway] {$1$};
	\draw[-latex] (1, -1.5) -- (3, -1.5) node[below, midway] {$s$};
\end{tikzpicture}
$=$
\begin{tikzpicture}[scale=0.5, baseline={([yshift=-0.5ex]current bounding box.center)}]
	\draw (0, -1) rectangle ++(1, 1) node[pos=0.5] {$\U$};
	\draw (0, -2) rectangle ++(1, 1) node[pos=0.5] {$\A$};
	\draw (3, -1) rectangle ++(1, 1) node[pos=0.5] {$\U$};
	\draw (3, -2) rectangle ++(1, 1) node[pos=0.5] {$\A$};
	\draw[-latex] (1, -0.5) -- (3, -0.5) node[above, midway] {$s^2$};
	\draw[-latex] (1, -1.5) -- (3, -1.5) node[below, midway] {$s^2$};
\end{tikzpicture}
$\circ$
\begin{tikzpicture}[scale=0.5, baseline={([yshift=-.7ex]current bounding box.center)}]
	\draw (0, -1) rectangle ++(1, 1) node[pos=0.5] {$\U$};
	\draw (0, -2) rectangle ++(1, 1) node[pos=0.5] {$\A$};
	\draw (3, -1) rectangle ++(1, 1) node[pos=0.5] {$\U$};
	\draw (3, -2) rectangle ++(1, 1) node[pos=0.5] {$\A$};
	\draw[-latex] (1, -0.5) -- (3, -0.5) node[above, midway] {$\beta_{\varepsilon}^2$};
	\draw[-latex] (1, -1.5) -- (3, -1.5) node[below, midway] {$\beta_{\varepsilon}$};
\end{tikzpicture}
$\circ$
\begin{tikzpicture}[scale=0.5, baseline={([yshift=-.7ex]current bounding box.center)}]
	\draw (0, -1) rectangle ++(1, 1) node[pos=0.5] {$\U$};
	\draw (0, -2) rectangle ++(1, 1) node[pos=0.5] {$\A$};
	\draw (3, -1) rectangle ++(1, 1) node[pos=0.5] {$\U$};
	\draw (3, -2) rectangle ++(1, 1) node[pos=0.5] {$\A$};
	\draw[-latex] (1, -0.5) -- (3, -0.5) node[above, midway] {$\beta_{\varepsilon}^2$};
	\draw[-latex] (1, -1.5) -- (3, -1.5) node[below, midway] {$\beta_{\varepsilon}$};
\end{tikzpicture}
$=$
\begin{tikzpicture}[scale=0.5, baseline={([yshift=-.7ex]current bounding box.center)}]
	\draw (0, -1) rectangle ++(1, 1) node[pos=0.5] {$\U$};
	\draw (0, -2) rectangle ++(1, 1) node[pos=0.5] {$\A$};
	\draw (4, -1) rectangle ++(1, 1) node[pos=0.5] {$\U$};
	\draw (4, -2) rectangle ++(1, 1) node[pos=0.5] {$\A$};
	\draw[-latex] (1, -0.5) -- (4, -0.5) node[above, midway] {$s^2\cdot \beta_{\varepsilon}^4$};
	\draw[-latex] (1, -1.5) -- (4, -1.5) node[below, midway] {$s^2\cdot \beta_{\varepsilon}^2$};
\end{tikzpicture}
\end{center}
\caption{\label{Figure:TwistCondition}Condition $\theta_{\A\otimes \A} = (\theta_{\A}\otimes \theta_{\A})\circ c_{\A, \A}\circ c_{\A, \A}$}
\end{figure}

This leads to the system $$
\begin{cases}
s^2\cdot \beta_{\varepsilon}^4 = 1 \\
s^2\cdot \beta_{\varepsilon}^2 = s
\end{cases}
$$

This system has a unique solution $s = \frac{1}{\beta_{\varepsilon}^2}$.
\end{remark}

Denote by $diag(v_1, \ldots, v_n)$ the diagonal matrix of size $n\times n$, with diagonal elements $v_1, \ldots, v_n$. Denote $E_n = diag(1, \ldots, 1)$.

\begin{lemma}
\label{Lemma:TwistUnitMatrix}
Let $X = \sum\limits_{i = 1}^{n}x_i$, $x_i\in I,$ be an object of the category $\E$, $|X|_{\U} = n_{\U}$, $|X|_{\A} = n_{\A}$. Then the following holds:
\begin{enumerate}
\item $[\theta_{X}\otimes id_{X}]_{\U} = [id_{X}\otimes \theta_{X}]_{\U}$;
\item $[(\theta_{X}\otimes id_{X})\circ c_{X, X}]_{\U} = E_{n_{\U}^2 + n_{\A}^2}$.
\end{enumerate}
\end{lemma}
\begin{proof}
It follows from the definition that $[\theta_X]_{\U} = E_{n_{\U}}$ and $[\theta_{X}]_{\A} = \frac{1}{\beta_{\varepsilon}^2}\cdot E_{n_{\A}}$. It's clear that $[id_{X}]_{\U} = E_{n_{\U}}$ and $[id_{X}]_{\A} = E_{n_{\A}}$. Then $$[\theta_{X}\otimes id_X]_{\U} = diag(\delta_1, \delta_2, \ldots, \delta_{n_{\U}^2 + n_{\A}^2}),$$ where the value $\delta_i$ is equal to $1$ when the corresponding summand is obtained by multiplying $\U\otimes \U$, and the value $\delta_i$ is equal to $\frac{1}{\beta_{\varepsilon}^2}$ if the corresponding summand is obtained by multiplying $\A\otimes \A$ (this multiplication gives two summands: $\U$ and $\A$). Similarly, $$[id_X\otimes \theta_{X}]_{\U} = diag (\delta_1, \delta_2, \ldots, \delta_{n_{\U}^2 + n_{\A}^2}).$$

So, the first statement of the lemma it proved.

To prove the second statement, note that $$[c_{X, X}]_{\U} = diag(\sigma_1, \sigma_2, \ldots, \sigma_{n_{\U}^2 + n_{\A}^2}),$$ where $\sigma_i$ is equal to $1$ if the corresponding summand is obtained by multiplying $x_p\otimes x_q$, where at least one of $x_p$ or $x_q$ is equal to $\U$, and $\sigma_i$ is equal to $\beta_{\varepsilon}^2$ if the corresponding summand obtained by multiplying $\A\otimes \A$. Therefore the product $[c_{X, X}]_{\U}\cdot [\theta_{X}\otimes id_X]_{\U}$ is a matrix $E_{n_{\U}^2 + n_{\A}^2}$.
\end{proof}

\subsubsection{Duality}

In the category $\E$ every object is self-dual, that is $X^{*} = X$ for every object $X$ of the category $\E$.

Let $X = \sum\limits_{i = 1}^{n}x_i$, $x_i\in I$, be an object of the category $\E$, $|X|_{\U} = n_{\U}$, $|X|_{\A} = n_{\A}$. Define the morphisms $b_X\in Hom(\U, X\otimes X)$ and $d_X\in Hom(X\otimes X, \U)$ by the following.  Note that $x_i\otimes x_i$ is either $\U$ (if $x_i = \U$) or $\U + \A$ (if $x_i = \A$). Let $(x_i\otimes x_i)_1$ be the first summand (always equal to $\U$) in this product. The morphisms $b_X$ and $d_X$ completely defined by the matrices $$[b_X]_{\U} = 
\begin{pmatrix}
\delta_1 \\
\delta_2 \\
\vdots \\
\delta_{n_{\U}^2 + n_{\A}^2}
\end{pmatrix}, [d_{X}]_{\U} = \begin{pmatrix}
\sigma_1 & \sigma_2 & \ldots & \sigma_{n_{\U}^2 + n_{\A}^2}
\end{pmatrix}.
$$

Define $\delta_s = \sigma_s = 0$ if the $s$-th summand $\U$ in $X\otimes X$ is obtained as $x_i\otimes x_j$, $i\neq j$, and define $\delta_s = \varphi_i$, $\sigma_s = \psi_i$ if the $s$-th summand is $(x_i\otimes x_i)_1$. Then define $\varphi_i = y$, $\psi_i = \frac{1}{y}$ if $x_i = \U$ and $\varphi_i = y\sqrt{\varepsilon}$, $\psi_i = \frac{\sqrt{\varepsilon}}{y}$ if $x_i = \A$, $i \in \{1, \ldots, n\}$. Here $y\in \mathbb{C}$ is any fixed nonzero complex number.

\begin{example}
\label{Example:Dualities}
It's clear that $b_{\U} = d_{\U} = id_{\U}$.

Diagrams for $b_{\A}$ and $d_{\A}$ are shown in the figure \ref{Figure:DualityAExample}. In this case $\A\otimes \A = \U + \A$ and the first summand $\U$ is exactly $(\A\otimes \A)_1$. So the values associated with the unique arrow in diagrams of $b_{\A}$ and $d_{\A}$ are $y\sqrt{\varepsilon}$ and $\frac{\sqrt{\varepsilon}}{y}$ respectively.

\begin{figure}[h]
\begin{center}
\hfill $b_{\A} = $
\begin{tikzpicture}[scale=0.5, baseline={([yshift=-1.7ex]current bounding box.center)}]
	\draw (0, -1) rectangle ++(1, 1) node[pos=0.5] {$\U$};
	
	\draw (3, -1) rectangle ++(1, 1) node[pos=0.5] {$\U$};
	\draw (3, -2) rectangle ++(1, 1) node[pos=0.5] {$\A$};
	\draw[-latex] (1, -0.5) -- (3, -0.5) node[above, midway] {$y\sqrt{\varepsilon}$};
\end{tikzpicture}
\hfill
$d_{\A} = $
\begin{tikzpicture}[scale=0.5, baseline={([yshift=-2.15ex]current bounding box.center)}]
	\draw (3, -1) rectangle ++(1, 1) node[pos=0.5] {$\U$};
	
	\draw (0, -1) rectangle ++(1, 1) node[pos=0.5] {$\U$};
	\draw (0, -2) rectangle ++(1, 1) node[pos=0.5] {$\A$};
	\draw[-latex] (1, -0.5) -- (3, -0.5) node[above, midway] {$\frac{\sqrt{\varepsilon}}{y}$};
\end{tikzpicture}
\hfill $b_{\U + \A} = $
\begin{tikzpicture}[scale=0.5, baseline={([yshift=-1.5ex]current bounding box.center)}]
	\draw (0, -1) rectangle ++(1, 1) node[pos=0.5] {$\U$};
	
	\draw (3, -1) rectangle ++(1, 1) node[pos=0.5] {$\U$};
	\draw (3, -2) rectangle ++(1, 1) node[pos=0.5] {$\A$};
	\draw (3, -3) rectangle ++(1, 1) node[pos=0.5] {$\A$};
	\draw (3, -4) rectangle ++(1, 1) node[pos=0.5] {$\U$};
	\draw (3, -5) rectangle ++(1, 1) node[pos=0.5] {$\A$};
	\draw[-latex] (1, -0.5) -- (3, -0.5) node[above, midway] {$y$};
	\draw[-latex] (1, -0.5) -- (3, -3.5) node[below left, midway] {$y\sqrt{\varepsilon}$};
\end{tikzpicture}
\hfill $d_{\U + \A} = $
\begin{tikzpicture}[scale=0.5, baseline={([yshift=-2.25ex]current bounding box.center)}]
	\draw (3, -1) rectangle ++(1, 1) node[pos=0.5] {$\U$};
	
	\draw (0, -1) rectangle ++(1, 1) node[pos=0.5] {$\U$};
	\draw (0, -2) rectangle ++(1, 1) node[pos=0.5] {$\A$};
	\draw (0, -3) rectangle ++(1, 1) node[pos=0.5] {$\A$};
	\draw (0, -4) rectangle ++(1, 1) node[pos=0.5] {$\U$};
	\draw (0, -5) rectangle ++(1, 1) node[pos=0.5] {$\A$};
	\draw[-latex] (1, -0.5) -- (3, -0.5) node[above, midway] {$\frac{1}{y}$};
	\draw[-latex] (1, -3.5) -- (3, -0.5) node[below right, midway] {$\frac{\sqrt{\varepsilon}}{y}$};
\end{tikzpicture}
\hfill \ \ 
\end{center}
\caption{\label{Figure:DualityAExample}Morphisms $b_{\A}$, $d_{\A}$, $b_{\U + \A}$ and $d_{\U + \A}$}
\end{figure}

Diagrams for the morphisms $b_{\U + \A}$ and $d_{\U + \A}$ shown in the figure \ref{Figure:DualityAExample}. In this case $(\U + \A)\otimes (\U + \A) = \U + \A + \A + \U + \A$, where the first symbol $\U$ is $(\U\otimes \U)_1$ and the second symbol $\U$ is $(\A\otimes \A)_1$. So in the diagram for $b_{\U + \A}$ the value $y$ is associated with the first arrow and the value $y\sqrt{\varepsilon}$ is associated with the second arrow. Similarly, in the diagram for $d_{\U + \A}$, the value $\frac{1}{y}$ is associated with the first arrow, and the value $\frac{\sqrt{\varepsilon}}{y}$ is associated with the second arrow.
\end{example}

\begin{lemma}
\label{Lemma:Duality}
For any object $X$ in the category $\E$:
\begin{enumerate}
\item $(b_X\otimes id_X)\circ \alpha_{X, X, X}\circ (id_X\otimes d_X) = id_X$;
\item $(id_X\otimes b_X)\circ \alpha_{X, X, X}^{-1}\circ (d_X\otimes id_X) = id_X$.
\end{enumerate}
\end{lemma}
\begin{proof}
Let $X = \sum\limits_{i = 1}^{n}x_i$, $x_i\in I$. We will prove the first statement of the lemma by using diagram language.

It follows from the definition of the morphism $b_{X}\in Hom(\U, X\otimes X)$ that the diagram of this morphism $b_X$ contains exactly $n$ non-zero arrows. These arrows connect the cell with the symbol $\U$ with the cells $(x_1\otimes x_1)_1, \ldots, (x_n\otimes x_n)_1$. Let $\varphi_i\in \{y, y\sqrt{\varepsilon}\}$ be the value associated with the $i$-th arrow (figure \ref{Figure:DualityProveB} on the left). Then the diagram of the morphism $b_X\otimes id_X \in Hom(X, (X\otimes X)\otimes X)$ contains exactly $n^2$ arrows. These arrows connect the cell $x_i$ with the cells $(x_1\otimes x_1)_1\otimes x_i, \ldots, (x_n\otimes x_n)_1\otimes x_i$, and the values associated with these arrows are $\varphi_1, \ldots, \varphi_n$ respectively (figure \ref{Figure:DualityProveB} on the right).

\begin{figure}[h]
\begin{center}
\ \hfill
\begin{tikzpicture}[scale=0.5, baseline={([yshift=0.0ex]current bounding box.center)}]
	\draw (0, -1) rectangle ++(1, 1) node[pos=0.5] {$\U$};
	
	\draw (3, -2) rectangle ++(1, 2) node[pos=0.5] {$\vdots$};
	\draw (3, -3) rectangle ++(1, 1);
	\draw (4, -2.5) node[right] {$(x_i\otimes x_i)_1$};
	\draw (3, -5) rectangle ++(1, 2) node[pos=0.5] {$\vdots$};
	\draw[-latex] (1, -0.5) -- (3, -2.5) node[below, midway] {$\varphi_i$};
\end{tikzpicture}
\hfill
\begin{tikzpicture}[scale=0.5, baseline={([yshift=0.0ex]current bounding box.center)}]
	\draw (0, -2) rectangle ++(1, 2) node[pos=0.5] {$\vdots$};
	\draw (0, -3) rectangle ++(1, 1) node[pos=0.5] {$x_i$};
	\draw (0, -5) rectangle ++(1, 2) node[pos=0.5] {$\vdots$};
	
	\draw (3, -2) rectangle ++(1, 2) node[pos=0.5] {$\vdots$};
	\draw (3, -3) rectangle ++(1, 1);
	\draw (4, -2.5) node[right] {$(x_1\otimes x_1)_1\otimes x_i$};
	\draw (3, -5) rectangle ++(1, 2) node[pos=0.5] {$\vdots$};
	\draw (3, -6) rectangle ++(1, 1);
	\draw (4, -5.5) node[right] {$(x_n\otimes x_n)_1\otimes x_i$};
	\draw (3, -8) rectangle ++(1, 2) node[pos=0.5] {$\vdots$};
	\draw[-latex] (1, -2.5) -- (3, -2.5) node[above, midway] {$\varphi_1$};
	\draw[-latex] (1, -2.5) -- (3, -5.5) node[below, midway] {$\varphi_n$};
\end{tikzpicture}
\hfill \ \ 
\end{center}
\caption{\label{Figure:DualityProveB}Diagrams of the morphism $b_X$ (on the left) and $b_X\otimes id_X$ (on the right)}
\end{figure}
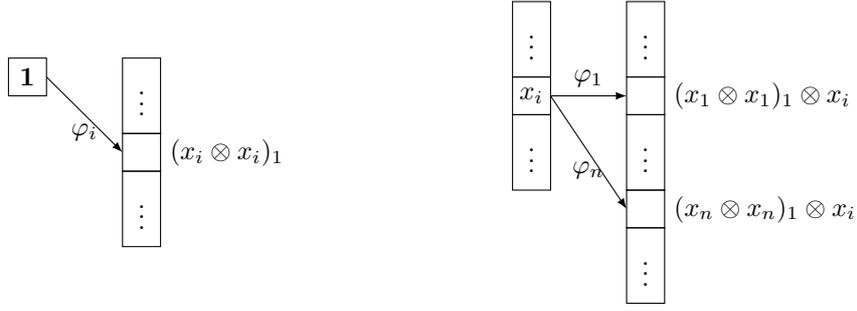

Similarly, the diagram of the morphism $d_X$ contains exactly $n$ non-zero arrows. The $i$-th arrow connects $(x_i\otimes x_i)_1$ with $\U$, and the associated value is $\psi_i\in \{\frac{1}{y}, \frac{\sqrt{\varepsilon}}{y}\}$ (figure \ref{Figure:DualityProveD} on the left). Then the digram of the morphism $id_X\otimes d_X \in Hom(X\otimes (X\otimes X), X)$ contains exactly $n^2$ arrows. The arrow with the associated value $\psi_j$ connects the cell $x_i\otimes (x_j\otimes x_j)_1$ with the cell $x_i$ for each $i\in\{1, \ldots, n\}$ (figure \ref{Figure:DualityProveD} on the right).

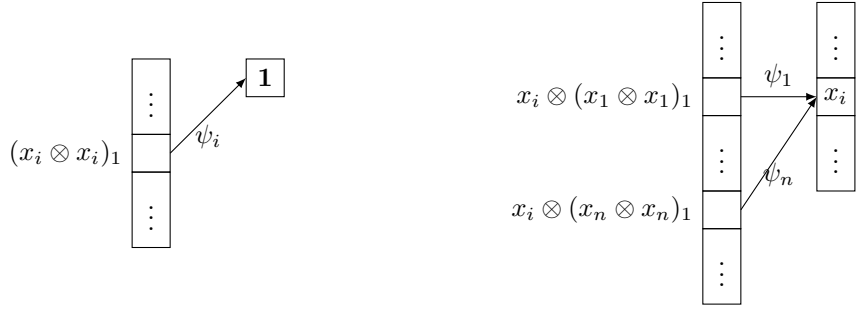
\begin{figure}[h]
\begin{center}
\ \hfill
\begin{tikzpicture}[scale=0.5, baseline={([yshift=0.0ex]current bounding box.center)}]
	\draw (0, -2) rectangle ++(1, 2) node[pos=0.5] {$\vdots$};
	\draw (0, -3) rectangle ++(1, 1);
	\draw (0, -2.5) node[left] {$(x_i\otimes x_i)_1$};
	\draw (0, -5) rectangle ++(1, 2) node[pos=0.5] {$\vdots$};
	
	\draw (3, -1) rectangle ++(1, 1) node[pos=0.5] {$\U$};

	\draw[-latex] (1, -2.5) -- (3, -0.5) node[below, midway] {$\psi_i$};
\end{tikzpicture}
\hfill
\begin{tikzpicture}[scale=0.5, baseline={([yshift=0.0ex]current bounding box.center)}]	
	\draw (0, -2) rectangle ++(1, 2) node[pos=0.5] {$\vdots$};
	\draw (0, -3) rectangle ++(1, 1);
	\draw (0, -2.5) node[left] {$x_i\otimes (x_1\otimes x_1)_1$};
	\draw (0, -5) rectangle ++(1, 2) node[pos=0.5] {$\vdots$};
	\draw (0, -6) rectangle ++(1, 1);
	\draw (0, -5.5) node[left] {$x_i\otimes (x_n\otimes x_n)_1$};
	\draw (0, -8) rectangle ++(1, 2) node[pos=0.5] {$\vdots$};
	
	\draw (3, -2) rectangle ++(1, 2) node[pos=0.5] {$\vdots$};
	\draw (3, -3) rectangle ++(1, 1) node[pos=0.5] {$x_i$};
	\draw (3, -5) rectangle ++(1, 2) node[pos=0.5] {$\vdots$};
	
	\draw[-latex] (1, -2.5) -- (3, -2.5) node[above, midway] {$\psi_1$};
	\draw[-latex] (1, -5.5) -- (3, -2.5) node[below, midway] {$\psi_n$};
\end{tikzpicture}
\hfill \ \ 
\end{center}
\caption{\label{Figure:DualityProveD}Diagrams of the morphism $d_X$ (on the left) and $id_X\otimes d_X$ (on the right)}
\end{figure}

Finally, note that the associativity isomorphism $\alpha_{X, X, X}$ contains non-zero arrows connected by $(x_j\otimes x_j)_1\otimes x_i$ and $x_p\otimes (x_q\otimes x_q)_1$ if and only if $i = j = p = q$. If $x_i = x_j = x_p = x_q = \U$, then the value associated with this arrow is $1$, and if $x_i = x_j = x_p = x_q = \A$, then this value is $\frac{1}{\varepsilon}$. Thus, in the diagram of the composition $(b_X\otimes id_X)\circ \alpha_{X, X, X}\circ (id_X\otimes d_X)$, the left cell $x_i$ is connected to only the right cell $x_i$, and the value associated with this arrow is equal to one. This morphism coincides with $id_X$.

The proof of the second statement of the lemma is similar.
\end{proof}

\begin{lemma}
\label{Lemma:DualityIdentities}
Let $X$ be an object of the category $\E$. Then
\begin{enumerate}
\item $b_X\circ (\theta_{X}\otimes id_X)\circ c_{X, X} = b_X$;
\item $(\theta_X\otimes id_X)\circ c_{X, X}\circ d_X = d_X$;
\item $b_X\circ (\theta_X\otimes id_X) = b_X\circ (id_X\otimes \theta_X)$.
\end{enumerate}
\end{lemma}
\begin{proof}
The morphism $b_{X}$ defined only by the matrix $[b_X]_{\U}$, because the matrix $[b_{X}]_{\A}$ has zero columns. Hence the composition $b_X\circ f$ with any suitable morphism $f$ defined only by $[b_X]_{\U}$ and $[f]_{\U}$. The first two statements of the lemma follow from the second statement of the lemma \ref{Lemma:TwistUnitMatrix}. The third statement of the lemma follows from the first statement of the same lemma \ref{Lemma:TwistUnitMatrix}.
\end{proof}

\begin{theorem}
\label{Theorem:RibbonCategory}
The category $\E$ is ribbon.
\end{theorem}
\begin{proof}
The family of isomorphisms $\theta_{X}$ form a twist on the braided category $\E$. The naturality of $\theta_{X}$ follows from the lemma \ref{Lemma:TwistNaturality}, compatibility with braiding have been proved in the lemma \ref{Lemma:TwistAndBraiding}.

Family of morphisms $b_X, d_X$ define the duality on the braided category $\E$. Their compatibility follows from lemmas \ref{Lemma:Duality} and \ref{Lemma:DualityIdentities} (third statement).
\end{proof}

\begin{theorem}
\label{Theorem:Modular}
The category $\E$ is modular.
\end{theorem}
\begin{proof}
It's clear that the category $\E$ is an Ab-category. Indeed, for any two objects $X, Y$ the set $Hom(X, Y)$ admits a natural structure of an abelian group such that the composition of morphisms is bilinear.

Every object in $\E$ is dominated by simple objects $\U, \A$. It's also obvious, because every object is a sum of these simple ones. So for the modularity of $\E$ we should only check that the matrix $$S = \begin{pmatrix}
S_{\U, \U} & S_{\U, \A} \\
S_{\A, \U} & S_{\A, \A}
\end{pmatrix}
$$
is invertible over $\mathbb{C}$. Here the value $S_{X, Y}$ for any two objects $X, Y$ of the category $\E$ defined by the composition $$S_{X, Y} = b_{X\otimes Y}\circ ((c_{X, Y}\circ c_{Y, X})\otimes id_{X\otimes Y})\circ d_{X\otimes Y}.$$

Note that by using the lemma \ref{Lemma:DualityIdentities} (the first and second statements) this formula is simplified from the general ones.

It's clear that $S_{\U, \U} = b_{\U}\circ ((c_{\U, \U}\circ c_{\U, \U})\otimes id_\U)\circ d_{\U} = 1$.

If $X, Y\in I$ and $X \neq Y$, then $b_{X\otimes Y} = b_{\A}$, $c_{X, Y} = c_{Y, X} = id_{\A}$, $id_{X\otimes Y} = id_{\A}$ and $d_{X\otimes Y} = d_{\A}$. So, $S_{\U, \A} = S_{\A, \U} = b_{\A}\circ d_{\A} = \varepsilon$.

Finally, consider the case $X = Y = \A$. The diagram of the composition $b_{\U + \A}\circ ((c_{\A, \A}\circ c_{\A, \A})\otimes id_{\U + \A})\circ d_{\U + \A}$ is shown in the figure \ref{Figure:SAADiagram}. So, $$S_{\A, \A} = \beta_{\varepsilon}^4 + \varepsilon\cdot \beta_{\varepsilon}^2 = \beta_{\varepsilon} + \varepsilon - 1 + \varepsilon\cdot (-1 - \frac{\beta_{\varepsilon}}{\varepsilon}) = -1.$$

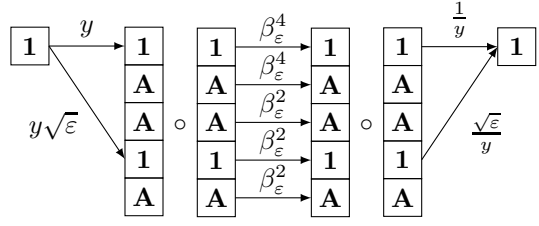
\begin{figure}[h]
\begin{center}
\begin{tikzpicture}[scale=0.5, baseline={([yshift=-1.5ex]current bounding box.center)}]
	\draw (0, -1) rectangle ++(1, 1) node[pos=0.5] {$\U$};
	
	\draw (3, -1) rectangle ++(1, 1) node[pos=0.5] {$\U$};
	\draw (3, -2) rectangle ++(1, 1) node[pos=0.5] {$\A$};
	\draw (3, -3) rectangle ++(1, 1) node[pos=0.5] {$\A$};
	\draw (3, -4) rectangle ++(1, 1) node[pos=0.5] {$\U$};
	\draw (3, -5) rectangle ++(1, 1) node[pos=0.5] {$\A$};
	\draw[-latex] (1, -0.5) -- (3, -0.5) node[above, midway] {$y$};
	\draw[-latex] (1, -0.5) -- (3, -3.5) node[below left, midway] {$y\sqrt{\varepsilon}$};
\end{tikzpicture}
$\circ$
\begin{tikzpicture}[scale=0.5, baseline={([yshift=-1.7ex]current bounding box.center)}]
	\draw (0, -1) rectangle ++(1, 1) node[pos=0.5] {$\U$};
	\draw (0, -2) rectangle ++(1, 1) node[pos=0.5] {$\A$};
	\draw (0, -3) rectangle ++(1, 1) node[pos=0.5] {$\A$};
	\draw (0, -4) rectangle ++(1, 1) node[pos=0.5] {$\U$};
	\draw (0, -5) rectangle ++(1, 1) node[pos=0.5] {$\A$};
	
	\draw (3, -1) rectangle ++(1, 1) node[pos=0.5] {$\U$};
	\draw (3, -2) rectangle ++(1, 1) node[pos=0.5] {$\A$};
	\draw (3, -3) rectangle ++(1, 1) node[pos=0.5] {$\A$};
	\draw (3, -4) rectangle ++(1, 1) node[pos=0.5] {$\U$};
	\draw (3, -5) rectangle ++(1, 1) node[pos=0.5] {$\A$};
	\draw[-latex] (1, -0.5) -- (3, -0.5) node[above=-2.0pt, midway] {$\beta_{\varepsilon}^4$};
	\draw[-latex] (1, -1.5) -- (3, -1.5) node[above=-2.0pt, midway] {$\beta_{\varepsilon}^4$};
	\draw[-latex] (1, -2.5) -- (3, -2.5) node[above=-2.0pt, midway] {$\beta_{\varepsilon}^2$};
	\draw[-latex] (1, -3.5) -- (3, -3.5) node[above=-2.0pt, midway] {$\beta_{\varepsilon}^2$};
	\draw[-latex] (1, -4.5) -- (3, -4.5) node[above=-2.0pt, midway] {$\beta_{\varepsilon}^2$};
\end{tikzpicture}
$\circ$
\begin{tikzpicture}[scale=0.5, baseline={([yshift=-2.25ex]current bounding box.center)}]
	\draw (3, -1) rectangle ++(1, 1) node[pos=0.5] {$\U$};
	
	\draw (0, -1) rectangle ++(1, 1) node[pos=0.5] {$\U$};
	\draw (0, -2) rectangle ++(1, 1) node[pos=0.5] {$\A$};
	\draw (0, -3) rectangle ++(1, 1) node[pos=0.5] {$\A$};
	\draw (0, -4) rectangle ++(1, 1) node[pos=0.5] {$\U$};
	\draw (0, -5) rectangle ++(1, 1) node[pos=0.5] {$\A$};
	\draw[-latex] (1, -0.5) -- (3, -0.5) node[above, midway] {$\frac{1}{y}$};
	\draw[-latex] (1, -3.5) -- (3, -0.5) node[below right, midway] {$\frac{\sqrt{\varepsilon}}{y}$};
\end{tikzpicture}
\end{center}
\caption{\label{Figure:SAADiagram}The diagram of the composition $b_{\U + \A}\circ ((c_{\A, \A}\circ c_{\A, \A})\otimes id_{\U + \A})\circ d_{\U + \A}$}
\end{figure}

As a result $$S = \begin{pmatrix}
1 & \varepsilon \\
\varepsilon & -1
\end{pmatrix}, \det S = -2 - \varepsilon \neq 0.
$$
\end{proof}

\section{Invariant $tr_{\varepsilon}$}

Let $L = l_1\cup\ldots\cup l_k$ be a digram of the $k$-component link. The colouring of $L$ is a map $\xi\colon \{l_1, \ldots, l_k\}\to \{\U, \A\}$. Define the morphism $\{L\}_{\xi}\in Hom(\U, \U)$ as follows. First, remove all diagram components with the colour $\U$. We will get a sub-link $L'$ where all components are coloured by $\A$. Split the diagram $L'$ by vertical lines into layers, so that each layer contains any number of parallel horizontal arcs of the diagram and one of the segments shown in the figure \ref{Figure:DiagramParts}. Let $\lambda$ be one of the layers of our split. Match the positive crossing within the layer $\lambda$ with the morphism $c_{\A, \A}$, the negative crossing with the morphism $c_{\A, \A}^{-1}$, the left half-circle with the morphism $b_{\A}$, right half-circle with the morphism $d_{\A}$, positive loop with the morphism $\theta_{\A}$, negative loop with the morphism $\theta_{\A}^{-1}$ and trivial horizontal arc with the identity morphism $id_{\A}$ (figure \ref{Figure:DiagramParts}). Stacking the digram segments from top to bottom is a tensor product of the corresponding morphisms. Fix an arbitrary order of these multiplications. If the left side of the layer $\lambda$ intersects the diagram $L'$ at $n$ points, and the right side of the layer intersects the diagram at $m$ points, then this layer defines the morphism from $\A^{\otimes n}$ to $\A^{\otimes m}$ (with an arbitrary order of multiplications) (see figure \ref{Figure:LayerExample} for two examples). Denote $\mu_{\lambda}$ the morphism in category $\E$ obtained by this procedure for the layer $\lambda$.

\begin{figure}[h]
\begin{center}
\ \hfill
\begin{tikzpicture}[scale=0.5]
\begin{knot}[
	consider self intersections=true,
	% draft mode=crossings,
	clip width=7,
	% flip crossing=2
]
\strand (-1, 1) -- (1, -1);
\strand (-1, -1) -- (1, 1);
\end{knot}
\filldraw (-1, -1) circle (0.05);
\filldraw (-1, 1) circle (0.05);
\filldraw (1, -1) circle (0.05);
\filldraw (1, 1) circle (0.05);
\end{tikzpicture}
\hfill
\begin{tikzpicture}[scale=0.5]
\begin{knot}[
	consider self intersections=true,
	% draft mode=crossings,
	clip width=7,
	% flip crossing=2
]
\strand (-1, -1) -- (1, 1);
\strand (-1, 1) -- (1, -1);
\end{knot}
\filldraw (-1, -1) circle (0.05);
\filldraw (-1, 1) circle (0.05);
\filldraw (1, -1) circle (0.05);
\filldraw (1, 1) circle (0.05);
\end{tikzpicture}
\hfill
\begin{tikzpicture}[scale=0.5]
\draw[knot_diagram] (0, 1) arc (90:270:1);
\filldraw (0, -1) circle (0.05);
\filldraw (0, 1) circle (0.05);
\end{tikzpicture}
\hfill
\begin{tikzpicture}[scale=0.5]
\draw[knot_diagram] (0, -1) arc (-90:90:1);
\filldraw (0, -1) circle (0.05);
\filldraw (0, 1) circle (0.05);
\end{tikzpicture}
\hfill
\begin{tikzpicture}[scale=0.75, use Hobby shortcut, baseline={([yshift=-3.5ex]current bounding box.center)}]
\filldraw (-1, 0) circle (0.03);
\filldraw (1, 0) circle (0.03);
\begin{knot}[
  consider self intersections=true,
  ignore endpoint intersections=false,
]
\strand (-1, 0) .. (-0.5, 0.1) .. (0.25, -0.25) .. (0, -0.5) .. (-0.25, -0.25) .. (0.5, -0.1) .. (1,0);
\end{knot}
\end{tikzpicture}
\hfill
\begin{tikzpicture}[scale=0.75, use Hobby shortcut, baseline={([yshift=-3.5ex]current bounding box.center)}]
\filldraw (-1, 0) circle (0.03);
\filldraw (1, 0) circle (0.03);
\begin{knot}[
  consider self intersections=true,
  ignore endpoint intersections=false,
  flip crossing=1
]
\strand (-1, 0) .. (-0.5, 0.1) .. (0.25, -0.25) .. (0, -0.5) .. (-0.25, -0.25) .. (0.5, -0.1) .. (1,0);
\end{knot}
\end{tikzpicture}
\hfill
\begin{tikzpicture}[scale=0.5, baseline={([yshift=-3.5ex]current bounding box.center)}]
\draw[knot_diagram] (-1, 0) -- (1, 0);
\filldraw (-1, 0) circle (0.05);
\filldraw (1, 0) circle (0.05);
\end{tikzpicture}
\hfill \ \ 
\end{center}
\caption{\label{Figure:DiagramParts}From left to right: positive crossing, negative crossing, left half-circle, right half-circle, positive loop, negative loop, trivial arc}
\end{figure}
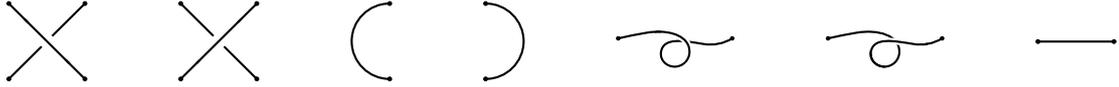

If the diagram $L'$ splits into layers $\lambda_1, \ldots, \lambda_n$ (reading from left to right), then the morphism $\{L\}_{\xi}$ is a composition $$\mu_{\lambda_1}\circ \tau_{1} \circ \mu_{\lambda_2}\circ \tau_2\circ \ldots\circ \tau_{n - 1}\circ \mu_{\lambda_n},$$ where $\tau_i$ is a combination of associativity isomorphisms from the right side of the layer $\lambda_i$ to the left side of the layer $\lambda_{i + 1}$ (these sides differ by the order of multiplication).

\begin{figure}[h]
\begin{center}
\ \hfill
\begin{tikzpicture}[scale=0.5]
	\begin{knot}[
		consider self intersections=true,
		% draft mode=crossings,
		clip width=7,
		flip crossing=1
		]
		\strand (0, 1) -- (4, 2);
		\strand (0, 2) -- (4, 1);
	\end{knot}
	\draw[knot_diagram] (0, 0) -- (4, 0);
	\draw[knot_diagram] (0, 3) -- (4, 3);
	\draw[split_line] (0, -1) -- (0, 4);
	\draw[split_line] (4, -1) -- (4, 4);
	
	\filldraw (0, 0) circle (0.05);
	\filldraw (0, 1) circle (0.05);
	\filldraw (0, 2) circle (0.05);
	\filldraw (0, 3) circle (0.05);
	
	\filldraw (4, 0) circle (0.05);
	\filldraw (4, 1) circle (0.05);
	\filldraw (4, 2) circle (0.05);
	\filldraw (4, 3) circle (0.05);
	
	\draw (0, 4.5) node {$(\A\otimes (\A\otimes \A))\otimes \A$};
	\draw (4, -1.5) node {$(\A\otimes (\A\otimes \A))\otimes \A$};
\end{tikzpicture}
\hfill
\begin{tikzpicture}[scale=0.5]
	\draw[knot_diagram] (0, 0) -- (4, 0);
	\draw[knot_diagram] (0, 3) -- (4, 3);
	\draw[knot_diagram] (0, 1) arc (-90:90:0.5);
	\draw[split_line] (0, -1) -- (0, 4);
	\draw[split_line] (4, -1) -- (4, 4);
	
	\filldraw (0, 0) circle (0.05);
	\filldraw (0, 1) circle (0.05);
	\filldraw (0, 2) circle (0.05);
	\filldraw (0, 3) circle (0.05);
	
	\filldraw (4, 0) circle (0.05);
	\filldraw (4, 3) circle (0.05);
	
	\draw (0, 4.5) node {$\A\otimes ((\A\otimes \A)\otimes \A)$};
	\draw (4, -1.5) node {$\A\otimes \A$};
\end{tikzpicture}
\hfill \ \ 
\end{center}
\caption{\label{Figure:LayerExample}Morphism $(id_{\A}\otimes c_{\A, \A})\otimes id_{\A}$ (on the left) and $id_{\A}\otimes (d_{\A}\otimes id_{\A})$ (on the right)}
\end{figure}
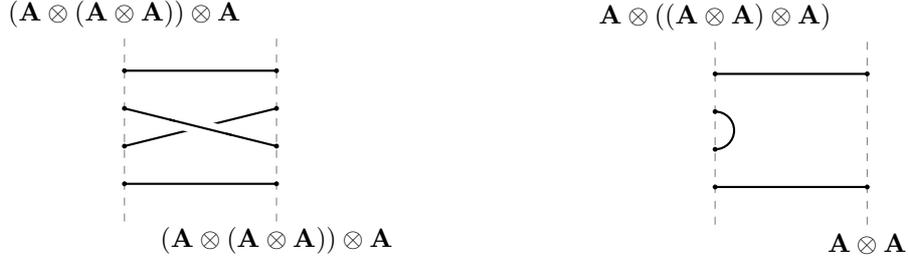

The diagram $L$ with the colour $\xi$ can be understood as a diagram of the morphism $\{L\}_{\xi}$. To distinguish it from other already used diagrams of morphisms (with cells and arrows), we will call the diagram $L$ a knotted diagram of the morphism $\{L\}_{\xi}$.

The morphism $\{L\}_{\xi}$ is a morphism from $\U$ to $\U$, so it is defined by a complex number. We can identify this morphism with this number and assume that $\{L\}_{\xi}\in \mathbb{C}$.

\begin{remark}
\label{Remark:LMorphismComposition}
If $\{L\}_{\xi} = \mu_1\circ \mu_2\circ \ldots\circ \mu_{n - 1}\circ \mu_n$, then $$\{L\}_{\xi} = [\mu_n]_{\U}\cdot [\mu_{n - 1}]_{\U}\cdot \ldots\cdot [\mu_2]_{\U}\cdot [\mu_1]_{\U}.$$
\end{remark}

\begin{example}
\label{Example:Trefoil}
Let $L$ be a knot diagram as shown in the figure \ref{Figure:Trefoil}, and let $\xi$ be a colouring that maps the unique component of $L$ to $\A$. The diagram $L$ is divided into eleven levels $\lambda_1, \ldots, \lambda_{11}$. The morphisms corresponding to these levels are the following:
\begin{center}
\begin{longtable}{l}
$\mu_{\lambda_1} = b_{\A}\in Hom (\U, \A\otimes \A)$, \\
$\mu_{\lambda_2} = (id_{\A}\otimes id_{\A})\otimes b_{\A} \in Hom(\A\otimes \A, (\A\otimes \A)\otimes (\A\otimes \A))$ \\
$\mu_{\lambda_3} = \alpha_{\A, \A, \A\otimes \A}\in Hom ((\A\otimes \A)\otimes (\A\otimes \A), \A\otimes (\A\otimes (\A\otimes \A)))$ \\
$\mu_{\lambda_4} = id_{\A}\otimes \alpha_{\A, \A, \A}^{-1} \in Hom (\A\otimes (\A\otimes (\A\otimes \A)), \A\otimes ((\A\otimes \A)\otimes \A))$ \\
$\mu_{\lambda_5} = \mu_{\lambda_6} = \mu_{\lambda_7} = id_{\A}\otimes (c_{\A, \A}\otimes id_{\A})\in Hom(\A\otimes ((\A\otimes \A)\otimes \A), \A\otimes ((\A\otimes \A)\otimes \A))$ \\
$\mu_{\lambda_8} = id_{\A}\otimes \alpha_{\A, \A, \A}\in Hom(\A\otimes ((\A\otimes \A)\otimes \A), \A\otimes (\A\otimes (\A\otimes \A)))$ \\
$\mu_{\lambda_9} = \alpha_{\A, \A, \A\otimes \A}^{-1} \in Hom(\A\otimes (\A\otimes (\A\otimes \A)), (\A\otimes \A)\otimes (\A\otimes \A))$ \\
$\mu_{\lambda_{10}} = d_{\A}\otimes (id_{\A}\otimes id_{\A})\in Hom((\A\otimes \A)\otimes (\A\otimes \A), \A\otimes \A)$ \\
$\mu_{\lambda_{11}} = d_{\A}\in Hom (\A\otimes \A, \U)$.
\end{longtable}
\end{center}

\begin{figure}[h]
\begin{center}
\begin{tikzpicture}[scale=0.75, use Hobby shortcut]
\begin{knot}[
	clip width=5,
	consider self intersections=true,
	ignore endpoint intersections=false,
	flip crossing/.list={1, 3}
]
\strand ([closed]1, 2) .. (2, 2) .. (3, 2) .. (4, 2) .. (5, 1) .. (6, 2) .. (7, 1) .. (8, 1) .. (9, 1) .. (10, 1) .. (10.5, 0.5) ..(10, 0) .. (9, 0) .. (8, 0) .. (7, 0) .. (6, 0) .. (5, 0) .. (4, 0) .. (3, 0) .. (2, 0) .. (1.5, 0.5) .. (2, 1) .. (3, 1) .. (4, 1) .. (5, 2) .. (6, 1) .. (7, 2) .. (8, 2) .. (9, 2) .. (9.5, 2.5) .. (9, 3) .. (8, 3) .. (7, 3) .. (6, 3) .. (5, 3) .. (4, 3) .. (3, 3) .. (2, 3) .. (1, 3) ..(0.5, 2.5);
\end{knot}

\foreach \i in {0, 1, 2, 3, 4, 5, 6, 7, 8, 9, 10, 11}{
    \draw[split_line] (\i, -1) -- (\i, 4);
}

\foreach \i in {2, 3, 4, 5, 6, 7, 8, 9}{
	\filldraw (\i, 0) circle (0.03);
	\filldraw (\i, 1) circle (0.03);
	\filldraw (\i, 2) circle (0.03);
	\filldraw (\i, 3) circle (0.03);
}
\filldraw (1, 2) circle (0.03);
\filldraw (1, 3) circle (0.03);
\filldraw (10, 0) circle (0.03);
\filldraw (10, 1) circle (0.03);

\draw (0.5, -0.75) node {$\lambda_1$};
\draw (1.5, -0.75) node {$\lambda_2$};
\draw (2.5, -0.75) node {$\lambda_3$};
\draw (3.5, -0.75) node {$\lambda_4$};
\draw (4.5, -0.75) node {$\lambda_5$};
\draw (5.5, -0.75) node {$\lambda_6$};
\draw (6.5, -0.75) node {$\lambda_7$};
\draw (7.5, -0.75) node {$\lambda_8$};
\draw (8.5, -0.75) node {$\lambda_9$};
\draw (9.5, -0.75) node {$\lambda_{10}$};
\draw (10.5, -0.75) node {$\lambda_{11}$};
\end{tikzpicture}
\end{center}
\caption{\label{Figure:Trefoil}Trefoil knot and splits into eleven layers $\lambda_1, \ldots, \lambda_{11}$}
\end{figure}
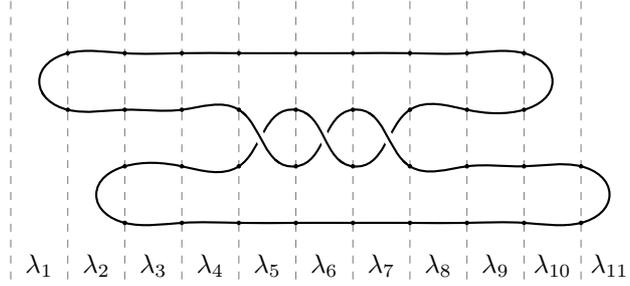

In the figure \ref{Figure:TrafoilMorphisms} the diagram of the composition $\mu_{\lambda_1}\circ\ldots\circ \mu_{\lambda_{11}}$ is shown.

\begin{figure}[h]
\begin{center}
\begin{tikzpicture}[scale=0.5]
	\draw (0, -1) rectangle ++(1, 1) node[pos=0.5] {$\U$};
	
	\draw (3, -1) rectangle ++(1, 1) node[pos=0.5] {$\U$};
	\draw (3, -2) rectangle ++(1, 1) node[pos=0.5] {$\A$};
	\draw[-latex] (1, -0.5) -- (3, -0.5) node[above, midway] {$y\sqrt{\varepsilon}$};
	
	\draw (6, -1) rectangle ++(1, 1) node[pos=0.5] {$\U$};
	\draw (6, -2) rectangle ++(1, 1) node[pos=0.5] {$\A$};
	\draw (6, -3) rectangle ++(1, 1) node[pos=0.5] {$\A$};
	\draw (6, -4) rectangle ++(1, 1) node[pos=0.5] {$\U$};
	\draw (6, -5) rectangle ++(1, 1) node[pos=0.5] {$\A$};
	\draw[-latex] (4, -0.5) -- (6, -0.5) node[above, midway] {$y\sqrt{\varepsilon}$};	
	\draw[-latex] (4, -1.5) -- (6, -2.5) node[below, midway] {$y\sqrt{\varepsilon}$};
	
	\draw (9, -1) rectangle ++(1, 1) node[pos=0.5] {$\U$};
	\draw (9, -2) rectangle ++(1, 1) node[pos=0.5] {$\A$};
	\draw (9, -3) rectangle ++(1, 1) node[pos=0.5] {$\A$};
	\draw (9, -4) rectangle ++(1, 1) node[pos=0.5] {$\U$};
	\draw (9, -5) rectangle ++(1, 1) node[pos=0.5] {$\A$};
	\draw[ma_black] (7, -0.5) -- (9, -0.5);
	\draw[ma_black] (7, -2.5) -- (9, -1.5);
	\draw[ma_black] (7, -3.5) -- (9, -3.5);
	\draw[ma_blue] (7, -1.5) -- (9, -2.5);
	\draw[ma_green] (7, -1.5) -- (9, -4.5);
	\draw[ma_red] (7, -4.5) -- (9, -2.5);
	\draw[ma_orange] (7, -4.5) -- (9, -4.5);
	
	\draw (12, -1) rectangle ++(1, 1) node[pos=0.5] {$\U$};
	\draw (12, -2) rectangle ++(1, 1) node[pos=0.5] {$\A$};
	\draw (12, -3) rectangle ++(1, 1) node[pos=0.5] {$\A$};
	\draw (12, -4) rectangle ++(1, 1) node[pos=0.5] {$\U$};
	\draw (12, -5) rectangle ++(1, 1) node[pos=0.5] {$\A$};
	\draw[ma_blue] (10, -0.5) -- (12, -0.5);
	\draw[ma_blue] (10, -1.5) -- (12, -1.5);
	\draw[ma_green] (10, -0.5) -- (12, -3.5);
	\draw[ma_green] (10, -1.5) -- (12, -4.5);
	\draw[ma_black] (10, -2.5) -- (12, -2.5);
	\draw[ma_red] (10, -3.5) -- (12, -0.5);
	\draw[ma_red] (10, -4.5) -- (12, -1.5);
	\draw[ma_orange] (10, -3.5) -- (12, -3.5);
	\draw[ma_orange] (10, -4.5) -- (12, -4.5);
	
	\draw (15, -1) rectangle ++(1, 1) node[pos=0.5] {$\U$};
	\draw (15, -2) rectangle ++(1, 1) node[pos=0.5] {$\A$};
	\draw (15, -3) rectangle ++(1, 1) node[pos=0.5] {$\A$};
	\draw (15, -4) rectangle ++(1, 1) node[pos=0.5] {$\U$};
	\draw (15, -5) rectangle ++(1, 1) node[pos=0.5] {$\A$};
	\draw[-latex] (13, -0.5) -- (15, -0.5) node[above=-2.0pt, midway] {$\beta_{\varepsilon}^2$};
	\draw[-latex] (13, -1.5) -- (15, -1.5) node[above=-2.0pt, midway] {$\beta_{\varepsilon}^2$};
	\draw[-latex] (13, -2.5) -- (15, -2.5) node[above=-2.0pt, midway] {$\beta_{\varepsilon}$};
	\draw[-latex] (13, -3.5) -- (15, -3.5) node[above=-2.0pt, midway] {$\beta_{\varepsilon}$};
	\draw[-latex] (13, -4.5) -- (15, -4.5) node[above=-2.0pt, midway] {$\beta_{\varepsilon}$};
	
	\draw (18, -1) rectangle ++(1, 1) node[pos=0.5] {$\U$};
	\draw (18, -2) rectangle ++(1, 1) node[pos=0.5] {$\A$};
	\draw (18, -3) rectangle ++(1, 1) node[pos=0.5] {$\A$};
	\draw (18, -4) rectangle ++(1, 1) node[pos=0.5] {$\U$};
	\draw (18, -5) rectangle ++(1, 1) node[pos=0.5] {$\A$};
	\draw[-latex] (16, -0.5) -- (18, -0.5) node[above=-2.0pt, midway] {$\beta_{\varepsilon}^2$};
	\draw[-latex] (16, -1.5) -- (18, -1.5) node[above=-2.0pt, midway] {$\beta_{\varepsilon}^2$};
	\draw[-latex] (16, -2.5) -- (18, -2.5) node[above=-2.0pt, midway] {$\beta_{\varepsilon}$};
	\draw[-latex] (16, -3.5) -- (18, -3.5) node[above=-2.0pt, midway] {$\beta_{\varepsilon}$};
	\draw[-latex] (16, -4.5) -- (18, -4.5) node[above=-2.0pt, midway] {$\beta_{\varepsilon}$};
	
	\draw (21, -1) rectangle ++(1, 1) node[pos=0.5] {$\U$};
	\draw (21, -2) rectangle ++(1, 1) node[pos=0.5] {$\A$};
	\draw (21, -3) rectangle ++(1, 1) node[pos=0.5] {$\A$};
	\draw (21, -4) rectangle ++(1, 1) node[pos=0.5] {$\U$};
	\draw (21, -5) rectangle ++(1, 1) node[pos=0.5] {$\A$};
	\draw[-latex] (19, -0.5) -- (21, -0.5) node[above=-2.0pt, midway] {$\beta_{\varepsilon}^2$};
	\draw[-latex] (19, -1.5) -- (21, -1.5) node[above=-2.0pt, midway] {$\beta_{\varepsilon}^2$};
	\draw[-latex] (19, -2.5) -- (21, -2.5) node[above=-2.0pt, midway] {$\beta_{\varepsilon}$};
	\draw[-latex] (19, -3.5) -- (21, -3.5) node[above=-2.0pt, midway] {$\beta_{\varepsilon}$};
	\draw[-latex] (19, -4.5) -- (21, -4.5) node[above=-2.0pt, midway] {$\beta_{\varepsilon}$};
	
	\draw (24, -1) rectangle ++(1, 1) node[pos=0.5] {$\U$};
	\draw (24, -2) rectangle ++(1, 1) node[pos=0.5] {$\A$};
	\draw (24, -3) rectangle ++(1, 1) node[pos=0.5] {$\A$};
	\draw (24, -4) rectangle ++(1, 1) node[pos=0.5] {$\U$};
	\draw (24, -5) rectangle ++(1, 1) node[pos=0.5] {$\A$};
	\draw[ma_blue] (22, -0.5) -- (24, -0.5);
	\draw[ma_blue] (22, -1.5) -- (24, -1.5);
	\draw[ma_green] (22, -0.5) -- (24, -3.5);
	\draw[ma_green] (22, -1.5) -- (24, -4.5);
	\draw[ma_black] (22, -2.5) -- (24, -2.5);
	\draw[ma_red] (22, -3.5) -- (24, -0.5);
	\draw[ma_red] (22, -4.5) -- (24, -1.5);
	\draw[ma_orange] (22, -3.5) -- (24, -3.5);
	\draw[ma_orange] (22, -4.5) -- (24, -4.5);
	
	\draw (27, -1) rectangle ++(1, 1) node[pos=0.5] {$\U$};
	\draw (27, -2) rectangle ++(1, 1) node[pos=0.5] {$\A$};
	\draw (27, -3) rectangle ++(1, 1) node[pos=0.5] {$\A$};
	\draw (27, -4) rectangle ++(1, 1) node[pos=0.5] {$\U$};
	\draw (27, -5) rectangle ++(1, 1) node[pos=0.5] {$\A$};
	\draw[ma_black] (25, -0.5) -- (27, -0.5);
	\draw[ma_black] (25, -1.5) -- (27, -2.5);
	\draw[ma_black] (25, -3.5) -- (27, -3.5);
	\draw[ma_blue] (25, -2.5) -- (27, -1.5);
	\draw[ma_red] (25, -4.5) -- (27, -1.5);
	\draw[ma_green] (25, -2.5) -- (27, -4.5);
	\draw[ma_orange] (25, -4.5) -- (27, -4.5);
	
	\draw (30, -1) rectangle ++(1, 1) node[pos=0.5] {$\U$};
	\draw (30, -2) rectangle ++(1, 1) node[pos=0.5] {$\A$};
	\draw[-latex] (28, -0.5) -- (30, -0.5) node[above, midway] {$\frac{\sqrt{\varepsilon}}{y}$};	
	\draw[-latex] (28, -1.5) -- (30, -1.5) node[below, midway] {$\frac{\sqrt{\varepsilon}}{y}$};
	
	\draw (33, -1) rectangle ++(1, 1) node[pos=0.5] {$\U$};
	\draw[-latex] (31, -0.5) -- (33, -0.5) node[above, midway] {$\frac{\sqrt{\varepsilon}}{y}$};
\end{tikzpicture}
\end{center}
\caption{\label{Figure:TrafoilMorphisms}The diagram of the morphism $\mu_{\lambda_1}\circ\ldots\circ \mu_{\lambda_1}$}
\end{figure}
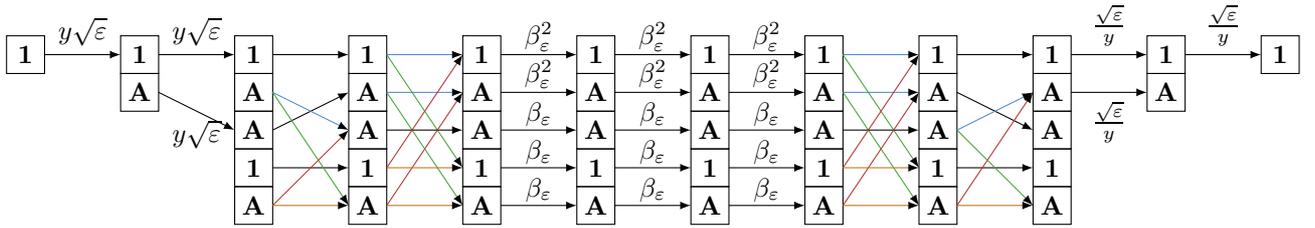

As a result $$\{L\}_{\xi} = y\sqrt{\varepsilon}\cdot y\sqrt{\varepsilon}\cdot 1\cdot \left(\frac{1}{\varepsilon}\cdot \beta_{\varepsilon}^6\cdot \frac{1}{\varepsilon} + \frac{1}{x\sqrt{\varepsilon}}\cdot \beta_{\varepsilon}^3\cdot \frac{x}{\sqrt{\varepsilon}}\right) \cdot 1\cdot \frac{\sqrt{\varepsilon}}{y}\cdot \frac{\sqrt{\varepsilon}}{y} = \beta_{\varepsilon}^6 + \varepsilon\cdot \beta_{\varepsilon}^3 = -\beta_{\varepsilon} + 1 - \beta_{\varepsilon} = 1 - 2\beta_{\varepsilon}.$$
\end{example}

\begin{proposition}
\label{Proposition:NoY}
Let $L$ be a link diagram and $\xi$ a colouring of this diagram. Then the value $\{L\}_{\xi}$ does not depend on the choice of $y$ in the morphisms $b_{\A}$ and $d_{\A}$.
\end{proposition}
\begin{proof}
Without loss of generality, we can assume that the colouring $\xi$ maps each component of $L$ to $\A$. Note that the number of left half-circles in the diagram $L$ is equal to the number of right half-circles in the diagram $L$. So in the composition of the morphism $\{L\}_{\xi}$ the number of morphisms $b_{\A}$ is equal to the number of morphisms $d_{\A}$. So all values $y$ from $b_{\A}$ are reduced with values $\frac{1}{y}$ from $d_{\A}$.
\end{proof}

\subsection{Invariant for links}

Let $L = l_1\cup \ldots\cup l_k$ be a diagram of the unoriented $k$-component link in the 3-sphere $S^3$. Let $\xi_{\A}\colon \{l_1, \ldots, l_k\}\to \{\U, \A\}$ be a constant colouring which maps each component $l_i$ to $\A$ for all $i\in\{1, \ldots, k\}$.

For each component $l_i$ of the diagram $L$ with any fixed orientation, define the value $w(l_i)$ which is equal to the difference $\#p - \#n$, where $\#p$ is a number of positive crossings of the diagram $l_i$ and $\#n$ is a number of negative crossings of the diagram $l_i$. Note that both values $\#p$ and $\#n$ are correctly defined and independent of the orientation of $l_i$. Let $w(L) = \sum\limits_{i = 1}^{k}w(l_i)$.

Define $$tr_{\varepsilon}(L) = \frac{\beta_{\varepsilon}^{2w(L)}}{\varepsilon}\cdot \{L\}_{\xi_{\A}}.$$

\begin{theorem}
\label{Theorem:TRKnotInvariant}
Let $L_1$ and $L_2$ be two diagrams of the same link in $S^3$. Then $tr_{\varepsilon}(L_1) = tr_{\varepsilon}(L_2)$.
\end{theorem}
\begin{proof}
It's enough to prove the theorem only for the case where $L_2$ was obtained from $L_1$ by a single Reidemeister move.

If $L_2$ is obtained from $L_1$ by a second or third Reidemeister move, then $w(L_1) = w(L_2)$. From \cite[Theorem I.2.5]{T} and the theorem \ref{Theorem:RibbonCategory} it follows that $\{L_1\}_{\xi_{\A}} = \{L_2\}_{\xi_{\A}}$. Strictly speaking, theorem I.2.5 from \cite{T} applies only to strict ribbon categories. Our category $\E$ is not strict. But it's well known that every non-strict monoidal category is equivalent to a strict one. Let $\E'$ be a strict monoidal category equivalent to $\E$. Then, first, the invariants $tr_{\varepsilon}$ for $\E$ and $\E'$ coincide, and second, in $\E'$ we can apply theorem I.2.5 from \cite{T}. So $tr_{\varepsilon}(L_1) = tr_{\varepsilon}(L_2)$.

Let $L_2$ be obtained from $L_1$ by a positive first Reidemeister move (add a positive loop at any string of the diagram). Then $w(L_2) = w(L_1) + 1$.

Let the morphism $\{L_1\}_{\xi_{\A}}$ be represented as a composition $\mu_1\circ \mu_2$, and let the morphism $\{L_2\}_{\xi_{\A}}$ be represented as a composition $\mu_1\circ (id_{\A^{\otimes n_1}}\otimes \theta_{\A}\otimes id_{\A^{\otimes n_2}})\circ \mu_2$. Then 
\begin{center}
$\{L_1\}_{\xi_{\A}} = [\mu_2]_{\U}\cdot [\mu_1]_{\U}$ and $\{L_2\}_{\xi_{\A}} = [\mu_2]_{\U}\cdot [id_{\A^{\otimes n_1}}\otimes \theta_{\A}\otimes id_{\A^{\otimes n_2}}]_{\U}\cdot [\mu_1]_{\U}$.
\end{center}

Notice that $[id_{\A^{\otimes n_1}}\otimes \theta_{\A}\otimes id_{\A^{\otimes n_2}}]_{\U} = \frac{1}{\beta_{\varepsilon}^2}\cdot E_{f_{n_1 + n_2 + 1}}$, where $f_n$ is $n$-th Fibonacci number. Hence $\{L_2\}_{\xi_{\A}} = \frac{1}{\beta_{\varepsilon}^2} \{L_{1}\}_{\xi_{\A}}$. Finally $$tr_{\varepsilon}(L_2) = \frac{\beta_{\varepsilon}^{2w(L_2)}}{\varepsilon}\cdot \{L_2\}_{\xi_{\A}} = \frac{\beta_{\varepsilon}^{2w(L_1)}\cdot \beta_{\varepsilon}^2}{\varepsilon}\cdot \frac{1}{\beta_{\varepsilon}^2}\cdot \{L_1\}_{\xi_{\A}} = \frac{\beta_{\varepsilon}^{2w(L_1)}}{\varepsilon}\cdot \{L_1\}_{\xi_{\A}} = tr_{\varepsilon}(L_1).$$

For the negative first Reidemeister move the proof is similar.
\end{proof}

Let $\mathbb{L}$ be a link in $S^3$, and let $L$ be a digram of that link. Then we can define $tr_{\varepsilon}(\mathbb{L}) = tr_{\varepsilon}(L)$. Theorem \ref{Theorem:TRKnotInvariant} implies that this definition is correct.

\begin{example}
Let $\mathbb{U}$ be a trivial knot (unknot) with trivial diagram $U$. Then $\{U\}_{\xi_{\A}} = b_{\A}\circ d_{\A} = \varepsilon$. So $tr_{\varepsilon}(\mathbb{U}) = 1$.

Let $3_1$ be a right trefoil knot. Using the example \ref{Example:Trefoil} we can calculate $tr_{\varepsilon}(3_1) = \frac{\beta_{\varepsilon}^6}{\varepsilon}\cdot (1 - 2\beta_{\varepsilon})= \frac{\beta_{\varepsilon}}{\varepsilon}\cdot (2\beta_{\varepsilon} - 1)$. If we fix certain values for $\varepsilon$ and $\beta_{\varepsilon}$, then $tr_{\varepsilon}(3_1)$ will be a concrete complex number.

Let $\mathbb{H}_2$ be a Hopf link. Then $tr_{\varepsilon}(\mathbb{H}_2) = \frac{S_{\A, \A}}{\varepsilon}$, where $S_{\A, \A}$ is a value of the $S$-matrix from the proof of the theorem \ref{Theorem:Modular}. Hence $tr_{\varepsilon}(H_2) = -\frac{1}{\varepsilon}$.
\end{example}

Let $\mathbb{H}_k$, $k\geqslant 1$, be a generalised $k$-component Hopf link. The diagram $H_k$ of this link consists of an ordered set of $k$ trivial circles $U_1, \ldots, U_k$ such that for each $i\in\{1, k - 1\}$ the pair of components $U_i, U_{i + 1}$ is linked as a minimal diagram of classical Hopf link (see figure \ref{Figure:GeneralHopfLink}). Note that $\mathbb{H}_1 = \mathbb{U}$ is a trivial knot, $\mathbb{H}_2$ is a classical Hopf link.

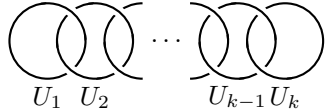
\begin{figure}[h]
\begin{center}
\begin{tikzpicture}[scale=0.5]
	\begin{knot}[
		clip width=5,
		% draft mode=crossings,
		% consider self intersections=true,
		ignore endpoint intersections=false,
		flip crossing/.list={2, 4, 5, 8, 10, 12, 14}
		]
		\strand (0, 0) circle (1.0);
		\strand (1.25, 0) circle (1.0);
		\strand (2.5, 1) arc (90:270:1);
		\strand (3.75, -1) arc (-90:90:1);
		\strand (5.0, 0) circle (1.0);	
		\strand (6.25, 0) circle (1.0);
	\end{knot}
	\draw (3.125, 0) node {$\ldots$};
	\draw (0.0, -1.5) node {$U_1$};
	\draw (1.25, -1.5) node {$U_2$};
	\draw (5.0, -1.5) node {$U_{k - 1}$};
	\draw (6.25, -1.5) node {$U_k$};
\end{tikzpicture}
\end{center}
\caption{\label{Figure:GeneralHopfLink}Generalised $k$-component Hopf link}
\end{figure}

\begin{theorem}
\label{Theorem:HopfLink}
$tr_{\varepsilon}(\mathbb{H}_k) = (-1)^{k - 1}\cdot \varepsilon^{1 - k}$ for $k\geqslant 1$.
\end{theorem}
\begin{proof}
Let $h\in Hom (\A\otimes \A, \A\otimes \A)$ be a morphism with a knotted diagram as shown in the figure \ref{Figure:MorphismH}. This diagram is splitted into eight layers $\lambda_1, \ldots, \lambda_8$. Corresponding morphisms are the following:

\begin{center}
\begin{tabular}{l}
$\mu_{\lambda_1} = (id_{\A}\otimes b_{\A})\otimes id_{\A}\in Hom (\A\otimes \A, (\A\otimes (\A\otimes \A))\otimes \A$), \\
$\mu_{\lambda_2} =\alpha_{\A, \A, \A}^{-1}\otimes id_{\A}  \in Hom((\A\otimes (\A\otimes \A))\otimes \A, ((\A\otimes \A)\otimes \A)\otimes \A)$, \\
$\mu_{\lambda_3} = (c_{\A, \A}\otimes id_{\A})\otimes id_{\A} \in Hom(((\A\otimes \A)\otimes \A)\otimes \A, ((\A\otimes \A)\otimes \A)\otimes \A)$, \\
$\mu_{\lambda_4} = \alpha_{\A\otimes\A, \A, \A} \in Hom(((\A\otimes \A)\otimes \A)\otimes \A, (\A\otimes \A)\otimes (\A\otimes \A))$, \\
$\mu_{\lambda_5} = (id_{\A}\otimes id_{\A})\otimes c_{\A, \A} \in Hom((\A\otimes \A)\otimes (\A\otimes \A), (\A\otimes \A)\otimes (\A\otimes \A))$, \\
$\mu_{\lambda_6} = \alpha_{\A, \A, \A\otimes \A} \in Hom((\A\otimes \A)\otimes (\A\otimes \A), \A\otimes (\A\otimes (\A\otimes \A)))$, \\
$\mu_{\lambda_7} = id_{\A}\otimes \alpha_{\A, \A, \A}^{-1} \in Hom(\A\otimes (\A\otimes (\A\otimes \A)), \A\otimes ((\A\otimes \A)\otimes \A))$, \\
$\mu_{\lambda_8} = id_{\A}\otimes (d_{\A}\otimes id_{\A}) \in Hom(\A\otimes ((\A\otimes \A)\otimes \A), \A\otimes \A)$.
\end{tabular}
\end{center}

\begin{figure}[h]
\begin{center}
\begin{tikzpicture}[scale=0.75, use Hobby shortcut]
\begin{knot}[
	clip width=5,
	consider self intersections=true,
	ignore endpoint intersections=false,
	flip crossing/.list={1}
]
\strand (0, 0) .. (1, 0) .. (2, 0) .. (3, 0) .. (4, 0) .. (5, 1) .. (6, 1) .. (7, 1) .. (7.5, 1.5) .. (7, 2) .. (6, 2) .. (5, 2) .. (4, 2) .. (3, 2) .. (2, 3) .. (1, 3) .. (0, 3);

\strand (8, 0) .. (7, 0) .. (6, 0) .. (5, 0) .. (4, 1) .. (3, 1) .. (2, 1) .. (1, 1) .. (0.5, 1.5) .. (1, 2) .. (2, 2) .. (3, 3) .. (4, 3) .. (5, 3) .. (6, 3) .. (7, 3) .. (8, 3);
\end{knot}

\foreach \i in {0, 1, 2, 3, 4, 5, 6, 7, 8}{
    \draw[split_line] (\i, -1) -- (\i, 4);
}

\foreach \i in {1, 2, 3, 4, 5, 6, 7}{
	\filldraw (\i, 0) circle (0.03);
	\filldraw (\i, 1) circle (0.03);
	\filldraw (\i, 2) circle (0.03);
	\filldraw (\i, 3) circle (0.03);
}
\filldraw (0, 0) circle (0.03);
\filldraw (0, 3) circle (0.03);
\filldraw (8, 0) circle (0.03);
\filldraw (8, 3) circle (0.03);

\draw (0.5, -0.75) node {$\lambda_1$};
\draw (1.5, -0.75) node {$\lambda_2$};
\draw (2.5, -0.75) node {$\lambda_3$};
\draw (3.5, -0.75) node {$\lambda_4$};
\draw (4.5, -0.75) node {$\lambda_5$};
\draw (5.5, -0.75) node {$\lambda_6$};
\draw (6.5, -0.75) node {$\lambda_7$};
\draw (7.5, -0.75) node {$\lambda_8$};
\end{tikzpicture}
\end{center}
\caption{\label{Figure:MorphismH}Knotted diagram of the morphism $h$}
\end{figure}
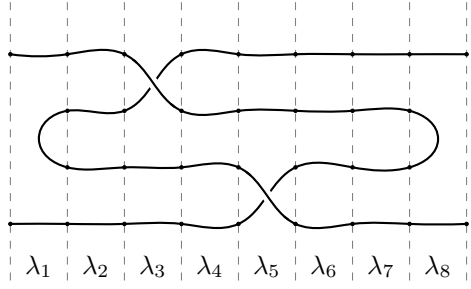

The diagram of the morphism $h$ is shown in the figure \ref{Figure:DiagramH}.

\begin{figure}[h]
\begin{center}
\begin{tikzpicture}[scale=0.5]
	\draw (0, -1) rectangle ++(1, 1) node[pos=0.5] {$\U$};
	\draw (0, -2) rectangle ++(1, 1) node[pos=0.5] {$\A$};
	\draw[-latex] (1, -0.5) -- (3, -0.5) node[above=0.0pt, midway] {$y\sqrt{\varepsilon}$};
	\draw[-latex] (1, -1.5) -- (3, -1.5) node[below=0.0pt, midway] {$y\sqrt{\varepsilon}$};
	
	\draw (3, -1) rectangle ++(1, 1) node[pos=0.5] {$\U$};
	\draw (3, -2) rectangle ++(1, 1) node[pos=0.5] {$\A$};
	\draw (3, -3) rectangle ++(1, 1) node[pos=0.5] {$\A$};
	\draw (3, -4) rectangle ++(1, 1) node[pos=0.5] {$\U$};
	\draw (3, -5) rectangle ++(1, 1) node[pos=0.5] {$\A$};
	\draw[ma_black] (4, -2.5) -- (6, -2.5);
	\draw[ma_blue] (4, -0.5) -- (6, -0.5);
	\draw[ma_blue] (4, -1.5) -- (6, -1.5);
	\draw[ma_green] (4, -0.5) -- (6, -3.5);
	\draw[ma_green] (4, -1.5) -- (6, -4.5);
	\draw[ma_red] (4, -3.5) -- (6, -0.5);
	\draw[ma_red] (4, -4.5) -- (6, -1.5);
	\draw[ma_orange] (4, -3.5) -- (6, -3.5);
	\draw[ma_orange] (4, -4.5) -- (6, -4.5);
	
	\draw (6, -1) rectangle ++(1, 1) node[pos=0.5] {$\U$};
	\draw (6, -2) rectangle ++(1, 1) node[pos=0.5] {$\A$};
	\draw (6, -3) rectangle ++(1, 1) node[pos=0.5] {$\A$};
	\draw (6, -4) rectangle ++(1, 1) node[pos=0.5] {$\U$};
	\draw (6, -5) rectangle ++(1, 1) node[pos=0.5] {$\A$};
	\draw[-latex] (7, -0.5) -- (9, -0.5) node[above=-2.0pt, midway] {$\beta_{\varepsilon}^2$};
	\draw[-latex] (7, -1.5) -- (9, -1.5) node[above=-2.0pt, midway] {$\beta_{\varepsilon}^2$};
	\draw[-latex] (7, -2.5) -- (9, -2.5) node[above=-2.0pt, midway] {$\beta_{\varepsilon}$};
	\draw[-latex] (7, -3.5) -- (9, -3.5) node[above=-2.0pt, midway] {$\beta_{\varepsilon}$};
	\draw[-latex] (7, -4.5) -- (9, -4.5) node[above=-2.0pt, midway] {$\beta_{\varepsilon}$};
	
	\draw (9, -1) rectangle ++(1, 1) node[pos=0.5] {$\U$};
	\draw (9, -2) rectangle ++(1, 1) node[pos=0.5] {$\A$};
	\draw (9, -3) rectangle ++(1, 1) node[pos=0.5] {$\A$};
	\draw (9, -4) rectangle ++(1, 1) node[pos=0.5] {$\U$};
	\draw (9, -5) rectangle ++(1, 1) node[pos=0.5] {$\A$};
	\draw[ma_black] (10, -0.5) -- (12, -0.5);
	\draw[ma_black] (10, -1.5) -- (12, -1.5);
	\draw[ma_black] (10, -3.5) -- (12, -3.5);
	\draw[ma_blue] (10, -2.5) -- (12, -2.5);
	\draw[ma_green] (10, -2.5) -- (12, -4.5);
	\draw[ma_red] (10, -4.5) -- (12, -2.5);
	\draw[ma_orange] (10, -4.5) -- (12, -4.5);
	
	\draw (12, -1) rectangle ++(1, 1) node[pos=0.5] {$\U$};
	\draw (12, -2) rectangle ++(1, 1) node[pos=0.5] {$\A$};
	\draw (12, -3) rectangle ++(1, 1) node[pos=0.5] {$\A$};
	\draw (12, -4) rectangle ++(1, 1) node[pos=0.5] {$\U$};
	\draw (12, -5) rectangle ++(1, 1) node[pos=0.5] {$\A$};
	\draw[-latex] (13, -0.5) -- (15, -0.5) node[above=-2.0pt, midway] {$\beta_{\varepsilon}^2$};
	\draw[-latex] (13, -1.5) -- (15, -1.5) node[above=-2.0pt, midway] {$\beta_{\varepsilon}$};
	\draw[-latex] (13, -2.5) -- (15, -2.5) node[above=-2.0pt, midway] {$\beta_{\varepsilon}^2$};
	\draw[-latex] (13, -3.5) -- (15, -3.5) node[above=-2.0pt, midway] {$\beta_{\varepsilon}$};
	\draw[-latex] (13, -4.5) -- (15, -4.5) node[above=-2.0pt, midway] {$\beta_{\varepsilon}$};
	
	\draw (15, -1) rectangle ++(1, 1) node[pos=0.5] {$\U$};
	\draw (15, -2) rectangle ++(1, 1) node[pos=0.5] {$\A$};
	\draw (15, -3) rectangle ++(1, 1) node[pos=0.5] {$\A$};
	\draw (15, -4) rectangle ++(1, 1) node[pos=0.5] {$\U$};
	\draw (15, -5) rectangle ++(1, 1) node[pos=0.5] {$\A$};
	\draw[ma_black] (16, -0.5) -- (18, -0.5);
	\draw[ma_black] (16, -2.5) -- (18, -1.5);
	\draw[ma_black] (16, -3.5) -- (18, -3.5);
	\draw[ma_blue] (16, -1.5) -- (18, -2.5);
	\draw[ma_green] (16, -1.5) -- (18, -4.5);
	\draw[ma_red] (16, -4.5) -- (18, -2.5);
	\draw[ma_orange] (16, -4.5) -- (18, -4.5);
	
	\draw (18, -1) rectangle ++(1, 1) node[pos=0.5] {$\U$};
	\draw (18, -2) rectangle ++(1, 1) node[pos=0.5] {$\A$};
	\draw (18, -3) rectangle ++(1, 1) node[pos=0.5] {$\A$};
	\draw (18, -4) rectangle ++(1, 1) node[pos=0.5] {$\U$};
	\draw (18, -5) rectangle ++(1, 1) node[pos=0.5] {$\A$};
	\draw[ma_black] (19, -2.5) -- (21, -2.5);
	\draw[ma_blue] (19, -0.5) -- (21, -0.5);
	\draw[ma_blue] (19, -1.5) -- (21, -1.5);
	\draw[ma_green] (19, -0.5) -- (21, -3.5);
	\draw[ma_green] (19, -1.5) -- (21, -4.5);
	\draw[ma_red] (19, -3.5) -- (21, -0.5);
	\draw[ma_red] (19, -4.5) -- (21, -1.5);
	\draw[ma_orange] (19, -3.5) -- (21, -3.5);
	\draw[ma_orange] (19, -4.5) -- (21, -4.5);
	
	\draw (21, -1) rectangle ++(1, 1) node[pos=0.5] {$\U$};
	\draw (21, -2) rectangle ++(1, 1) node[pos=0.5] {$\A$};
	\draw (21, -3) rectangle ++(1, 1) node[pos=0.5] {$\A$};
	\draw (21, -4) rectangle ++(1, 1) node[pos=0.5] {$\U$};
	\draw (21, -5) rectangle ++(1, 1) node[pos=0.5] {$\A$};
	\draw[-latex] (22, -0.5) -- (24, -0.5) node[above=-0.0pt, midway] {$\frac{\sqrt{\varepsilon}}{y}$};
	\draw[-latex] (22, -1.5) -- (24, -1.5) node[below=-0.0pt, midway] {$\frac{\sqrt{\varepsilon}}{y}$};
	
	\draw (24, -1) rectangle ++(1, 1) node[pos=0.5] {$\U$};
	\draw (24, -2) rectangle ++(1, 1) node[pos=0.5] {$\A$};
\end{tikzpicture}
\end{center}
\caption{\label{Figure:DiagramH}Diagram of the morphism $h$}
\end{figure}
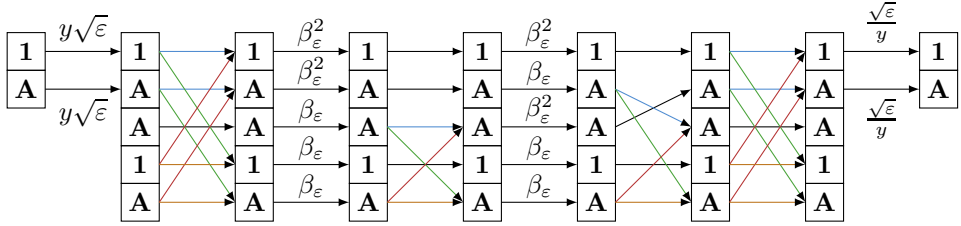

Next calculate
\begin{multline*}
[h]_{\U} = \begin{pmatrix}
\frac{\sqrt{\varepsilon}}{y} & 0
\end{pmatrix} \cdot \begin{pmatrix}
\frac{1}{\varepsilon} & \frac{x}{\sqrt{\varepsilon}} \\
\frac{1}{x\sqrt{\varepsilon}} & -\frac{1}{\varepsilon}
\end{pmatrix} \cdot \begin{pmatrix}
1 & 0 \\
0 & 1 
\end{pmatrix} \cdot \begin{pmatrix}
\beta_{\varepsilon}^2 & 0 \\
0 & \beta_{\varepsilon}
\end{pmatrix} \cdot \begin{pmatrix}
1 & 0 \\
0 & 1
\end{pmatrix} \cdot \\ \cdot \begin{pmatrix}
\beta_{\varepsilon}^2 & 0 \\
0 & \beta_{\varepsilon}
\end{pmatrix} \cdot \begin{pmatrix}
\frac{1}{\varepsilon} & \frac{x}{\sqrt{\varepsilon}} \\ 
\frac{1}{x\sqrt{\varepsilon}} & -\frac{1}{\varepsilon}
\end{pmatrix} \cdot \begin{pmatrix}
y\sqrt{\varepsilon} \\
0
\end{pmatrix} = \begin{pmatrix}
\frac{\beta_{\varepsilon}^4}{\varepsilon} + \beta_{\varepsilon}^2
\end{pmatrix} = \begin{pmatrix}
-\frac{1}{\varepsilon}
\end{pmatrix}.
\end{multline*}

Finally, notice that $$\{H_k\}_{\xi_{\A}} = b_{\A}\circ \underbrace{h\circ\ldots\circ h}_{k - 1 \text{\ times}}\circ d_{\A} = [d_{\A}]_{\U}\cdot [h]_{\U}^{k - 1} \cdot [b_{\A}]_{\U} = \frac{\sqrt{\varepsilon}}{y}\cdot \left(-\frac{1}{\varepsilon}\right)^{k - 1}\cdot y\sqrt{\varepsilon} = (-1)^{k - 1}\cdot \varepsilon^{2 - k}.$$

Hence $$tr_{\varepsilon}(\mathbb{H}_k) = \frac{1}{\varepsilon}\cdot \{H_k\}_{\xi_{\A}} = (-1)^{k - 1}\cdot \varepsilon^{1 - k}.$$
\end{proof}

\subsection{Invariant for 3-manifolds}

By framed link we mean the pair $(\mathbb{L}, \mathbb{F})$, where $\mathbb{L}$ is an unoriented $k$-component link with fixed order on the set of components, and $\mathbb{F} = (f_1, \ldots, f_k)$ is a tuple of integers. Each value $f_i$, $i\in \{1, \ldots, k\}$, is a framing of the $i$-th component. The diagram of the framed link is a diagram $L = l_1\cup \ldots l_k$ of the link $\mathbb{L}$ such that $w(l_i) = f_i$ for each $i\in\{1, \ldots, k\}$.

By the signature of the framed link $(\mathbb{L}, \mathbb{F})$ we mean a signature of the matrix $$[(\mathbb{L}, \mathbb{F})] = \begin{pmatrix}
f_1 & lk(l_1, l_2) & \ldots & lk(l_1, l_k) \\
lk(l_2, l_1) & f_2 & \ldots & lk(l_2, l_k) \\
\vdots & \vdots & & \vdots \\
lk(l_k, l_1) & lk(l_k, l_2) & \ldots & f_k
\end{pmatrix},$$
where $lk(l_i, l_j)$ is a linking number of the components $l_i$ and $l_j$. The signature can be computed in the following way. First of all we should consider the quadratic form with the matrix $[(\mathbb{L}, \mathbb{F})]$. Then the signature is a difference between the number of positive and negative summands in the canonical form of this quadratic form.

It's well known that any closed 3-manifold can be obtained from $S^3$ by surgery along a framed link. Denote the manifold obtained from the framed link $(\mathbb{L}, \mathbb{F})$ by $M_{(\mathbb{L}, \mathbb{F})}$.

Define $$tr_{\varepsilon}(M_{(\mathbb{L}, \mathbb{F})}) = \Delta^{\sigma}\cdot D^{-\sigma - k - 1}\cdot \sum\limits_{\xi}\varepsilon^{|\xi|_{\A}}\cdot \{L\}_{\xi},$$ where $\sigma$ is a signature of the framed link $(\mathbb{L}, \mathbb{F})$, $\Delta = 1 + \varepsilon^2\cdot \beta_{\varepsilon}^2$, $D = \sqrt{2 + \varepsilon}$, the sum takes over all colourings $\xi$ of the diagram $L$, and $|\xi|_{\A}$ is a number of values $\A$ in the colouring $\xi$. It follows from \cite[Theorem II.2.2.2]{T} that $tr_{\varepsilon}$ is an invariant for closed 3-manifolds, i.e. the value $tr_{\varepsilon}(M_{(\mathbb{L}, \mathbb{F})})$ does not depend on the framed link $(\mathbb{L}, \mathbb{F})$ which defines the manifold.

\begin{example}
\label{Example:TRForPoincareSphere}
Let $M$ be a Poincare sphere. It's well known that $M$ can be obtained from $S^3$ by surgery along the right trefoil $3_1$ with framing $f_1 = +1$. The diagram $T$ of this framed knot is shown in the figure \ref{Figure:TrefoilWithFraming}.

\begin{figure}[h]
\begin{center}
\begin{tikzpicture}[scale=0.75, use Hobby shortcut]
\begin{knot}[
	% draft mode=crossings,
	clip width=5,
	consider self intersections=true,
	ignore endpoint intersections=false,
	flip crossing/.list={5, 7}
]
\strand ([closed]1, 0) .. (2, 0) .. (2.5, -0.25) .. (3.25, 0.25) .. (3, 0.5) .. (2.75, 0.25) .. (3.5, 0.25) .. (4, 0) .. (4.5, 0.5) .. (4, 1) .. (3, 2) .. (2, 1) .. (1, 2) .. (0.5, 2.5) .. (1, 3) .. (1.5, 3.1) .. (2.25, 2.75) .. (2, 2.5) .. (1.75, 2.75) .. (2.5, 2.9) .. (3, 3) .. (4, 3) .. (4.5, 2.5) .. (4, 2) .. (3, 1) .. (2, 2) .. (1, 1) .. (0.5, 0.5);
\end{knot}
\end{tikzpicture}
\end{center}
\caption{\label{Figure:TrefoilWithFraming}Diagram $T$ of the right trefoil with framing $+1$}
\end{figure}
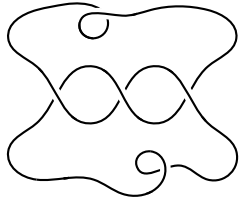

In this case $\sigma = 1$, and there are only two colourings: $\xi_{\U}$ colours $T$ by the object $\U$, $\xi_{\A}$ colours $T$ by the object $\A$. Then $\{T\}_{\xi_{\U}} = 1$ and $\{T\}_{\xi_{\A}} = (1 - 2\beta_{\varepsilon})\cdot \beta_{\varepsilon}^4$ (here we use the result of the example \ref{Example:Trefoil} and the fact that $[\theta_{\A}^{-1}]_{\A} = \beta_{\varepsilon}^2$). Finally $$
tr_{\varepsilon}(M) = \frac{(1 + \varepsilon^2\cdot \beta_{\varepsilon}^2)\cdot (1 + \varepsilon\cdot (\beta_{\varepsilon}^4 + 2))}{(\varepsilon + 2)^{\frac{3}{2}}}.
$$
\end{example}

Let $(\mathbb{H}_{k}, \mathbb{F})$ be a generalised Hopf link with the framing $\mathbb{F} = (f_1, \ldots, f_k)$, and let $H_k^{f_1, \ldots, f_k}$ be a digram of this framed link (figure \ref{Figure:FramedGeneralisedHopf}).

\begin{figure}[h]
\begin{center}
\begin{tikzpicture}[scale=0.75, use Hobby shortcut]
\begin{knot}[
	% draft mode=crossings,
	clip width=5,
	consider self intersections=true,
	ignore endpoint intersections=false,
	flip crossing/.list={3, 4, 5, 7, 11, 12}
]
	\strand ([closed]1, 0) .. (1.5, -0.1) .. (2.25, 0.25) .. (2, 0.5) .. (1.75, 0.25) .. (2.5, 0.1) .. (3, 0) .. (3.5, -0.1) .. (4.25, 0.25) .. (4, 0.5) .. (3.75, 0.25) .. (4.5, 0.1) .. (5, 0) .. (6, 1) .. (5, 2) .. (1, 2) .. (0, 1);
	\strand (6, 2) .. (5, 1) .. (6, 0);
	\strand (7, 0) .. (8, 1) .. (7, 2);
	\strand ([closed]8, 0) .. (8.5, -0.1) .. (9.25, 0.25) .. (9, 0.5) .. (8.75, 0.25) .. (9.5, 0.1) .. (10, 0) .. (10.5, -0.1) .. (11.25, 0.25) .. (11, 0.5) .. (10.75, 0.25) .. (11.5, 0.1) .. (12, 0) .. (13, 1) .. (12, 2) .. (8, 2) .. (7, 1);
\end{knot}
\draw (3, 0.3) node[] {$\ldots$};
\draw (3, 0.3) node[above] {$f_1$};

\draw (10, 0.3) node[] {$\ldots$};
\draw (10, 0.3) node[above] {$f_k$};

\draw (6.5, 1) node[] {$\ldots$};
\end{tikzpicture}
\end{center}
\caption{\label{Figure:FramedGeneralisedHopf}Diagram $H_{k}^{f_1, \ldots, f_k}$ of the framed generalised Hopf link}
\end{figure}
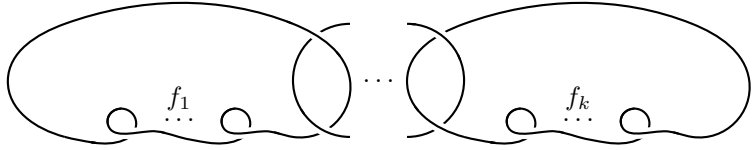

It's well known (see \cite{PS}) that $M_{(\mathbb{H}_{k}, \mathbb{F})}$ is a lens space $L_{p, q}$, where $$\frac{p}{q} = f_1 - \frac{1}{f_2 - \frac{1}{\ldots - \frac{1}{f_k}}}.$$

\begin{lemma}
\label{Lemma:MorphismGHL}
Let $\xi_{\A}$ be a constant colouring of the diagram $H_k^{f_1, \ldots, f_k}$, mapping each component to the object $\A$. Then
$$\{H_k^{f_1, \ldots, f_k}\}_{\xi_{\A}} = \frac{(-1)^{k - 1}}{\varepsilon^{k - 2}\cdot\beta_{\varepsilon}^{2f_1 + \ldots + 2f_k}}.$$
\end{lemma}
\begin{proof}
The proof is similar to the proof of the theorem \ref{Theorem:HopfLink}. Let $h_{f_{i}}\in Hom (\A\otimes \A, \A\otimes \A)$ be a morphism with a knotted diagram as shown in the figure \ref{Figure:MorphismsFramedH}.

\begin{figure}[h]
\begin{center}
\begin{tikzpicture}[scale=0.75, use Hobby shortcut]
\begin{knot}[
	% draft mode=crossings,
	clip width=5,
	consider self intersections=true,
	ignore endpoint intersections=false,
	flip crossing/.list={1, 3, 4}
]
	\strand (0, 0) .. (1, 1) .. (0, 2);
	\strand (5.5, 2) .. (1.5, 2) .. (0.5, 1) .. (1.5, 0) .. (2.0, -0.1) .. (2.75, 0.25) .. (2.5, 0.5) .. (2.25, 0.25) .. (3.0, 0.1) .. (3.5, 0) .. (4.0, -0.1) .. (4.75, 0.25) .. (4.5, 0.5) .. (4.25, 0.25) .. (5.0, 0.1) .. (5.5, 0);
\end{knot}
\filldraw (0, 0) circle (0.03);
\filldraw (0, 2) circle (0.03);
\filldraw (5.5, 0) circle (0.03);
\filldraw (5.5, 2) circle (0.03);

\draw (3.5, 0.25) node {$\ldots$};
\draw (3.5, 0.25) node[above] {$f_i$};
\end{tikzpicture}
\end{center}
\caption{\label{Figure:MorphismsFramedH}Knotted diagram of the morphism $h_{f_{i}}$}
\end{figure}
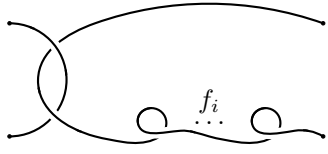

Note that $\underbrace{\theta_{\A}\circ \ldots\circ \theta_{\A}}_{f_i}$ is a morphism from $\A$ to $\A$ defined by the value $\frac{1}{\beta_{\varepsilon}^{2f_i}}$. Then 
\begin{center}
$[h_{f_i}]_{\U} = \begin{pmatrix}
-\frac{1}{\varepsilon\cdot \beta_{\varepsilon}^{2f_i}}
\end{pmatrix},
$ and $
[b_{\A}\circ \underbrace{(id_{\A}\otimes \theta_{\A})\circ \ldots\circ (id_{\A}\otimes \theta_{\A})}_{f_1\text{\ times}}]_{\U} = \begin{pmatrix}
\frac{y\sqrt{\varepsilon}}{\beta_{\varepsilon}^{2f_1}}
\end{pmatrix}.
$
\end{center}

Finally, $$\{H_{k}^{f_1, \ldots, f_k}\}_{\xi_{\A}} = (b_{\A}\circ \underbrace{(id_{\A}\otimes\theta_{\A})\circ \ldots\circ (id_{\A}\otimes\theta_{\A})}_{f_1\text{\ times}})\circ h_{f_2}\circ\ldots\circ h_{f_k}\circ d_{\A} = \frac{(-1)^{k - 1}\cdot \varepsilon}{\varepsilon^{k - 1}\cdot\beta_{\varepsilon}^{2f_1}\cdot \ldots \cdot \beta_{\varepsilon}^{2f_k}}.$$
\end{proof}

Let $\{i_1, i_2, \ldots, i_s\}$ be a sequence of natural numbers sorted in ascending order. Define the function $$c\colon \{i_1, \ldots, i_s\}\mapsto c(i_1, \ldots, i_s)$$ as follows: split the sequence $\{i_1, \ldots, i_s\}$ into maximal sub-sequences containing consecutive numbers (i.e. the number $i$ contained in a sub-sequence iff either $i - 1$ or $i + 1$ is contained in that sub-sequence). Let $r$ be a number of these maximal sub-sequences, and let their lengths be $l_1, \ldots, l_r$. Then $$c(i_1, \ldots, i_s) = \prod_{j = 1}^{r}\frac{(-1)^{l_j - 1}}{\varepsilon^{l_j - 2}}.$$

\begin{example}
\label{Example:CFunction}
$c(1, 3) = \varepsilon^2$, because there are only two sub-sequences in $\{1, 3\}$: $\{1\}$ and $\{3\}$. For both, $l_1 = l_2 = 1$, and so both multipliers are $\frac{(-1)^0}{\varepsilon^{-1}} = \varepsilon$.

$c(1, 2, 3, 4) = -\frac{1}{\varepsilon^2}$, because for the sequence $\{1, 2, 3, 4\}$: $r = 1$, $l_1 = 4$ and therefore the unique multiplier is $\frac{(-1)^3}{\varepsilon^2}$.

$c(2, 3, 4, 6, 7, 9, 10, 11) = -\frac{1}{\varepsilon^2}$, because the sequence $\{2, 3, 4, 6, 7, 9, 10, 11\}$ splits into three sub-sequences: $\{2, 3, 4\}$ with length $l_1 = 3$, $\{6, 7\}$ with length $l_2 = 2$ and $\{9, 10, 11\}$ with length $l_3 = 3$. So $$c(2, 3, 4, 6, 7, 9, 10, 11) = \frac{1}{\varepsilon}\cdot (-1)\cdot \frac{1}{\varepsilon} = -\frac{1}{\varepsilon^2}.$$
\end{example}

\begin{theorem}
\label{Theorem:TRLenseSpaces}
Let $\mathbb{F} = (f_1, \ldots, f_k)$, and let $L_{p, q}$ be a lens space, obtained by surgery $S^3$ along the framed link $(\mathbb{H}_k, \mathbb{F})$. Then $$tr_{\varepsilon}(L_{p, q}) = \frac{(1 + \varepsilon^2\cdot \beta_{\varepsilon}^2)^{\sigma}}{(\varepsilon + 2)^{\frac{\sigma + k + 1}{2}}}\cdot \left(1 + \sum\limits_{s = 1}^{k}\left(\varepsilon^{s}\cdot \sum\limits_{\{i_1, \ldots, i_s\}}\frac{c(i_1, \ldots, i_s)}{\beta_{\varepsilon}^{2f_{i_1} + \ldots + 2f_{i_s}}}\right)\right),$$ where the second sum takes over all subsets $\{i_1, \ldots, i_s\}\subseteq \{1, \ldots, k\}$ of order $s$.
\end{theorem}
\begin{proof}
Let $\xi$ be a colouring of the diagram $H_{k}^{f_1, \ldots, f_k}$. This colouring is defined by selecting a subset $\{i_1, \ldots, i_s\}\subseteq \{1, \ldots, k\}$. Components of the diagram $H_{k}^{f_1, \ldots, f_k}$ with indices in this subset are coloured by $\A$, all other components are coloured by $\U$. For this colouring $|\xi|_{\A} = s$. By the lemma \ref{Lemma:MorphismGHL} $$\{H_{k}^{f_1, \ldots, f_k}\}_{\xi} = \frac{c(i_1, \ldots, i_s)}{\beta_{\varepsilon}^{2f_{i_1} + \ldots + 2f_{i_s}}}.$$

The theorem statement obtained by substituting these values and values for $\Delta$ and $D$ into the formula of the invariant $tr_{\varepsilon}$.
\end{proof}

\begin{proposition}
\label{Proposition:EqualLenses}
Let $p\geqslant 1$ then $tr_{\varepsilon}(L_{p, q}) = tr_{\varepsilon}(L_{p + 5q, q})$.
\end{proposition}
\begin{proof}
Note that if $L_{p, q} = M_{(\mathbb{H}_k, \mathbb{F})}$, where $\mathbb{F} = (f_1, f_2, \ldots, f_k)$, then $L_{5q + p, q} = M_{(\mathbb{H}_k, \mathbb{F}')}$, where $\mathbb{F}' = (f_1 + 5, f_2, \ldots, f_k)$. It's clear that the signatures of the framed links $(\mathbb{H}_k, \mathbb{F})$ and $(\mathbb{H}_k, \mathbb{F}')$ are the same. Next, the value $\beta_{\varepsilon}$ is a root of unity of degree 10. So changing $f_1$ to $f_1 + 5$ has no effect on the value $tr_{\varepsilon}$ from the theorem \ref{Theorem:TRLenseSpaces}.
\end{proof}

\section{Connection between $tv_{\varepsilon}$ and $\varepsilon$-invariant}

\subsection{Invariant $tv_{\varepsilon}$}

The book \cite{T} contains the algorithm which allows to extract the Turaev -- Viro type invariant for 3-manifolds from any modular category. In this section we apply this algorithm to the category $\E$ and extract the invariant $tv_{\varepsilon}$.

\subsubsection{Multiplicity modules}

Define $H^{XYZ} = Hom(\U, (X\otimes Y)\otimes Z)$ for all simple objects $X, Y, Z\in I = \{\U, \A\}$. Every $H^{XYZ}$ is a module over $\mathbb{C}$. It's clear that the module $H^{XYZ}$ is isomorphic to $\mathbb{C}$ if among the objects $X, Y, Z$ the number of $\A$ is not equal to one (i.e. either $X = Y = Z = \U$, or $X = Y = Z = \A$, or two of them are $\A$ and the other is $\U$). All other modules are trivial. We will say that the triplet $(X, Y, Z)$ is admissible if the module $H^{XYZ}$ is not trivial.

For any $s\in \mathbb{C}$ and any object $X$ from the category $\E$, let $u^X_s\in Hom (X, X)$ denote the morphism from $X$ to $X$ defined by the following matrices:
\begin{center}
$[u^X_s]_{\U} = diag(s, \ldots, s)$ and $[u^X_s]_{\A} = diag (s, \ldots, s)$.
\end{center}

Denote $v_{\U}' = 1$ and $v_{\A}' = \frac{1}{\beta_{\varepsilon}}$.

For each triplet of simple objects $X, Y, Z\in I$, consider two isomorphisms $\mathcal{I}^{XYZ}_{12}\colon H^{XYZ}\to H^{YXZ}$ and $\mathcal{I}^{XYZ}_{23}\colon H^{XYZ}\to H^{XZY}$, defined by the following compositions: $$\mathcal{I}^{XYZ}_{12}(a) = a\circ (c_{X, Y}\otimes id_{Z})\circ u^{(Y\otimes X)\otimes Z}_{v_{X}' v_{Y}' (v_{Z}')^{-1}},$$ $$\mathcal{I}^{XYZ}_{23}(a) = a\circ \alpha_{X, Y, Z}\circ (id_{X}\otimes c_{Y, Z})\circ \alpha^{-1}_{X, Z, Y}\circ u^{(X\otimes Z)\otimes Y}_{(v_{X}')^{-1}v_{Y}' v_{Z}'}$$ for each morphism $a\in H^{XYZ}$.

\begin{remark}
\label{Remark:ModulesIsomorphisms}
The knotted diagram of the isomorphism $\mathcal{I}^{XYZ}_{12}$ is shown in the figure \ref{Figure:ModulesIsomorphisms} on the left, the diagram of the isomorphism $\mathcal{I}^{XYZ}_{23}$ is shown in the figure \ref{Figure:ModulesIsomorphisms} on the right.
\end{remark}

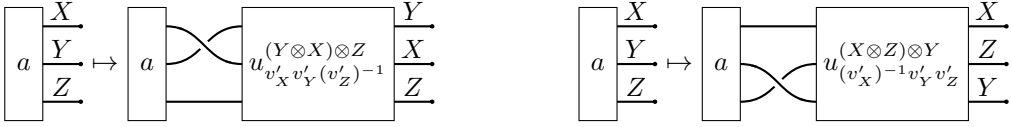
\begin{figure}[h]
\begin{center}
\ \hfill
\begin{tikzpicture}[scale=0.5, baseline={([yshift=-1.1ex]current bounding box.center)}]
	\draw (0, -3) rectangle ++(1, 3) node[pos=0.5] {$a$};
	\filldraw[knot_diagram] (1, -0.5) -- (2, -0.5) node[above=-1.5pt, midway] {$X$} circle (0.03);
	\filldraw[knot_diagram] (1, -1.5) -- (2, -1.5) node[above=-1.5pt, midway] {$Y$} circle (0.03);
	\filldraw[knot_diagram] (1, -2.5) -- (2, -2.5) node[above=-1.5pt, midway] {$Z$} circle (0.03);
\end{tikzpicture} $\mapsto$
\begin{tikzpicture}[scale=0.5, baseline={([yshift=-1.1ex]current bounding box.center)}]
	\draw (0, -3) rectangle ++(1, 3) node[pos=0.5] {$a$};
	
	\begin{knot}[
		clip width=5,
		% draft mode=crossings,
		% consider self intersections=true,
		ignore endpoint intersections=false,
		% flip crossing/.list={2, 4, 5, 8, 10, 12, 14}
		]
		\strand (1, -0.5) .. controls (2.0, -0.5) and (2.0 , -1.5) .. (3, -1.5);
		\strand (1, -1.5) .. controls (2.0, -1.5) and (2.0 , -0.5) .. (3, -0.5);
		\strand (1, -2.5) -- (3, -2.5);
	\end{knot}

	\draw (3, -3) rectangle ++(4, 3) node[pos=0.5] {$u^{(Y\otimes X)\otimes Z}_{v_{X}' v_{Y}' (v_{Z}')^{-1}}$};
	\filldraw[knot_diagram] (7, -0.5) -- (8, -0.5) node[above=-1.5pt, midway] {$Y$} circle (0.03);
	\filldraw[knot_diagram] (7, -1.5) -- (8, -1.5) node[above=-1.5pt, midway] {$X$} circle (0.03);
	\filldraw[knot_diagram] (7, -2.5) -- (8, -2.5) node[above=-1.5pt, midway] {$Z$} circle (0.03);
\end{tikzpicture}
\hfill
\begin{tikzpicture}[scale=0.5, baseline={([yshift=-1.1ex]current bounding box.center)}]
	\draw (0, -3) rectangle ++(1, 3) node[pos=0.5] {$a$};
	\filldraw[knot_diagram] (1, -0.5) -- (2, -0.5) node[above=-1.5pt, midway] {$X$} circle (0.03);
	\filldraw[knot_diagram] (1, -1.5) -- (2, -1.5) node[above=-1.5pt, midway] {$Y$} circle (0.03);
	\filldraw[knot_diagram] (1, -2.5) -- (2, -2.5) node[above=-1.5pt, midway] {$Z$} circle (0.03);
\end{tikzpicture} $\mapsto$
\begin{tikzpicture}[scale=0.5, baseline={([yshift=-1.1ex]current bounding box.center)}]
	\draw (0, -3) rectangle ++(1, 3) node[pos=0.5] {$a$};
	
	\begin{knot}[
		clip width=5,
		% draft mode=crossings,
		% consider self intersections=true,
		ignore endpoint intersections=false,
		% flip crossing/.list={2, 4, 5, 8, 10, 12, 14}
		]
		\strand (1, -1.5) .. controls (2.0, -1.5) and (2.0 , -2.5) .. (3, -2.5);
		\strand (1, -2.5) .. controls (2.0, -2.5) and (2.0 , -1.5) .. (3, -1.5);
		\strand (1, -0.5) -- (3, -0.5);
	\end{knot}

	\draw (3, -3) rectangle ++(4, 3) node[pos=0.5] {$u^{(X\otimes Z)\otimes Y}_{(v_{X}')^{-1} v_{Y}' v_{Z}'}$};
	\filldraw[knot_diagram] (7, -0.5) -- (8, -0.5) node[above=-1.5pt, midway] {$X$} circle (0.03);
	\filldraw[knot_diagram] (7, -1.5) -- (8, -1.5) node[above=-1.5pt, midway] {$Z$} circle (0.03);
	\filldraw[knot_diagram] (7, -2.5) -- (8, -2.5) node[above=-1.5pt, midway] {$Y$} circle (0.03);
\end{tikzpicture}
\hfill \ \ 
\end{center}
\caption{\label{Figure:ModulesIsomorphisms}Isomorphism $\mathcal{I}^{XYZ}_{12}$ (on the left) and $\mathcal{I}^{XYZ}_{23}$ (on the right)}
\end{figure}

\begin{proposition}
\label{Proposition:ModulesIsomorphisms}
For any admissible triplet $(X,Y,Z)$ the two isomorphisms $\mathcal{I}^{XYZ}_{12}$ and $\mathcal{I}^{XYZ}_{23}$ are identities.
\end{proposition}
\begin{proof}
Consider five different cases with respect to different admissible triplets $(X, Y, Z)$.

$X = Y = Z = \U$. In this case $\mathcal{I}^{\U\U\U}_{12}(a) = a\circ (c_{\U, \U}\otimes id_{\U})\circ u^{\U}_{1} = a$ and $\mathcal{I}^{\U\U\U}_{23}(a) = a\circ \alpha_{\U, \U, \U}\circ (id_{\U}\otimes c_{\U, \U})\circ \alpha_{\U, \U, \U}^{-1}\circ u^{\U}_{1} = a$, because $c_{\U, \U} = id_{\U}$ and $\alpha_{\U, \U, \U} = id_{\U}$.

$X = \U$, $Y = Z = \A$. In this case $\mathcal{I}^{\U\A\A}_{12}(a) = a\circ (c_{\U, \A}\otimes id_{\A})\circ u^{\U + \A}_{1} = a$ and $\mathcal{I}^{\U\A\A}_{23}(a) = a\circ \alpha_{\U, \A, \A}\circ (id_{\U}\otimes c_{\A, \A})\circ \alpha_{\U, \A, \A}^{-1}\circ u^{\U + \A}_{\frac{1}{\beta_{\varepsilon}^2}}$. The diagram of the second composition is shown in the figure \ref{Figure:ModulesIsomorpismsUAA}.

\begin{figure}[h]
\begin{center}
\begin{tikzpicture}[scale=0.5, baseline={([yshift=-0.7ex]current bounding box.center)}]
	\draw (0, -1) rectangle ++(1, 1) node[pos=0.5] {$\U$};
	
	\draw (3, -1) rectangle ++(1, 1) node[pos=0.5] {$\U$};
	\draw (3, -2) rectangle ++(1, 1) node[pos=0.5] {$\A$};
	\draw[-latex] (1, -0.5) -- (3, -0.5) node[above, midway] {$a$};
	
	\draw (6, -1) rectangle ++(1, 1) node[pos=0.5] {$\U$};
	\draw (6, -2) rectangle ++(1, 1) node[pos=0.5] {$\A$};
	\draw[-latex] (4, -0.5) -- (6, -0.5) node[above, midway] {$\beta_{\varepsilon}^2$};
	\draw[-latex] (4, -1.5) -- (6, -1.5) node[below, midway] {$\beta_{\varepsilon}$};
	
	\draw (9, -1) rectangle ++(1, 1) node[pos=0.5] {$\U$};
	\draw (9, -2) rectangle ++(1, 1) node[pos=0.5] {$\A$};
	\draw[-latex] (7, -0.5) -- (9, -0.5) node[above, midway] {$\frac{1}{\beta_{\varepsilon}^2}$};
	\draw[-latex] (7, -1.5) -- (9, -1.5) node[below, midway] {$\frac{1}{\beta_{\varepsilon}^2}$};
\end{tikzpicture} $=$
\begin{tikzpicture}[scale=0.5, baseline={([yshift=-1.2ex]current bounding box.center)}]
	\draw (0, -1) rectangle ++(1, 1) node[pos=0.5] {$\U$};
	
	\draw (3, -1) rectangle ++(1, 1) node[pos=0.5] {$\U$};
	\draw (3, -2) rectangle ++(1, 1) node[pos=0.5] {$\A$};
	\draw[-latex] (1, -0.5) -- (3, -0.5) node[above, midway] {$a$};
\end{tikzpicture}
\end{center}
\caption{\label{Figure:ModulesIsomorpismsUAA}Diagram of the compositions $a\circ \alpha_{\U, \A, \A}\circ (id_{\U}\otimes c_{\A, \A})\circ \alpha_{\U, \A, \A}^{-1}\circ u^{\U + \A}_{\frac{1}{\beta_{\varepsilon}^2}}$ and $a\circ (c_{\A, \A}\otimes id_{\U})\circ u^{\U + \A}_{\frac{1}{\beta_{\varepsilon}^2}}$}
\end{figure}
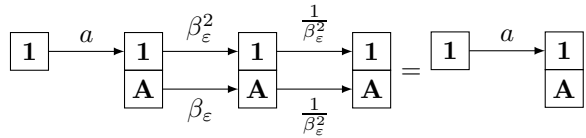

$X = \A$, $Y = \U$, $Z = \A$. In this case $\mathcal{I}^{\A\U\A}_{12}(a) = a\circ (c_{\A, \U}\otimes id_{\A})\circ u^{\U + \A}_{1} = a$ and $\mathcal{I}^{\A\U\A}_{23}(a) = a\circ \alpha_{\A, \U, \A}\circ (id_{\A}\otimes c_{\U, \A})\circ \alpha_{\A, \A, \U}^{-1}\circ u^{\U + \A}_{1} = a$, because $c_{\A, \U} = c_{\U, \A} = id_{\A}$ and $\alpha_{\A, \U, \A} = \alpha_{\A, \A, \U} = id_{\U + \A}$.

$X = Y = \A$, $Z = \U$. In this case $\mathcal{I}^{\A\A\U}_{12}(a) = a\circ (c_{\A, \A}\otimes id_{\U})\circ u^{\U + \A}_{\frac{1}{\beta_{\varepsilon}^2}}$ and $\mathcal{I}^{\A\A\U}_{23}(a) = a\circ \alpha_{\A, \A, \U}\circ (id_{\A}\otimes c_{\A, \U})\circ \alpha_{\A, \U, \A}^{-1}\circ u^{\U + \A}_{1} = a$. The digram of the first composition is shown in the figure \ref{Figure:ModulesIsomorpismsUAA}.

$X = Y = Z = \A$. In this case $\mathcal{I}^{\A\A\A}_{12}(a) = a\circ (c_{\A, \A}\otimes id_{\A})\circ u^{\A + \U + \A}_{\frac{1}{\beta_{\varepsilon}}}$ and $\mathcal{I}^{\A\A\A}_{23}(a) = a\circ \alpha_{\A, \A, \A}\circ (id_{\A}\otimes c_{\A, \A})\circ \alpha_{\A, \A, \A}^{-1}\circ u^{\A + \U + \A}_{\frac{1}{\beta_{\varepsilon}}}$. Diagram of the first composition is shown in the figure \ref{Figure:ModulesIsomorpismsAAA12}. Diagram of the second composition is shown in the figure \ref{Figure:ModulesIsomorpismsAAA23}.

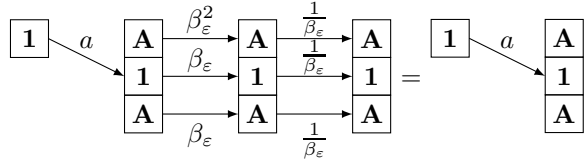
\begin{figure}[h]
\begin{center}
\begin{tikzpicture}[scale=0.5, baseline={([yshift=-0.4ex]current bounding box.center)}]
	\draw (0, -1) rectangle ++(1, 1) node[pos=0.5] {$\U$};
	
	\draw (3, -1) rectangle ++(1, 1) node[pos=0.5] {$\A$};
	\draw (3, -2) rectangle ++(1, 1) node[pos=0.5] {$\U$};
	\draw (3, -3) rectangle ++(1, 1) node[pos=0.5] {$\A$};
	\draw[-latex] (1, -0.5) -- (3, -1.5) node[above, midway] {$a$};
	
	\draw (6, -1) rectangle ++(1, 1) node[pos=0.5] {$\A$};
	\draw (6, -2) rectangle ++(1, 1) node[pos=0.5] {$\U$};
	\draw (6, -3) rectangle ++(1, 1) node[pos=0.5] {$\A$};
	\draw[-latex] (4, -0.5) -- (6, -0.5) node[above=-2.0pt, midway] {$\beta_{\varepsilon}^2$};
	\draw[-latex] (4, -1.5) -- (6, -1.5) node[above=-2.0pt, midway] {$\beta_{\varepsilon}$};
	\draw[-latex] (4, -2.5) -- (6, -2.5) node[below, midway] {$\beta_{\varepsilon}$};
	
	\draw (9, -1) rectangle ++(1, 1) node[pos=0.5] {$\A$};
	\draw (9, -2) rectangle ++(1, 1) node[pos=0.5] {$\U$};
	\draw (9, -3) rectangle ++(1, 1) node[pos=0.5] {$\A$};
	\draw[-latex] (7, -0.5) -- (9, -0.5) node[above=-3.0pt, midway] {$\frac{1}{\beta_{\varepsilon}}$};
	\draw[-latex] (7, -1.5) -- (9, -1.5) node[above=-3.0pt, midway] {$\frac{1}{\beta_{\varepsilon}}$};
	\draw[-latex] (7, -2.5) -- (9, -2.5) node[below, midway] {$\frac{1}{\beta_{\varepsilon}}$};
\end{tikzpicture} $=$
\begin{tikzpicture}[scale=0.5, baseline={([yshift=-0.75ex]current bounding box.center)}]
	\draw (0, -1) rectangle ++(1, 1) node[pos=0.5] {$\U$};
	
	\draw (3, -1) rectangle ++(1, 1) node[pos=0.5] {$\A$};
	\draw (3, -2) rectangle ++(1, 1) node[pos=0.5] {$\U$};
	\draw (3, -3) rectangle ++(1, 1) node[pos=0.5] {$\A$};
	\draw[-latex] (1, -0.5) -- (3, -1.5) node[above, midway] {$a$};
\end{tikzpicture}
\end{center}
\caption{\label{Figure:ModulesIsomorpismsAAA12}Diagram of the composition $a\circ (c_{\A, \A}\otimes id_{\A})\circ u^{\A + \U + \A}_{\frac{1}{\beta_{\varepsilon}}}$}
\end{figure}

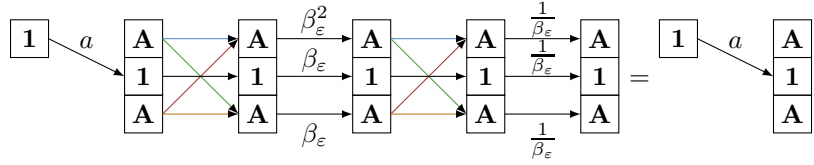
\begin{figure}[h]
\begin{center}
\begin{tikzpicture}[scale=0.5, baseline={([yshift=-0.4ex]current bounding box.center)}]
	\draw (0, -1) rectangle ++(1, 1) node[pos=0.5] {$\U$};
	
	\draw (3, -1) rectangle ++(1, 1) node[pos=0.5] {$\A$};
	\draw (3, -2) rectangle ++(1, 1) node[pos=0.5] {$\U$};
	\draw (3, -3) rectangle ++(1, 1) node[pos=0.5] {$\A$};
	\draw[-latex] (1, -0.5) -- (3, -1.5) node[above, midway] {$a$};
	
	\draw (6, -1) rectangle ++(1, 1) node[pos=0.5] {$\A$};
	\draw (6, -2) rectangle ++(1, 1) node[pos=0.5] {$\U$};
	\draw (6, -3) rectangle ++(1, 1) node[pos=0.5] {$\A$};
	\draw[ma_blue] (4, -0.5) -- (6, -0.5);
	\draw[ma_green] (4, -0.5) -- (6, -2.5);
	\draw[ma_red] (4, -2.5) -- (6, -0.5);
	\draw[ma_orange] (4, -2.5) -- (6, -2.5);
	\draw[ma_black] (4, -1.5) -- (6, -1.5);
	
	\draw (9, -1) rectangle ++(1, 1) node[pos=0.5] {$\A$};
	\draw (9, -2) rectangle ++(1, 1) node[pos=0.5] {$\U$};
	\draw (9, -3) rectangle ++(1, 1) node[pos=0.5] {$\A$};
	\draw[-latex] (7, -0.5) -- (9, -0.5) node[above=-2.0pt, midway] {$\beta_{\varepsilon}^2$};
	\draw[-latex] (7, -1.5) -- (9, -1.5) node[above=-2.0pt, midway] {$\beta_{\varepsilon}$};
	\draw[-latex] (7, -2.5) -- (9, -2.5) node[below, midway] {$\beta_{\varepsilon}$};
	
	\draw (12, -1) rectangle ++(1, 1) node[pos=0.5] {$\A$};
	\draw (12, -2) rectangle ++(1, 1) node[pos=0.5] {$\U$};
	\draw (12, -3) rectangle ++(1, 1) node[pos=0.5] {$\A$};
	\draw[ma_blue] (10, -0.5) -- (12, -0.5);
	\draw[ma_green] (10, -0.5) -- (12, -2.5);
	\draw[ma_red] (10, -2.5) -- (12, -0.5);
	\draw[ma_orange] (10, -2.5) -- (12, -2.5);
	\draw[ma_black] (10, -1.5) -- (12, -1.5);
	
	\draw (15, -1) rectangle ++(1, 1) node[pos=0.5] {$\A$};
	\draw (15, -2) rectangle ++(1, 1) node[pos=0.5] {$\U$};
	\draw (15, -3) rectangle ++(1, 1) node[pos=0.5] {$\A$};
	\draw[-latex] (13, -0.5) -- (15, -0.5) node[above=-3.0pt, midway] {$\frac{1}{\beta_{\varepsilon}}$};
	\draw[-latex] (13, -1.5) -- (15, -1.5) node[above=-3.0pt, midway] {$\frac{1}{\beta_{\varepsilon}}$};
	\draw[-latex] (13, -2.5) -- (15, -2.5) node[below, midway] {$\frac{1}{\beta_{\varepsilon}}$};
\end{tikzpicture} $=$
\begin{tikzpicture}[scale=0.5, baseline={([yshift=-0.75ex]current bounding box.center)}]
	\draw (0, -1) rectangle ++(1, 1) node[pos=0.5] {$\U$};
	
	\draw (3, -1) rectangle ++(1, 1) node[pos=0.5] {$\A$};
	\draw (3, -2) rectangle ++(1, 1) node[pos=0.5] {$\U$};
	\draw (3, -3) rectangle ++(1, 1) node[pos=0.5] {$\A$};
	\draw[-latex] (1, -0.5) -- (3, -1.5) node[above, midway] {$a$};
\end{tikzpicture}
\end{center}
\caption{\label{Figure:ModulesIsomorpismsAAA23}Diagram of the composition $a\circ \alpha_{\A, \A, \A}\circ (id_{\A}\otimes c_{\A, \A})\circ \alpha_{\A, \A, \A}^{-1}\circ u^{\A + \U + \A}_{\frac{1}{\beta_{\varepsilon}}}$}
\end{figure}
\end{proof}

For any admissible triplet $(X, Y, Z)$, define the module $H(X, Y, Z)$ obtained by identifying $H^{XYZ}$, $H^{XZY}$, $H^{YXZ}$, $H^{YZX}$, $H^{ZXY}$ and $H^{ZYX}$ by isomorphisms generated by $\mathcal{I}^{XYZ}_{12}$ and $\mathcal{I}^{XYZ}_{23}$. In fact, to get an arbitrary element $a\in H(X, Y, Z)$ we should choose the same element $a$ from all these six modules.

\subsubsection{Pairing}

Let $w_{\U}\in Hom (\U, \U)$ be a morphism of the category $\E$, defined by the value $1$ (i.e. $w_{\U} = id_{\U}$), and let $w_{\A}\in Hom (\A, \A)$ be a morphism of the category $\E$, defined by the non-zero value $z\in \mathbb{C}$.

\begin{lemma}
\label{lemma:WMOrphims}
For each simple object $X\in I$: $(w_{X}\otimes id_{X})\circ d_X = (id_{X}\otimes w_{X})\circ d_X$.
\end{lemma}
\begin{proof}
The lemma statement is obvious.
\end{proof}

For any admissible triplet of simple objects $(X, Y, Z)$, define the pairing $(\cdot, \cdot)^{XYZ}\colon H(X, Y, Z)\times H(X, Y, Z)\to\mathbb{C}$ as follows. For any two elements $a, b\in H(X, Y, Z)$, the value $(a, b)^{XYZ}$ is equal to the morphism shown in the figure \ref{Figure:PairingDiagram}. This morphism is a morphism from $\U$ to $\U$, so it is defined by a complex number.

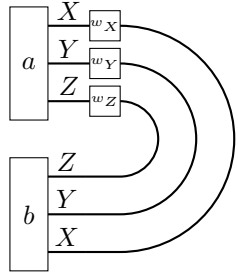
\begin{figure}[h]
\begin{center}
\begin{tikzpicture}[scale=0.5]
	\draw (0, -3) rectangle ++(1, 3) node[pos=0.5] {$a$};
	\draw (0, -7) rectangle ++(1, 3) node[pos=0.5] {$b$};
	
	\draw[knot_diagram] (1, -0.5) -- (2.1, -0.5) node[above=-1.5pt, midway] {$X$};
	\draw[knot_diagram] (1, -1.5) -- (2.1, -1.5) node[above=-1.5pt, midway] {$Y$};
	\draw[knot_diagram] (1, -2.5) -- (2.1, -2.5) node[above=-1.5pt, midway] {$Z$};
	
	\draw[knot_diagram] (1, -4.5) -- (2, -4.5) node[above=-1.5pt, midway] {$Z$};
	\draw[knot_diagram] (1, -5.5) -- (2, -5.5) node[above=-1.5pt, midway] {$Y$};
	\draw[knot_diagram] (1, -6.5) -- (2, -6.5) node[above=-1.5pt, midway] {$X$};
	
	\draw (2.1, -0.9) rectangle ++(0.8, 0.8) node[pos=0.5] {{\tiny $w_X$}};
	\draw (2.1, -1.9) rectangle ++(0.8, 0.8) node[pos=0.5] {{\tiny $w_Y$}};
	\draw (2.1, -2.9) rectangle ++(0.8, 0.8) node[pos=0.5] {{\tiny $w_Z$}};
	
	\draw[knot_diagram] (2.0, -4.5) -- (2.9, -4.5) arc (-90:90:1);
	\draw[knot_diagram] (2.0, -5.5) -- (2.9, -5.5) arc (-90:90:2);
	\draw[knot_diagram] (2.0, -6.5) -- (2.9, -6.5) arc (-90:90:3);

\end{tikzpicture}
\end{center}
\caption{\label{Figure:PairingDiagram}Knotted diagram of the morphism for the pairing $(a, b)^{XYZ}$}
\end{figure}

\begin{theorem}
\label{Theorem:Pairings}
$(a, b)^{\U\U\U} = ab$, $(a, b)^{\U\A\A} = ab y^2 z^2$ and $(a, b)^{\A\A\A} = ab xy^3 z^3$.
\end{theorem}
\begin{proof}
For any triplet of simple objects $(X, Y, Z)$ and for any $a, b\in H(X, Y, Z)$, the morphism $(a, b)^{XYZ}$ is a composition $\mu_1\circ\ldots \circ \mu_9$ of nine morphisms, where:
\begin{center}
\begin{longtable}{l}
$\mu_1 = a\otimes b$, \\
$\mu_2 = ((w_X\otimes w_Y)\otimes w_Z)\otimes ((id_Z\otimes id_Y)\otimes id_X)$, \\
$\mu_3 = \alpha^{-1}_{(X\otimes Y)\otimes Z, Z\otimes Y, X}$, \\
$\mu_4 = \alpha^{-1}_{(X\otimes Y)\otimes Z,Z, Y}\otimes id_X$, \\
$\mu_5 = (\alpha_{X\otimes Y, Z, Z}\otimes id_Y)\otimes id_X$, \\
$\mu_6 = (((id_X\otimes id_Y)\otimes d_Z)\otimes id_Y)\otimes  id_X$, \\
$\mu_7 = \alpha_{X, Y, Y}\otimes id_X$, \\
$\mu_8 = (id_X\otimes d_Y)\otimes id_X$, \\
$\mu_9 = d_X$.
\end{longtable}
\end{center}

The theorem follows from the careful computation of these morphisms for three different cases where $(X, Y, Z)$ coincides with $(\U, \U, \U)$, $(\U, \A, \A)$ or $(\A, \A, \A)$.
\end{proof}

\subsubsection{6j-symbols}

For any collection of simple objects $X_1, Y_1, Z_1, X_2, Y_2, Z_2\in I$, define 6j-symbol $$\left|\begin{array}{lll}
X_1 & Y_1 & Z_1 \\
X_2 & Y_2 & Z_2
\end{array}\right|\colon H(X_1, Y_1, Z_1)\otimes H(X_1, Y_2, Z_2)\otimes H(Y_1, Z_2, X_2)\otimes H(Z_1, X_2, Y_2)\to\mathbb{C}$$ as follows. If at least one of the triplets $(X_1, Y_1, Z_1)$, $(X_1, Y_2, Z_2)$, $(Y_1, Z_2, X_2)$ or $(Z_1, X_2, Y_2)$ is not admissible, then the corresponding module is trivial and the 6j-symbol is also a trivial map. If all triplets $(X_1, Y_1, Z_1)$, $(X_1, Y_2, Z_2)$, $(Y_1, Z_2, X_2)$ and $(Z_1, X_2, Y_2)$ are admissible, then all four modules $H(X_1, Y_1, Z_1)$, $H(X_1, Y_2, Z_2)$, $H(Y_1, Z_2, X_2)$ and $H(Z_1, X_2, Y_2)$ are isomorphic to $\mathbb{C}$. In this case the 6j-symbol maps $a_1\otimes a_2\otimes a_3\otimes a_4$, for any $a_1\in H(X_1, Y_1, Z_1)$, $a_2\in H(X_1, Y_1, Z_2)$, $a_3\in H(Y_1, Z_2, X_2)$, $a_4\in H(Z_1, X_2, Y_2)$, to the value of the morphism shown in the figure \ref{Figure:6JSYmbolDiagram}. This morphism is a morphism from $\U$ to $\U$, so it is defined by a complex value. This complex value defines the image $$\left|\begin{array}{lll}
X_1 & Y_1 & Z_1 \\
X_2 & Y_2 & Z_2
\end{array}\right|(a_1\otimes a_2\otimes a_3\otimes a_4).$$

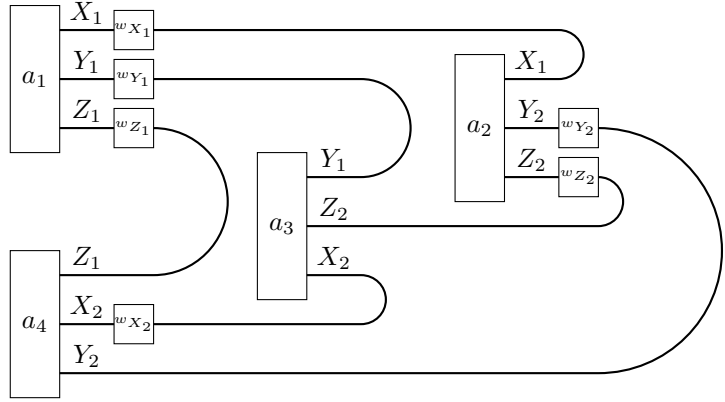
\begin{figure}[h]
\begin{center}
\begin{tikzpicture}[scale=0.65]
	\draw (0, 0) rectangle ++(1, 3) node[pos=0.5] {$a_4$};
	\draw (0, 5) rectangle ++(1, 3) node[pos=0.5] {$a_1$};
	\draw (5, 2) rectangle ++(1, 3) node[pos=0.5] {$a_3$};
	\draw (9, 4) rectangle ++(1, 3) node[pos=0.5] {$a_2$};
	
	\draw (2.1, 1.1) rectangle ++(0.8, 0.8) node[pos=0.5] {{\tiny $w_{X_2}$}};
	\draw (2.1, 5.1) rectangle ++(0.8, 0.8) node[pos=0.5] {{\tiny $w_{Z_1}$}};
	\draw (2.1, 6.1) rectangle ++(0.8, 0.8) node[pos=0.5] {{\tiny $w_{Y_1}$}};
	\draw (2.1, 7.1) rectangle ++(0.8, 0.8) node[pos=0.5] {{\tiny $w_{X_1}$}};
	\draw (11.1, 4.1) rectangle ++(0.8, 0.8) node[pos=0.5] {{\tiny $w_{Z_2}$}};
	\draw (11.1, 5.1) rectangle ++(0.8, 0.8) node[pos=0.5] {{\tiny $w_{Y_2}$}};
	
	\draw[knot_diagram] (1, 0.5) -- (2.1, 0.5) node[above=-1.5pt, midway] {$Y_2$} -- (11.9, 0.5) arc (-90:90:2.5);
	\draw[knot_diagram] (1, 1.5) -- (2.1, 1.5) node[above=-1.5pt, midway] {$X_2$};
	\draw[knot_diagram] (2.9, 1.5) -- (7.1, 1.5) arc (-90:90:0.5) -- (6, 2.5) node[above=-1.5pt, midway] {$X_2$};
	\draw[knot_diagram] (1, 2.5) -- (2.1, 2.5) node[above=-1.5pt, midway] {$Z_1$} -- (2.9, 2.5) arc (-90:90:1.5);
	\draw[knot_diagram] (1, 5.5) -- (2.1, 5.5) node[above=-1.5pt, midway] {$Z_1$};
	\draw[knot_diagram] (1, 6.5) -- (2.1, 6.5) node[above=-1.5pt, midway] {$Y_1$};
	\draw[knot_diagram] (1, 7.5) -- (2.1, 7.5) node[above=-1.5pt, midway] {$X_1$};
	\draw[knot_diagram] (6, 4.5) -- (7.1, 4.5) node[above=-1.5pt, midway] {$Y_1$} arc (-90:90:1) -- (2.9, 6.5);
	\draw[knot_diagram] (6, 3.5) -- (7.1, 3.5) node[above=-1.5pt, midway] {$Z_2$} -- (11.9, 3.5) arc (-90:90:0.5);
	\draw[knot_diagram] (10, 4.5) -- (11.1, 4.5) node[above=-1.5pt, midway] {$Z_2$};
	\draw[knot_diagram] (10, 5.5) -- (11.1, 5.5) node[above=-1.5pt, midway] {$Y_2$};
	\draw[knot_diagram] (10, 6.5) -- (11.1, 6.5) node[above=-1.5pt, midway] {$X_1$} arc (-90:90:0.5) -- (2.9, 7.5);

\end{tikzpicture}
\end{center}
\caption{\label{Figure:6JSYmbolDiagram}Knotted diagram of the 6j-symbol}
\end{figure}

\begin{remark}
\label{Remark:6JSymbolsSymmetries}
All 6j-symbols are symmetrical in the following sense: for any simple objects $X_1, Y_1, Z_1, X_2, Y_2, Z_2\in I$ and $a_1\in H(X_1, Y_1, Z_1)$, $a_2\in H(X_1, Y_2, Z_2)$, $a_3\in H(Y_1, Z_2, X_2)$, $a_4\in H(Z_1, X_2, Y_2)$ we have $$\left|\begin{array}{lll}
X_1 & Y_1 & Z_1 \\
X_2 & Y_2 & Z_2
\end{array}\right|(a_1\otimes a_2\otimes a_3\otimes a_4) = \left|\begin{array}{lll}
X_1 & Z_2 & Y_1 \\
X_2 & Z_1 & Y_1
\end{array}\right|(a_2\otimes a_1\otimes a_3\otimes a_4) = \left|\begin{array}{lll}
Z_1 & Y_2 & X_2 \\
Z_2 & Y_1 & X_1
\end{array}\right|(a_4\otimes a_1\otimes a_2\otimes a_3).$$
\end{remark}

\begin{theorem}
\label{Theorem:6JSymbols}
Non-trivial 6j-symbols are as follows:
\begin{enumerate}
\item[] $\left|\begin{array}{lll}
\U & \U & \U \\
\U & \U & \U
\end{array}\right|(a_1\otimes a_2\otimes a_3\otimes a_4) = a_1 a_2 a_3 a_4$,

\item[] $\left|\begin{array}{lll}
\U & \U & \U \\
\A & \A & \A
\end{array}\right|(a_1\otimes a_2\otimes a_3\otimes a_4) = a_1 a_2 a_3 a_4 \frac{y^3 z^3}{\sqrt{\varepsilon}}$,

\item[] $\left|\begin{array}{lll}
\U & \A & \A \\
\U & \A & \A
\end{array}\right|(a_1\otimes a_2\otimes a_3\otimes a_4) = a_1 a_2 a_3 a_4 \frac{y^4 z^4}{\varepsilon}$,

\item[] $\left|\begin{array}{lll}
\U & \A & \A \\
\A & \A & \A
\end{array}\right|(a_1\otimes a_2\otimes a_3\otimes a_4) = a_1 a_2 a_3 a_4 \frac{x y^5 z^5}{\varepsilon}$,

\item[] $\left|\begin{array}{lll}
\A & \A & \A \\
\A & \A & \A
\end{array}\right|(a_1\otimes a_2\otimes a_3\otimes a_4) = a_1 a_2 a_3 a_4 \frac{x^2 y^6 z^6}{-\varepsilon^2}$.
\end{enumerate}
\end{theorem}
\begin{proof}
For any six simple objects $X_1, Y_1, Z_1, X_2, Y_2, Z_2$ which define non-trivial 6j-symbols, and for any $a_1\in H(X_1, Y_1, Z_1)$, $a_2\in H(X_1, Y_2, Z_2)$, $a_3\in H(Y_1, Z_2, X_2)$, $a_4\in H(Z_1, X_2, Y_2)$ the morphism $$\left|\begin{array}{lll}
X_1 & Y_1 & Z_1 \\
X_2 & Y_2 & Z_2
\end{array}\right|(a_1\otimes a_2\otimes a_3\otimes a_4)
$$ is a composition $\mu_1\circ\ldots\circ \mu_{21}$, where
\begin{center}
\begin{longtable}{l}
$\mu_1 = a_1\otimes a_4$, \\
$\mu_2 = ((w_{X_1}\otimes w_{Y_1})\otimes w_{Z_1})\otimes ((id_{Z_1}\otimes w_{X_2})\otimes id_{Y_2})$, \\
$\mu_3 = \alpha^{-1}_{(X_1\otimes Y_1)\otimes Z_1, Z_1\otimes X_2, Y_2}$, \\
$\mu_4 = \alpha^{-1}_{(X_1\otimes Y_1)\otimes Z_1, Z_1, X_2}\otimes id_{Y_2}$, \\
$\mu_5 = (\alpha_{X_1\otimes Y_1, Z_1, Z_1}\otimes id_{X_2})\otimes id_{Y_2}$, \\
$\mu_6 = (((id_{X_1}\otimes id_{Y_1})\otimes d_{Z_1})\otimes id_{X_2})\otimes id_{Y_2}$, \\
$\mu_7 = (((id_{X_1}\otimes id_{Y_1})\otimes a_3)\otimes id_{X_2})\otimes id_{Y_2}$, \\
$\mu_8 = (\alpha^{-1}_{X_1\otimes Y_1, Y_1\otimes Z_2, X_2}\otimes id_{X_2})\otimes id_{Y_2}$, \\
$\mu_9 = ((\alpha^{-1}_{X_1\otimes Y_1, Y_1, Z_2}\otimes id_{X_2})\otimes id_{X_2})\otimes id_{Y_2}$, \\
$\mu_{10} = (((\alpha_{X_1, Y_1, Y_1}\otimes id_{Z_2})\otimes id_{X_2})\otimes id_{X_2})\otimes id_{Y_2}$, \\
$\mu_{11} = ((((id_{X_1}\otimes d_{Y_1})\otimes id_{Z_2})\otimes id_{X_2})\otimes id_{X_2})\otimes id_{Y_2}$, \\
$\mu_{12} = \alpha_{X_1\otimes Z_2, X_2, X_2}\otimes id_{Y_2}$, \\
$\mu_{13} = ((id_{X_1}\otimes id_{Z_2})\otimes d_{X_2})\otimes id_{Y_2}$, \\
$\mu_{14} = ((id_{X_1}\otimes a_2)\otimes id_{Z_2})\otimes id_{Y_2}$, \\
$\mu_{15} = ((id_{X_1}\otimes ((id_{X_1}\otimes w_{Y_2})\otimes w_{Z_2}))\otimes id_{Z_2})\otimes id_{Y_2}$, \\
$\mu_{16} = \alpha_{X_1, (X_1\otimes Y_2)\otimes Z_2, Z_2}\otimes id_{Y_2}$, \\
$\mu_{17} = (id_{X_1}\otimes \alpha_{X_1\otimes Y_2, Z_2, Z_2})\otimes id_{Y_2}$, \\
$\mu_{18} = (id_{X_1}\otimes ((id_{X_1}\otimes id_{Y_2})\otimes d_{Z_2}))\otimes id_{Y_2}$, \\
$\mu_{19} = \alpha^{-1}_{X_1, X_1, Y_2}\otimes id_{Y_2}$, \\
$\mu_{20} = (d_{X_1}\otimes id_{Y_2})\otimes id_{Y_2}$, \\
$\mu_{21} = d_{Y_2}$.
\end{longtable}
\end{center}

The theorem follows from the careful computation of these morphisms for all necessary combinations of objects $X_1, Y_1, Z_1, X_2, Y_2, Z_2\in I$.
\end{proof}

\subsubsection{Special spines and definition of $tv_{\varepsilon}$}

In the book \cite{T} the construction of the Turaev -- Viro type invariant, derived from modular category, is constructed by using triangulations of 3-manifolds. We will use the dual approach and describe the invariant $tv_{\varepsilon}$ using special spines of 3-manifolds.

A two-dimensional polyhedron $P$ is called special, if it satisfies to the following conditions:
\begin{enumerate}
\item The link of each point $x\in P$ is homeomorphic either to the circle (figure \ref{Figure:SpecialNeigh} on the left, these points are called regular points), or to the circle with diameter (figure \ref{Figure:SpecialNeigh} in the centre, these points are called triple points), or to the circle with two diameters (figure \ref{Figure:SpecialNeigh} on the right, these points are called true vertices);
\item The union of all triple points is a disjoint set of intervals called triple lines;
\item The union of all regular points is a disjoint union of open discs called 2-components.
\end{enumerate}

\begin{figure}[h]
\begin{center}
\ \hfill
\begin{tikzpicture}[scale=0.5, baseline={([yshift=-0.7ex]current bounding box.center)}]
	\draw[knot_diagram] (0, 0) -- (2, 2) -- (6, 2) -- (4, 0) -- (0, 0);
	\filldraw (3, 1) circle (0.05) node[below] {$x$};
\end{tikzpicture}
\hfill
\begin{tikzpicture}[scale=0.5, baseline={([yshift=-2.0ex]current bounding box.center)}]
\begin{knot}[
	clip width=5,
	ignore endpoint intersections=false,
]
	\strand (5, 2.75) -- (5, 1);
	\strand (2, 2) -- (6, 2);
\end{knot}
	\draw[knot_diagram] (2, 2) -- (0, 0) -- (4, 0) -- (6, 2);
	\draw[knot_diagram] (5, 1) -- (1, 1) -- (1, 2.75) -- (5, 2.75);
	\filldraw (3, 1) circle (0.05) node[below] {$x$};
\end{tikzpicture}
\hfill
\begin{tikzpicture}[scale=0.5, baseline={([yshift=0.7ex]current bounding box.center)}]
\begin{knot}[
	% draft mode=crossings,
	clip width=5,
	ignore endpoint intersections=false,
	flip crossing/.list={3}
]
	\strand (5, 2.75) -- (5, 1) -- (1, 1);
	\strand (2, 2) -- (6, 2);
	\strand (4, 1.75) -- (4, 0.25) -- (2, -1.75);
	\strand (0, 0) -- (4, 0);
\end{knot}
	\draw[knot_diagram] (0, 0) -- (2, 2);
	\draw[knot_diagram] (1, 1) -- (1, 2.75) -- (5, 2.75);
	\draw[knot_diagram] (4, 0) -- (6, 2);
	\draw[knot_diagram] (2, -1.75) -- (2, 0) -- (4, 2) -- (4, 1.75);
	\filldraw (3, 1) circle (0.05) node[below right] {$x$};
\end{tikzpicture}
\hfill \ \ 
\end{center}
\caption{\label{Figure:SpecialNeigh}Regular point (on the left), triple point (in the center), true vertex (on the right)}
\end{figure}
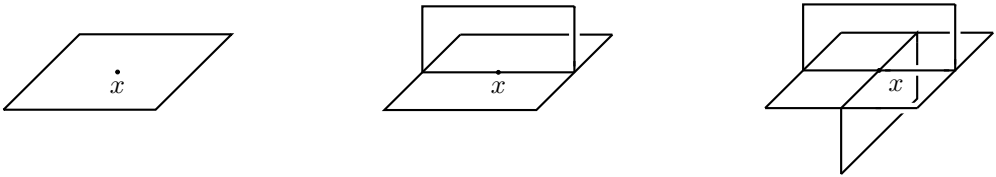

The special polyhedron is called a special spine of the closed 3-manifold $M$, if its complement $M\setminus P$ is homomorphic to an open 3-ball. We will restrict ourselves to closed 3-manifolds, so we don't need a general definition of the spine.

Now we are ready to define the invariant $tv_{\varepsilon}$. Let $P$ be a special spine of the closed 3-manifold $M$. Let $\mathcal{V}(P)$ be the set of all true vertices, $\mathcal{E}(P)$ the set of all triple lines and $\mathcal{C}(P)$ the set of all 2-components of the spine $P$. The colouring of $P$ is a map $\zeta\colon \mathcal{C}(P)\to \{\U, \A\}$. Define the weight of the coloured spine $\{P\}_{\zeta}$ as follows. For each true vertex $v\in \mathcal{V}(P)$, associate the 6j-symbol $\left|\begin{array}{lll}
X_1 & Y_1 & Z_1 \\
X_2 & Y_2 & Z_2
\end{array}\right|$, where $X_1, X_2, X_3$ are colours of three 2-components incident to a triple line in the neighbourhood of $v$, and $X_2, Y_2, Z_2$ are colours of opposite 2-components (figure \ref{Figure:ColoredVertex} on the left). Use the 6j-symbols from the theorem \ref{Theorem:6JSymbols}. Then the weight $\{P\}_{\zeta}$ is a contraction of all tensors associated with all true vertices. The contruction is done by pairings from the theorem \ref{Theorem:Pairings}.

From an equivalent point of view we can define $\{P\}_{\zeta}$ as follows. Assign to each triple line $e\in \mathcal{E}(P)$, incident to 2-components with colours $X, Y, Z$, the value opposite to $(1, 1)^{XYZ}$ (figure \ref{Figure:ColoredVertex} on the right). Then the weight $\{P\}_{\zeta}$ is equal to the product of these assigned values and the values of the 6j-symbols associated to the true vertices, calculated with the argument $1\otimes 1\otimes 1\otimes 1$.

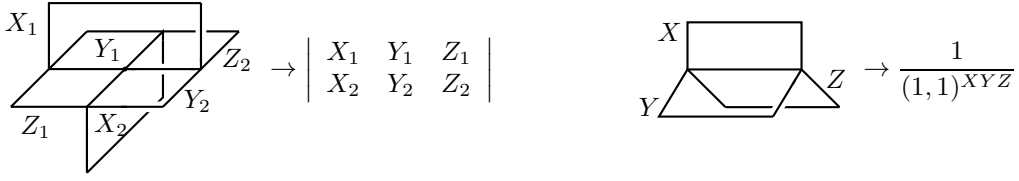
\begin{figure}[h]
\begin{center}
\ \hfill
\begin{tikzpicture}[scale=0.5, baseline={([yshift=0.9ex]current bounding box.center)}]
\begin{knot}[
	% draft mode=crossings,
	clip width=5,
	ignore endpoint intersections=false,
	flip crossing/.list={3}
]
	\strand (5, 2.75) -- (5, 1) -- (1, 1);
	\strand (2, 2) -- (6, 2);
	\strand (4, 1.75) -- (4, 0.25) -- (2, -1.75);
	\strand (0, 0) node[below right] {$Z_1$} -- (4, 0);
\end{knot}
	\draw[knot_diagram] (0, 0) -- (2, 2) node[below right=-1.0pt] {$Y_1$};
	\draw[knot_diagram] (1, 1) -- (1, 2.75) node[below left] {$X_1$} -- (5, 2.75);
	\draw[knot_diagram] (4, 0) -- (6, 2);
	\draw[knot_diagram] (2, -1.75) -- (2, 0) node[below right=-1.0pt] {$X_2$} -- (4, 2) -- (4, 1.75);
	\draw (4.3, 0.15) node[right] {$Y_2$};
	\draw (5.3, 1.25) node[right] {$Z_2$};
	\filldraw (3, 1) circle (0.05);
\end{tikzpicture} $\to \left|\begin{array}{lll}
X_1 & Y_1 & Z_1 \\
X_2 & Y_2 & Z_2
\end{array}
\right|$
\hfill
\begin{tikzpicture}[scale=0.5, baseline={([yshift=0.0ex]current bounding box.center)}]
\begin{knot}[
	% draft mode=crossings,
	clip width=5,
	ignore endpoint intersections=false,
	% flip crossing/.list={3}
]
	\strand (3, 0) -- (2.25, -1.25);
	\strand (1, -1) -- (4, -1);
\end{knot}
	\draw[knot_diagram] (2.25, -1.25) -- (-0.75, -1.25) -- (0, 0) -- (3, 0);
	\draw[knot_diagram] (1, -1) -- (0, 0) -- (0, 1.25) -- (3, 1.25) -- (3, 0) -- (4, -1);
	
	\draw (-1, -1) node {$Y$};
	\draw (0.1, 1) node[left] {$X$};
	\draw (3.9, -0.25) node {$Z$};
\end{tikzpicture} $\to \dfrac{1}{(1, 1)^{XYZ}}$
\hfill \ \ 
\end{center}
\caption{\label{Figure:ColoredVertex}6j-symbol corresponds to a true vertex (on the left), opposite value to pairing corresponds to a triple line (on the right)}
\end{figure}

Let $\zeta$ be a colouring. Denote by $|\zeta|_{\A}$ the number of 2-components coloured by $\A$. Then $$tv_{\varepsilon}(P) = \sum\limits_{\zeta}\varepsilon^{|\zeta|_{\A}}\cdot \{P\}_{\zeta},$$ where the sum is taken over all colourings of the spine $P$.

There is only one difference in our definition of the Turaev - Viro type invariant with respect to the original ones from \cite[Section VII.1.3]{T}. We do not use the coefficient $\frac{1}{D^{2V}}$, where $V$ is a number of vertices in the triangulation of the manifold. In our case, the complement of the spine is always an open 3-ball. So we can ignore this coefficient.

\begin{theorem}
\label{Theorem:TVInvariant}
If $P_1$ and $P_2$ are special spines of closed 3-manifold $M$ then $tv_{\varepsilon}(P_1) = tv_{\varepsilon}(P_2)$.
\end{theorem}
\begin{proof}
The theorem statement follows from \cite[Theorem VII.1.4]{T}.
\end{proof}

So we can correctly define $tv_{\varepsilon}(M)$ as equal to $tv_{\varepsilon}(P)$ for any special spine $P$ of the closed 3-manifold $M$.

\begin{theorem}
\label{Theorem:TRAndTV}
For any closed 3-manifold $M$: $$|tr_{\varepsilon}(M)|^2 = \frac{tv_{\varepsilon}(M)}{\varepsilon + 2}.$$
\end{theorem}
\begin{proof}
The theorem statement is a corollary of \cite[Theorem VII.4.1.1]{T}.
\end{proof}

\subsection{$\varepsilon$-invariant}

There are several equivalent ways to define the $\varepsilon$-invariant for 3-manifolds (see \cite[Chapter 8]{M}). We will use the approach that is very close to the definition of the $tv_{\varepsilon}$ invariant.

Let $P$ be a special spine of the closed 3-manifold $M$, and let $\eta\colon \mathcal{C}(P)\to \{0, 1\}$ be a colouring of $P$ by two colours: 0 and 1. Define the colour weights as follows: $\omega_0 = 1$ and $\omega_{1} = \varepsilon$, where, as before, $\varepsilon^2 = \varepsilon + 1$.

Let $v\in\mathcal{V}(P)$ be a true vertex of the spine $P$. The neighbourhood of this vertex contains six 2-components. Let $i, j, k\in \{0, 1\}$ be the colours of the components incident to a triple line in the neighbourhood of the vertex $v$, and let $l, m, n$ be the colours of the opposite 2-components. Define the weight $\omega_v$ of the vertex $v$ as follows: $$\omega_v = \left|\begin{array}{lll}
i & j & k \\
l & m & n
\end{array}\right|',$$ where 
\begin{center}
$\left|\begin{array}{lll}
0 & 0 & 0 \\
0 & 0 & 0
\end{array}\right|' = 1$, $\left|\begin{array}{lll}
0 & 0 & 0 \\
1 & 1 & 1
\end{array}\right|' = \frac{1}{\sqrt{\varepsilon}}$, $\left|\begin{array}{lll}
0 & 1 & 1 \\
0 & 1 & 1
\end{array}\right|' = \frac{1}{\varepsilon}$, $\left|\begin{array}{lll}
0 & 1 & 1 \\
1 & 1 & 1
\end{array}\right|' = \frac{1}{\varepsilon}$, $\left|\begin{array}{lll}
1 & 1 & 1 \\
1 & 1 & 1
\end{array}\right|' = -\frac{1}{\varepsilon^2}$,
\end{center}
and $\omega_v = 0$ in all other cases.

Here we use 6j-symbols with prime to distinguish them from the 6j-symbols used before. The value of the $\varepsilon$-invariant (denoted by $t$) for the manifold $M$ is defined as follows: $$t(M) = \sum\limits_{\eta}\left(\prod_{c\in \mathcal{C}(p)}\omega_{\eta(c)} \cdot \prod_{v\in \mathcal{V}(P)}\omega_v\right),$$ where the sum is taken over all colourings of $P$.

\begin{theorem}
\label{Theorem:TVEqualT}
For any closed 3-manifold $M$: $$tv_{\varepsilon}(M) = t(M).$$
\end{theorem}
\begin{proof}
The main step of the proof is to show that the value $tv_{\varepsilon}(M)$ does not depend on the parameters $x, y$ and $z$ used in the construction.

Let $P$ be a special spine of the manifold $M$, and let $\zeta\colon\mathcal{C}(P)\to \{\U, \A\}$ be a non-trivial colouring (i.e. $\{P\}_{\zeta} \neq 0$). Divide all true vertices of the coloured spine $P$ into five classes with respect to the number of 2-components coloured by $\A$ in the neighbourhood of the vertex. Let $k_0$ be the number of vertices where all incident 2-components are coloured by $\U$, $k_1$ the number of vertices with three incident 2-components coloured by $\A$, $k_2$ the number of vertices with four incident 2-components coloured by $\A$, $k_3$ the number of vertices with five incident 2-components coloured by $\A$, and finally $k_4$ the number of vertices with all incident 2-components coloured by $\A$. Then $$\{P\}_{\zeta} = \left(\frac{y^3z^3}{\sqrt{\varepsilon}}\right)^{k_1}\cdot \left(\frac{y^4 z^4}{\varepsilon}\right)^{k_2}\cdot \left(\frac{xy^5 z^5}{\varepsilon}\right)^{k_3}\cdot \left(- \frac{x^2 y^6 z^6}{\varepsilon^2}\right)^{k_4}\cdot \left(\frac{1}{xy^3 z^3}\right)^{k_3 + 2k_4}\cdot \left(\frac{1}{z^2 y^2}\right)^{\frac{3}{2}k_1 + 2 k_2 + k_3}.$$ The last two multipliers come from triple lines of $P$. The number of triple lines with only one incident 2-component coloured by $\U$ is equal to $\frac{3}{2}k_1 + 2k_2 + k_3$, and the number of triple lines with all three incident 2-components coloured by $\A$ is equal to $k_3 + 2k_4$.

In the expression $\{P\}_{\zeta}$, the total power of the parameter $x$ is $k_3 + 2k_4 - k_3 - 2k_4 = 0$, and the total powers of the parameters $y$ and $z$ are equal and equal to $3k_1 + 4k_2 + 5k_3 + 6k_4 - 3k_3 - 6k_4 - 3k_1 - 4k_2 - 2k_3 = 0$.

The value $\{P\}_{\zeta}$ does not depend on the parameters $x, y, z$. So we can choose $x = y = z = 1$, and then the formula for $tv_{\varepsilon}(M)$ is exactly the same as the formula for $t(M)$ (after changing $\U$ to $0$ and $\A$ to $1$).
\end{proof}

\begin{example}
\label{Example:TRStronger}
The $tr_{\varepsilon}$ invariant is stronger than $tv_{\varepsilon}$. It's known that there are only four different values of the $\varepsilon$-invariant for lens spaces (\cite{MOS}, \cite[Theorem 8.1.28]{M}).

It's easy to calculate $tv_{\varepsilon}(L_{1, 1}) = tv_{\varepsilon}(L_{4, 1}) = 1$, but $tr_{\varepsilon}(L_{1, 1}) = \frac{1}{\sqrt{\varepsilon + 2}}\in \mathbb{R}$ and $tr_{\varepsilon}(L_{4, 1}) = \frac{(\varepsilon + 1)\cdot (1 + \beta_{\varepsilon})^2}{(\varepsilon + 2)^{\frac{3}{2}}}\in\mathbb{C}\setminus\mathbb{R}$. So, $tr_{\varepsilon}(L_{1, 1})\neq tr_{\varepsilon}(L_{4, 1})$.
\end{example}

\end{document}